\documentclass[a4paper,oneside,english]{amsart}
\usepackage[T1]{fontenc}
\usepackage[latin9]{inputenc}
\usepackage{mathrsfs}
\usepackage{amsbsy}
\usepackage{amstext}
\usepackage{amsthm}
\usepackage{amssymb}
\usepackage{stmaryrd}
\usepackage[all]{xy}

\makeatletter

\pdfpageheight\paperheight
\pdfpagewidth\paperwidth

\newcommand{\noun}[1]{\textsc{#1}}

\numberwithin{equation}{section}
\numberwithin{figure}{section}
\theoremstyle{plain}
\newtheorem{thm}{\protect\theoremname}
  \theoremstyle{definition}
  \newtheorem{defn}[thm]{\protect\definitionname}
  \theoremstyle{remark}
  \newtheorem{rem}[thm]{\protect\remarkname}
  \theoremstyle{definition}
  \newtheorem{example}[thm]{\protect\examplename}
  \theoremstyle{plain}
  \newtheorem{prop}[thm]{\protect\propositionname}
  \theoremstyle{plain}
  \newtheorem{lem}[thm]{\protect\lemmaname}
  \theoremstyle{plain}
  \newtheorem{cor}[thm]{\protect\corollaryname}

\usepackage[dvipsnames,svgnames,x11names,hyperref]{xcolor}
\usepackage[colorlinks=true,linkcolor=Dark Red, citecolor=Dark Green,linktoc=all]{hyperref}
\usepackage{lmodern}
\renewcommand*{\epsilon}{\varepsilon}
\usepackage{amssymb}
\usepackage{amsfonts}
\usepackage{egothic}
\usepackage[T1]{fontenc}
\usepackage{url}
\usepackage{amsthm}
\usepackage{aliascnt}
\usepackage{lipsum}
\usepackage{mathrsfs}
\usepackage{esint}
\usepackage{url}
\usepackage{amsmath}
\usepackage{dsfont}
\usepackage{bbm}
\usepackage{pdflscape}
\usepackage{bbm}
\usepackage{bussproofs}

\newlength{\lhs} 
\newlength{\rhs}

\newcommand{\myar}[2]{\ar^-{#1}[#2]}
\newcommand{\myard}[2]{\ar_-{#1}[#2]}

\makeatother

\usepackage{babel}
  \providecommand{\corollaryname}{Corollary}
  \providecommand{\definitionname}{Definition}
  \providecommand{\examplename}{Example}
  \providecommand{\lemmaname}{Lemma}
  \providecommand{\propositionname}{Proposition}
  \providecommand{\remarkname}{Remark}
\providecommand{\theoremname}{Theorem}

\begin{document}

\title{Distributive Laws via Admissibility}

\author{Charles Walker}

\keywords{KZ-doctrines, lax-idempotent pseudomonads, pseudo-distributive laws}

\subjclass[2000]{18A35, 18C15, 18D05}

\address{Department of Mathematics, Macquarie University, NSW 2109, Australia}

\email{charles.walker1@mq.edu.au}

\thanks{The author acknowledges the support of an Australian Government Research
Training Program Scholarship.}

\date{\today}
\begin{abstract}
This paper concerns the problem of lifting a KZ doctrine $P$ to the
2-category of pseudo $T$-algebras for some pseudomonad $T$. Here
we show that this problem is equivalent to giving a pseudo-distributive
law (meaning that the lifted pseudomonad is automatically KZ), and
that such distributive laws may be simply described algebraically
and are essentially unique (as known to be the case in the (co)KZ
over KZ setting). 

Moreover, we give a simple description of these distributive laws
using Bunge and Funk's notion of admissible morphisms for a KZ doctrine
(the principal goal of this paper). We then go on to show that the
2-category of KZ doctrines on a 2-category is biequivalent to a poset.

We will also discuss here the problem of lifting a locally fully faithful
KZ doctrine, which we noted earlier enjoys most of the axioms of a
Yoneda structure, and show that a bijection between oplax and lax
structures is exhibited on the lifted ``Yoneda structure'' similar
to Kelly's doctrinal adjunction. We also briefly discuss how this
bijection may be viewed as a coherence result for oplax functors out
of the bicategories of spans and polynomials, but leave the details
for a future paper.

\tableofcontents{}
\end{abstract}

\maketitle

\section{Introduction}

It is well known that to give a lifting of a monad to the algebras
of another monad is to give a distributive law \cite{beckdist}. More
generally, to give a lifting of a pseudomonad to the pseudoalgebras
of another pseudomonad is to give a pseudo-distributive law \cite{marm1999,cheng2003}.
However, in this paper we are interested in the problem of lifting
a Kock-Z\"{o}berlein pseudomonad $P$ (also known as a lax idempotent
pseudomonad), as introduced by Kock \cite{kock1972} and Z\"{o}berlein
\cite{zober1976}, to the pseudoalgebras of some pseudomonad $T$.
These KZ pseudomonads are a particular type of pseudomonad for which
algebra structures are adjoint to units; an important example being
the free cocompletion under a class of colimits $\Phi$.

But what does it mean to give a lifting of a KZ doctrine to the setting
of pseudoalgebras such that the lifted pseudomonad is also KZ? One
objective of this paper is to show that this problem is equivalent
to giving a pseudo-distributive law (meaning a lifting of this pseudomonad
automatically inherits the KZ structure), and consequently that such
pseudo-distributive laws have a couple of simple descriptions. One
simple description being purely algebraic (a generalization and simplification
of a description given in \cite[Section 11]{marm1999}), and another
being a novel description purely in terms of left Kan extensions and
Bunge and Funk's admissible maps of a KZ doctrine \cite{bungefunk}.
In fact, Bunge and Funk's admissible maps are a central tool in the
proof of these results. We also see that these distributive laws are
essentially unique, a generalization capturing \cite[Theorem 7.4]{marm2012}
and strengthening parts of \cite[Prop. 4.1]{marm2002}.

These two descriptions of a pseudo-distributive law correspond to
two different descriptions of a KZ pseudomonad. The first, which from
now on we call a KZ pseudomonad, is a well known algebraic description
similar to Kock's \cite{kock1972}; the second, which we call a KZ
doctrine, is to be the description in terms of left Kan extensions
due to Marmolejo and Wood \cite[Definition 3.1]{marm2012}.

Bunge and Funk showed that admissibility in the setting of a KZ pseudomonad
also has both an algebraic definition and a definition in terms of
left Kan extensions. Indeed, Bunge and Funk defined a morphism $f$
to be admissible in the context of a KZ doctrine $P$ when $Pf$ has
a right adjoint \cite[Definition 1.1]{bungefunk}, and showed that
this notion of admissibility also has a description in terms of left
Kan extensions \cite[Prop. 1.5]{bungefunk}. We refer to this as $P$-admissibility.

The central idea here is that instead of thinking about the problem
of lifting a KZ doctrine algebraically, we think about the problem
in terms of algebraic left Kan extensions. Moreover, this notion of
admissibility is crucial here as it allows us to show that certain
left extensions exist and are preserved.

A well known and motivating example the reader may keep in mind is
the KZ doctrine for the free small cocompletion on locally small categories,
with its lifting to the setting of monoidal categories described by
Im and Kelly \cite{uniconvolution} via the Day convolution \cite{dayconvolution}.

In Section \ref{background} we give the necessary background for
this paper, and recall the basic definitions of pseudomonads, pseudo
algebras and morphisms between pseudo algebras. In particular, we
recall the notion of a KZ pseudomonad and KZ doctrine and some results
concerning them. In addition, we recall some results concerning algebraic
left extensions. These notions will be used regularly throughout the
paper.

In Section \ref{liftingkzdoctrines}, which is the bulk of this paper,
we use Bunge and Funk's notion of admissibility to generalize some
results of Marmolejo and Wood concerning pseudo-distributive laws
of (co)KZ doctrines over KZ doctrines, such as the simple form of
such distributive laws \cite[Section 11]{marm1999} or essential uniqueness
of them \cite[Theorem 7.4]{marm2012}. Our first improvement here
is to show that an axiom concerning the (co)KZ doctrine may be dropped,
allowing us to generalize these results to pseudo-distributive laws
of \emph{any pseudomonad} over a KZ doctrine. For example, this level
of generality allows us to capture the case studied by Im and Kelly
\cite{uniconvolution}; showing that the lifting of the small cocompletion
from categories to monoidal categories is essentially unique.

In addition, we use this simplification to give a simple algebraic
description of a pseudo-distributive law of a pseudomonad over a KZ
pseudomonad, consisting only of a pseudonatural transformation and
three invertible modifications subject to three coherence axioms,
and prove this definition is equivalent to the usual notion of pseudo-distributive
law. However, the main new result of this section is a simple description
of pseudo-distributive laws over a KZ doctrine purely in terms of
left Kan extensions and admissibility.

Furthermore, through these calculations we find that in the presence
of a such a distributive law, the lifting of a KZ doctrine $P$ to
pseudo-$T$-algebras (for a pseudomonad $T$) is automatically a KZ
doctrine. The proof of these results is highly technical, relying
on $T$ preserving $P$-admissible maps; however, the main result
of this section is simply stated in Theorem \ref{liftkzequiv}. 

In Section \ref{consequencesandexamples} we study some properties
of the lifted KZ doctrine $\widetilde{P}$, such as classifying the
$\widetilde{P}$-cocomplete $T$-algebras as those for which the underlying
object is $P$-cocomplete and the algebra map separately cocontinuous,
thus justifying the usual definition of algebraic cocompleteness.
We also compare our results to that of Im-Kelly \cite{uniconvolution},
but seen from the KZ doctrine viewpoint.

After checking that the 2-category of KZ doctrines on a 2-category
is biequivalent to a poset, we go on to give some examples in which
we apply our results. Our first example concerns the case of the small
cocompletion and monoidal categories, and our second example concerns
multi-adjoints as studied by Diers \cite{diers}.

In Section \ref{liftlffkz} we consider the problem of lifting a locally
fully faithful KZ doctrine. These locally fully faithful KZ doctrines
are of interest as they almost give rise to Yoneda structures \cite{yonedakz}.
In particular, it is the goal of this section to describe a bijection
between oplax and lax structures on the lifted ``Yoneda structure''
when we have such a distributive law; that is a bijection between
cells $\alpha$ exhibiting $L$ as an oplax $T$-morphism
\[
\xymatrix@=1em{\mathcal{B}\myar{R_{L}}{rr} &  & P\mathcal{A}\ar@{}[ld]|-{\stackrel{\varphi_{L}}{\Longleftarrow}} &  &  &  & \left(\mathcal{B},T\mathcal{B}\overset{y}{\rightarrow}\mathcal{B}\right)\myar{\left(R_{L},\beta\right)}{rr} &  & \left(P\mathcal{A},TP\mathcal{A}\overset{z_{x}}{\rightarrow}P\mathcal{A}\right)\ar@{}[ld]|-{\;\;\;\;\;\stackrel{\varphi_{L}}{\Longleftarrow}}\\
 & \; &  &  &  &  &  & \;\\
 &  & \mathcal{A}\ar[uu]_{y_{\mathcal{A}}}\ar[uull]^{L} &  &  &  &  &  & \left(\mathcal{A},T\mathcal{A}\overset{x}{\rightarrow}\mathcal{A}\right)\ar[uu]_{\left(y_{\mathcal{A}},\xi_{x}\right)}\ar[uull]^{\left(L,\alpha\right)}
}
\]
and cells $\beta$ exhibiting $R_{L}$ as a lax $T$-morphism for
diagrams as on the right above, underlain by a ``Yoneda structure''
diagram such as that on the left above. As an instance of this result
we recover Kelly's bijection between oplax structures on left adjoints
and lax structures on right adjoints \cite{doctrinal}. An interesting
application of this bijection is as a coherence result for the bicategories
of spans and polynomials (and in particular the oplax functors out
of these bicategories). We briefly discuss the applications here,
but leave this to be explored in more detail in a forthcoming paper.

\section{Background\label{background}}

It is the purpose of this section to give the background knowledge
necessary for this paper. We start off by recalling the basic definitions
of pseudomonads, pseudo algebras, and morphisms between pseudo algebras,
as these notions will be used regularly throughout the paper. We then
recall the notion of a left extension in a 2-category, and consider
when these left extensions lift to the setting of pseudo-algebras
and morphisms between them (in a sense which will be applicable in
later sections). Finally, we go on to recall the notion of a KZ pseudomonad,
a special type of pseudomonad for which the algebra structure maps
are adjoint to units, and give their basic properties and some examples.

\subsection{Pseudomonads and their Algebras}

In order to define pseudomonads, we first need the notions of pseudonatural
transformations and modifications. The notion of pseudonatural transformation
is the (weak) 2-categorical version of natural transformation. There
are weaker notions also of lax and oplax natural transformations,
however those will not be used here. Modifications, defined below,
take the place of morphisms between pseudonatural transformations.
\begin{defn}
A \emph{pseudonatural transformation} between pseudofunctors $t\colon F\to G\colon\mathscr{A}\to\mathscr{B}$
where $\mathscr{A}$ and $\mathscr{B}$ are bicategories provides
for each 1-cell $f\colon\mathcal{A}\to\mathcal{B}$ in $\mathscr{A}$,
1-cells $t_{\mathcal{A}}$ and $t_{\mathcal{B}}$ and an invertible
2-cell $t_{f}$ in $\mathscr{B}$ as below
\[
\xymatrix@=1em{F\mathcal{A}\ar[rr]^{Ff}\ar[dd]_{t_{\mathcal{A}}} &  & F\mathcal{B}\ar[dd]^{t_{\mathcal{B}}}\\
 & \ar@{}[]|-{\overset{t_{f}}{\implies}}\\
G\mathcal{A}\ar[rr]_{Gf} &  & G\mathcal{B}
}
\]
satisfying coherence conditions outlined in \cite[Definition 2.2]{kelly1974}.
Given two pseudonatural transformations $t,s\colon F\to G\colon\mathscr{A}\to\mathscr{B}$
as above, a \emph{modification} $\alpha\colon s\to t$ consists of,
for every object $\mathcal{A}\in\mathscr{A}$, a 2-cell $\alpha_{\mathcal{A}}\colon t_{\mathcal{A}}\to s_{\mathcal{A}}$
such that for each 1-cell $f\colon\mathcal{A}\to\mathcal{B}$ in $\mathscr{A}$
we have the equality $\alpha_{\mathcal{B}}\cdot Ff\circ t_{f}=s_{f}\circ Gf\cdot\alpha_{\mathcal{A}}$.
\end{defn}

The following defines the (weak) 2-categorical version of monad to
be used throughout this paper. For brevity, we will suppress pseudofunctoriality
constraints in this definition and those following.
\begin{defn}
A \emph{pseudomonad} on a 2-category $\mathscr{C}$ consists of a
pseudofunctor equipped with pseudonatural transformations as below
\[
T\colon\mathscr{C}\to\mathscr{C},\qquad u\colon1_{\mathscr{C}}\to T,\qquad m\colon T^{2}\to T
\]
along with three invertible modifications
\[
\xymatrix@=1em{T\ar[rr]^{uT}\ar[rdrd]_{\textnormal{id}} &  & T^{2}\ar[dd]|-{m} &  & T\ar[ll]_{Tu}\ar[ldld]^{\textnormal{id}} &  & T^{3}\ar[rr]^{Tm}\ar[dd]_{mT} &  & T^{2}\ar[dd]^{m}\\
 & \;\ar@{}[ru]|-{\overset{\alpha}{\Longleftarrow}} &  & \ar@{}[lu]|-{\overset{\beta}{\Longleftarrow}}\\
 &  & T &  &  &  & T^{2}\ar[rr]_{m} &  & T\ar@{}[lulu]|-{\overset{\gamma}{\Longleftarrow}}
}
\]
subject to the two coherence axioms
\[
\xymatrix@=1em{T^{4}\ar[rr]^{T^{2}m}\ar[dd]_{mT^{2}}\ar[rdrd]|-{TmT} &  & T^{3}\ar[rrdd]^{Tm} &  &  &  & T^{4}\ar[rr]^{T^{2}m}\ar[dd]_{mT^{2}} &  & T^{3}\ar[rrdd]^{Tm}\ar[dd]^{mT}\\
 &  & \ar@{}[]|-{\overset{T\gamma}{\Longleftarrow}} &  &  &  &  & \ar@{}[]|-{\overset{m_{m}^{-1}}{\Longleftarrow}}\\
T^{3}\ar[rrdd]_{mT} & \ar@{}[]|-{\overset{\gamma T}{\Longleftarrow}} & T^{3}\ar[rr]^{Tm}\ar[dd]_{mT} &  & T^{2}\ar[dd]^{m} & = & T^{3}\ar[rrdd]_{mT}\ar[rr]_{Tm} &  & T^{2}\ar[rrdd]^{m} & \ar@{}[]|-{\overset{\gamma}{\Longleftarrow}} & T^{2}\ar[dd]^{m}\\
 &  &  & \ar@{}[]|-{\overset{\gamma}{\Longleftarrow}} &  &  &  &  & \ar@{}[]|-{\overset{\gamma}{\Longleftarrow}}\\
 &  & T^{2}\ar[rr]_{m} &  & T &  &  &  & T^{2}\ar[rr]_{m} &  & T
}
\]
and
\[
\xymatrix@=1em{ &  &  &  & T^{2}\ar[rrd]^{m} &  &  &  &  &  & T^{3}\ar[rrd]^{Tm}\\
T^{2}\ar[rr]^{TuT} &  & T^{3}\ar[rru]^{Tm}\ar[rrd]_{mT} &  & \ar@{}[]|-{\Downarrow\gamma} &  & T & = & T^{2}\ar[rur]^{TuT}\ar[rrd]_{TuT}\ar[rrrr]|-{\textnormal{id}} &  & \ar@{}[u]|-{\Downarrow T\alpha}\ar@{}[d]|-{\Downarrow\beta T} &  & T^{2}\ar[rr]^{m} &  & T\\
 &  &  &  & T^{2}\ar[rru]_{m} &  &  &  &  &  & T^{3}\ar[urr]_{mT}
}
\]
\end{defn}

\begin{rem}
One should note here that there are three useful consequences of these
pseudomonad axioms \cite[Proposition 8.1]{marm1997} originally due
to Kelly \cite{kellymaclanecoherence}. Of these, we will only need
the consequence that
\begin{equation}
\xymatrix@=1em{ &  &  &  & T^{2}\ar[rrd]^{m} &  &  &  &  &  & T\ar[rrd]^{uT}\\
1_{\mathscr{C}}\ar[rr]^{u} &  & T\ar[rru]^{uT}\ar[rrd]_{Tu}\ar[rrrr]|-{\textnormal{id}} &  & \ar@{}[u]|-{\Downarrow\alpha}\ar@{}[d]|-{\Downarrow\beta} &  & T & = & 1_{\mathscr{C}}\ar[rru]^{u}\ar[rrd]_{u} &  & \ar@{}[]|-{\Downarrow u_{u}^{-1}} &  & T^{2}\ar[rr]^{m} &  & T\\
 &  &  &  & T^{2}\ar[rru]_{m} &  &  &  &  &  & T\ar[rru]_{Tu}
}
\label{thirdaxiom}
\end{equation}
\end{rem}

Given a pseudomonad $\left(T,u,m\right)$ on a 2-category $\mathscr{C}$
one may consider its strict $T$-algebras and strict $T$-morphisms,
or the weaker counterparts where conditions only hold up coherent
2-cells. These weaker notions are what will be used throughout this
paper, though usually with the coherent 2-cells in question being
invertible. For convenience, we will leave the modifications $\alpha,\beta$
and $\gamma$ in the above definition as unnamed isomorphisms throughout
the rest of the paper.
\begin{defn}
Given a pseudomonad $\left(T,u,m\right)$ on a 2-category $\mathscr{C}$,
a \emph{lax $T$-algebra} consists of an object $\mathcal{A}\in\mathscr{C}$,
a 1-cell $x\colon T\mathcal{A}\to\mathcal{A}$ and 2-cells
\[
\xymatrix{T^{2}\mathcal{A}\ar[d]_{m_{\mathcal{A}}}\ar[r]^{Tx}\ar@{}[rd]|-{\Downarrow\mu} & T\mathcal{A}\ar[d]^{x} & \mathcal{A}\ar[rr]^{\textnormal{id}}\ar[rd]_{u_{\mathcal{A}}} & \;\ar@{}[d]|-{\Downarrow\nu} & \mathcal{A}\\
T\mathcal{A}\ar[r]_{x} & \mathcal{A} &  & T\mathcal{A}\ar[ur]_{x}
}
\]
such that both 
\[
\xymatrix@=1em{ &  & \ar@{}[d]|-{\Downarrow\nu} &  &  &  &  &  &  &  & T\mathcal{A}\ar[rrdd]^{x}\ar@{}[dddd]|-{\Downarrow\mu}\\
\mathcal{A}\ar[rr]^{u_{\mathcal{A}}}\ar@/^{2pc}/[rrrr]^{\textnormal{id}}\ar@{}[drdr]|-{\Downarrow u_{x}^{-1}} &  & T\mathcal{A}\ar[rr]^{x}\ar@{}[ddrr]|-{\Downarrow\mu} &  & \mathcal{A} &  &  &  & \ar@{}[]|-{\Downarrow T\nu}\\
 &  &  &  &  &  & T\mathcal{A}\ar@/^{1pc}/[rrrruu]^{\textnormal{id}}\ar[rr]_{Tu_{\mathcal{A}}}\ar@/_{1pc}/[rrrrdd]_{\textnormal{id}} &  & T^{2}\mathcal{A}\ar[uurr]^{Tx}\ar[ddrr]_{m_{\mathcal{A}}} &  &  &  & \mathcal{A}\\
T\mathcal{A}\ar[rr]_{u_{T\mathcal{A}}}\ar@/_{2pc}/[rrrr]_{\textnormal{id}}\ar[uu]^{x} &  & T^{2}\mathcal{A}\ar[uu]_{Tx}\ar[rr]_{m_{\mathcal{A}}} &  & T\mathcal{A}\ar[uu]_{x} &  &  &  & \ar@{}[]|-{\cong}\\
 &  & \ar@{}[u]|-{\cong} &  &  &  &  &  &  &  & T\mathcal{A}\ar[urur]_{x}
}
\]
paste to the identity 2-cell at $x$, known as the left and right
unit axioms respectively. Moreover, the associativity axiom asks that
we have the equality 
\[
\xymatrix@=1em{ &  & T^{2}\mathcal{A}\ar[rr]^{Tx}\ar[rdrd]^{m_{\mathcal{A}}}\ar@{}[dddd]|-{\Downarrow m_{x}^{-1}} &  & T\mathcal{A}\ar[rrdd]^{x}\ar@{}[dd]|-{\Downarrow\mu} &  &  &  &  &  & T^{2}\mathcal{A}\ar[rr]^{Tx}\ar@{}[dd]|-{\Downarrow T\mu} &  & T\mathcal{A}\ar[rrdd]^{x}\ar@{}[dddd]|-{\Downarrow\mu}\\
\\
T^{3}\mathcal{A}\ar[urur]^{T^{2}x}\ar[rrdd]_{m_{T\mathcal{A}}} &  &  &  & T\mathcal{A}\ar[rr]^{x}\ar@{}[dd]|-{\Downarrow\mu} &  & \mathcal{A} & = & T^{3}\mathcal{A}\ar[rr]^{Tm_{\mathcal{A}}}\ar[uurr]^{T^{2}x}\ar[rdrd]_{m_{T\mathcal{A}}} &  & T^{2}\mathcal{A}\ar[uurr]^{Tx}\ar[ddrr]_{m_{\mathcal{A}}}\ar@{}[dd]|-{\cong} &  &  &  & \mathcal{A}\\
\\
 &  & T^{2}\mathcal{A}\ar[rr]_{m_{\mathcal{A}}}\ar[uurr]_{Tx} &  & T\mathcal{A}\ar[rruu]_{x} &  &  &  &  &  & T^{2}\mathcal{A}\ar[rr]_{m_{\mathcal{A}}} &  & T\mathcal{A}\ar[rruu]_{x}
}
\]
If the above 2-cells $\nu$ and $\mu$ are isomorphisms, we call this
a \emph{pseudo $T$-algebra}. If $\nu$ and $\mu$ are identity 2-cells,
we call this a \emph{strict $T$-algebra}.
\end{defn}

These $T$-algebras may be regarded as the objects of a category,
with morphisms of (pseudo) $T$-algebras defined as follows.
\begin{defn}
Given a pseudomonad $\left(T,u,m\right)$ on a 2-category $\mathscr{C}$,
an \emph{oplax $T$-morphism} of pseudo $T$-algebras
\[
\left(L,\alpha\right)\colon\left(\mathcal{A},T\mathcal{A}\stackrel{x}{\rightarrow}\mathcal{A}\right)\to\left(\mathcal{B},T\mathcal{B}\stackrel{y}{\rightarrow}\mathcal{B}\right)
\]
consists of a 1-cell $L\colon\mathcal{A}\to\mathcal{B}$ and a 2-cell
\[
\xymatrix{T\mathcal{B}\ar[r]^{y}\ar@{}[rd]|-{\Uparrow\alpha} & \mathcal{B}\\
T\mathcal{A}\ar[r]_{x}\ar[u]^{TL} & \mathcal{A}\ar[u]_{L}
}
\]
such that (leaving the pseudo $T$-algebra coherence cells as unnamed
isomorphisms)
\[
\xymatrix{ & \;\\
\mathcal{B}\ar[r]^{u_{\mathcal{B}}}\ar@{}[rd]|-{\Uparrow u_{L}}\ar@/^{2.5pc}/[rr]^{\textnormal{id}} & T\mathcal{B}\ar[r]^{y}\;\ar@{}[rd]|-{\Uparrow\alpha}\ar@{}[u]|-{\cong} & \mathcal{B}\\
\mathcal{A}\ar[u]^{L}\ar[r]_{u_{\mathcal{A}}}\ar@/_{2.5pc}/[rr]_{\textnormal{id}} & T\mathcal{A}\ar[r]_{x}\ar[u]|-{TL}\ar@{}[d]|-{\cong} & \mathcal{A}\ar[u]_{L}\\
 & \;
}
\]
is the identity 2-cell on $L$, and for which
\[
\xymatrix{ & T\mathcal{B}\ar[rd]^{y}\\
T^{2}\mathcal{B}\ar[r]^{Ty}\ar@{}[rd]|-{\Uparrow T\alpha}\ar[ru]^{m_{\mathcal{B}}} & T\mathcal{B}\ar[r]^{y}\ar@{}[rd]|-{\Uparrow\alpha}\ar@{}[u]|-{\cong} & \mathcal{B} & \ar@{}[d]|-{=} & T^{2}\mathcal{B}\ar[r]^{m_{\mathcal{B}}}\ar@{}[dr]|-{\Uparrow m_{L}} & T\mathcal{B}\ar[r]^{y}\;\ar@{}[rd]|-{\Uparrow\alpha} & \mathcal{B}\\
T^{2}\mathcal{A}\ar[u]^{T^{2}L}\ar[r]_{Tx}\ar[rd]_{m_{\mathcal{A}}} & T\mathcal{A}\ar[r]_{x}\ar[u]|-{TL} & \mathcal{A}\ar[u]_{L} & \; & T^{2}\mathcal{A}\ar[u]^{T^{2}L}\ar[r]_{m_{\mathcal{A}}} & T\mathcal{A}\ar[r]_{x}\ar[u]|-{TL} & \mathcal{A}\ar[u]_{L}\\
 & T\mathcal{A}\ar[ur]_{x}\ar@{}[u]|-{\cong}
}
\]

If the 2-cell $\alpha$ goes in the opposite direction, this is the
definition of a \emph{lax $T$-morphism, }and if $\alpha$ is invertible
this is then the definition of a \emph{pseudo $T$-morphism.}
\end{defn}

The usual definition of $T$-transformation between oplax or lax $T$-morphisms
is not general enough for our purposes as we will be considering situations
in which we have both oplax and lax $T$-morphisms, and so we define
$T$-transformations as based on the double category viewpoint \cite{adjointdouble}.
Such transformations are sometimes referred to as generalized $T$-transformations.
\begin{defn}
Suppose we are given a square of morphisms of pseudo $T$-algebras
\[
\xymatrix@=1em{\left(\mathcal{B},y\right)\ar[rr]^{\left(R,\beta\right)} &  & \left(\mathcal{C},z\right)\ar@{}[ldld]|-{\stackrel{\zeta}{\Longleftarrow}}\\
\\
\left(\mathcal{D},w\right)\ar[rr]_{\left(M,\varepsilon\right)}\ar[uu]^{\left(N,\varphi\right)} &  & \left(\mathcal{A},x\right)\ar[uu]_{\left(I,\xi\right)}
}
\]
where the vertical maps are oplax $T$-morphisms and the horizontal
maps are lax $T$-morphisms. A \emph{$T$-transformation }$\zeta$
as in the above square is a 2-cell $\zeta:I\cdot M\to R\cdot N$ for
which we have the equality of the two sides of the cube
\[
\xymatrix{ & T\mathcal{B}\ar[r]^{y}\ar@{}[d]|-{\Uparrow\varphi} & \mathcal{B}\ar[rd]^{R}\ar@{}[dd]|-{\Uparrow\zeta} &  &  &  & T\mathcal{B}\ar[r]^{y}\ar[dr]^{TR}\ar@{}[dd]|-{\Uparrow T\zeta} & \mathcal{B}\ar[rd]^{R}\ar@{}[d]|-{\Uparrow\beta}\\
T\mathcal{D}\ar[r]^{w}\ar[rd]_{TM}\ar[ru]^{TN} & \mathcal{D}\ar[rd]_{M}\ar[ru]^{N}\ar@{}[d]|-{\Uparrow\varepsilon} &  & \mathcal{C} & = & T\mathcal{D}\ar[rd]_{TM}\ar[ru]^{TN} &  & T\mathcal{C}\ar[r]^{z}\ar@{}[d]|-{\Uparrow\xi} & \mathcal{C}\\
 & T\mathcal{A}\ar[r]_{x} & \mathcal{A}\ar[ur]_{I} &  &  &  & T\mathcal{A}\ar[r]_{x}\ar[ur]_{TI} & \mathcal{A}\ar[ur]_{I}
}
\]

We will call the 2-category of pseudo $T$-algebras, pseudo $T$-morphisms,
and $T$-transformations $\text{ps}$-$T$-alg (we may consider squares
where both horizontal maps are identities or both vertical maps are
identities to recover the usual notions of transformation between
lax/oplax/pseudo $T$-morphisms). 
\end{defn}

\begin{rem}
\label{Ttransremark} Note that in this language it makes sense to
talk about the unit and counit of an adjunction where the left adjoint
is oplax and the right adjoint lax. Indeed the oplax-lax bijective
correspondence in Kelly's doctrinal adjunction \cite{doctrinal} is
unique such the counit $\epsilon$ (and unit $\eta$) of the adjunction
is a $T$-transformation\footnote{This is shown in more generality in Proposition \ref{docparadj}.}.
Note also that in this setting of a doctrinal adjunction $L\dashv R$
(with an oplax structure $\alpha$ on $L$ corresponding a lax structure
$\beta$ on $R$ via the mates correspondence) it makes sense to view
the unit and counit as $T$-transformations as we have squares 
\[
\xymatrix@=1em{\left(\mathcal{B},y\right)\ar[rr]^{\left(\textnormal{id},\textnormal{id}\right)} &  & \left(\mathcal{B},y\right)\ar@{}[ldld]|-{\stackrel{\epsilon}{\Longleftarrow}} &  & \left(\mathcal{B},y\right)\ar[rr]^{\left(R,\beta\right)} &  & \left(\mathcal{A},x\right)\ar@{}[ldld]|-{\stackrel{\eta}{\Longleftarrow}}\\
\\
\left(\mathcal{B},y\right)\ar[rr]_{\left(R,\beta\right)}\ar[uu]^{\left(\textnormal{id},\textnormal{id}\right)} &  & \left(\mathcal{A},x\right)\ar[uu]_{\left(L,\alpha\right)} &  & \left(\mathcal{A},x\right)\ar[rr]_{\left(\textnormal{id},\textnormal{id}\right)}\ar[uu]^{\left(L,\alpha\right)} &  & \left(\mathcal{A},x\right)\ar[uu]_{\left(\textnormal{id},\textnormal{id}\right)}
}
\]

As a convention, will will usually omit these identity $T$-morphisms.
The reader may just remember that it makes sense to consider $T$-transformations
from a lax followed by an oplax $T$-morphism, into an oplax followed
by a lax $T$-morphism, and that any such transformation may be uniquely
expressed as a square in the form of the above definition by inserting
the appropriate identity $T$-morphisms; which is what we have done
in the case of the unit and counit above.
\end{rem}

\begin{example}
One may define the category of $\mathbf{Cat}$ (the category of locally
small categories) enriched graphs, denoted $\mathbf{CatGrph}$, with
objects given as families of hom-categories 
\[
\left(\mathscr{C}\left(X,Y\right)\colon X,Y\in\textnormal{ob}\mathscr{C}\right)
\]
 and morphisms consisting of locally defined functors
\[
\left(F_{X,Y}:\mathscr{C}\left(X,Y\right)\to\mathscr{D}\left(FX,FY\right)\colon X,Y\in\mathscr{C}\right)
\]
which have not been endowed with the structure of a bicategory or
a lax/oplax functor respectively \cite{companion}. This gives rise
to, via a suitable 2-monad $T$ on $\mathbf{CatGrph}$, the 2-category
of bicategories, oplax functors and icons \cite{icons}. We may of
course replace oplax here with ``lax'' or ``pseudo''. Note that
inside this 2-category lives the one object bicategories (isomorphic
to monoidal categories), giving the 2-category of monoidal categories,
lax/oplax/strong monoidal functors and monoidal transformations (which
may also be constructed directly via a suitable 2-monad \cite{icons}).
\end{example}

\subsection{Left Extensions and Algebraic Left Extensions\label{doctrinalleftextensions} }

In this section we will consider how pseudomonads interact with left
extensions. In particular, we start off by recalling the notion of
a left extension in a 2-category, and go on to give conditions under
which such a left extension lifts to a suitable notion of left extension
in the setting of pseudo $T$-algebras, $T$-morphisms and $T$-transformations.
The results of this section are mostly due to Koudenburg, shown in
a more general double category setting \cite{roald2015}. 
\begin{defn}
Suppose we are given a 2-cell $\eta\colon I\to R\cdot L$ as in the
left diagram
\[
\xymatrix@=1em{ &  &  &  &  &  &  & \ar@{}[d]|-{\Uparrow\sigma}\\
\mathcal{B}\ar[rr]^{R} &  & \mathcal{C}\ar@{}[ld]|-{\stackrel{\eta}{\Longleftarrow}} &  &  &  & \mathcal{B}\ar[rr]|-{R}\ar@/^{1.5pc}/[rr]^{M} &  & \mathcal{C}\ar@{}[ld]|-{\stackrel{\eta}{\Longleftarrow}}\\
 & \; &  &  &  &  &  & \;\\
 &  & \mathcal{A}\ar[uu]_{I}\ar[uull]^{L} &  &  &  &  &  & \mathcal{A}\ar[uu]_{I}\ar[uull]^{L}
}
\]
in a 2-category $\mathscr{C}$. We say that $R$ is exhibited as a
\emph{left extension} of $I$ along $L$ by the 2-cell $\eta$ when
pasting 2-cells $\sigma:R\to M$ with the 2-cell $\eta:I\to R\cdot L$
as in the right diagram defines a bijection between 2-cells $R\to M$
and 2-cells $I\to M\cdot L$. Moreover, we say such a left extension
$\left(R,\eta\right)$ is \emph{respected} (also called \emph{preserved})
by a 1-cell $E\colon\mathcal{C}\to\mathcal{D}$ when the whiskering
of $\eta$ by $E$, as given by the pasting diagram below
\[
\xymatrix@=1em{\mathcal{B}\ar[rr]^{R} &  & \mathcal{C}\ar@{}[ld]|-{\stackrel{\eta}{\Longleftarrow}}\ar[rr]^{E}\ar@{}[rd]|-{\stackrel{\textnormal{id}}{\Longleftarrow}} &  & \mathcal{D}\\
 & \; &  & \;\\
 &  & \mathcal{A}\ar[uu]_{I}\ar[uull]^{L}\ar[uurr]_{E\cdot I}
}
\]
exhibits $E\cdot R$ as a left extension of $E\cdot I$ along $L$. 
\end{defn}

We now give a suitable description of when a lax $T$-morphism may
be regarded as a left extension in the setting of pseudo $T$-algebras.
\begin{defn}
Suppose we are given an oplax $T$-morphism $\left(L,\alpha\right)$
and lax $T$-morphisms $\left(R,\beta\right)$ and $\left(I,\sigma\right)$
between pseudo $T$-algebras equipped with a $T$-transformation $\eta\colon I\to R\cdot L$
as in the diagram
\[
\xymatrix@=1em{\left(\mathcal{B},T\mathcal{B}\overset{y}{\rightarrow}\mathcal{B}\right)\ar[rr]^{\left(R,\beta\right)} &  & \left(\mathcal{C},T\mathcal{C}\overset{z}{\rightarrow}\mathcal{C}\right)\ar@{}[ld]|-{\stackrel{\eta}{\Longleftarrow}}\\
 & \;\\
 &  & \left(\mathcal{A},T\mathcal{A}\overset{x}{\rightarrow}\mathcal{A}\right)\ar[uu]_{\left(I,\sigma\right)}\ar[uull]^{\left(L,\alpha\right)}
}
\]
We call such a diagram a \emph{$T$-left extension }if for any given
pseudo $T$-algebra $\left(\mathcal{D},w\right)$, lax $T$-morphism
$\left(M,\epsilon\right)$ and oplax $T$-morphism $\left(N,\varphi\right)$
as below
\[
\xymatrix@=1em{ & \left(\mathcal{D},T\mathcal{D}\overset{w}{\rightarrow}\mathcal{D}\right)\ar[rd]^{\left(M,\epsilon\right)}\\
\left(\mathcal{B},T\mathcal{B}\overset{y}{\rightarrow}\mathcal{B}\right)\ar[rr]_{\left(R,\beta\right)}\ar[ur]^{\left(N,\varphi\right)} & \;\ar@{}[u]|-{\Uparrow\overline{\zeta}} & \left(\mathcal{C},T\mathcal{C}\overset{z}{\rightarrow}\mathcal{C}\right)\ar@{}[ld]|-{\stackrel{\eta}{\Longleftarrow}}\\
 & \;\\
 &  & \left(\mathcal{A},T\mathcal{A}\overset{x}{\rightarrow}\mathcal{A}\right)\ar[uu]_{\left(I,\sigma\right)}\ar[uull]^{\left(L,\alpha\right)}
}
\]
pasting $T$-transformations of the form $\overline{\zeta}$ above
with the $T$-transformation $\eta$ defines the bijection of $T$-transformations:
\[
\xymatrix@=1em{\left(\mathcal{D},w\right)\ar[rr]^{\left(M,\epsilon\right)} &  & \left(\mathcal{C},z\right)\ar@{}[ldld]|-{\stackrel{\overline{\zeta}}{\Longleftarrow}} &  &  & \left(\mathcal{D},w\right)\ar[rr]^{\left(M,\epsilon\right)} &  & \left(\mathcal{C},z\right)\ar@{}[ldld]|-{\stackrel{\zeta}{\Longleftarrow}}\\
 &  &  & \ar@{}[r]|-{\sim} & \; & \left(\mathcal{B},y\right)\ar[u]^{\left(N,\varphi\right)}\\
\left(\mathcal{B},y\right)\ar[rr]_{\left(R,\beta\right)}\ar[uu]^{\left(N,\varphi\right)} &  & \left(\mathcal{C},z\right)\ar[uu]_{\left(\textnormal{id},\textnormal{id}\right)} &  &  & \left(\mathcal{A},x\right)\ar[rr]_{\left(I,\sigma\right)}\ar[u]^{\left(L,\alpha\right)} &  & \left(\mathcal{C},z\right)\ar[uu]_{\left(\textnormal{id},\textnormal{id}\right)}
}
\]
\end{defn}

\begin{rem}
Note that if $\overline{\zeta}$ and $\eta$ are both $T$-transformations
then so is the composite $\overline{\zeta}L\cdot\eta$; this is a
simple calculation which we omit.
\end{rem}

In order to lift left extensions to $T$-left extensions as above
we will require the following algebraic cocompleteness property.

\begin{defn}
\label{Tpreserved}Given a pseudomonad $\left(T,u,m\right)$ on a
2-category $\mathscr{C}$, we say a left extension $\left(H,\varphi\right)$
in $\mathscr{C}$ as on the left below is \emph{$T$-preserved} by
a 1-cell $z\colon T\mathcal{C}\to\mathcal{D}$ when 
\[
\xymatrix@=1em{\mathcal{B}\ar[rr]^{H} &  & \mathcal{C}\ar@{}[ld]|-{\stackrel{\varphi}{\Longleftarrow}} &  &  & T\mathcal{B}\ar[rr]^{TH} &  & T\mathcal{C}\ar@{}[ld]|-{\stackrel{T\varphi}{\Longleftarrow}}\ar[r]^{z} & \mathcal{D}\\
 & \; &  &  &  &  &  &  & \;\ar@{}[ul]|-{\stackrel{\textnormal{id}}{\Longleftarrow}\quad\;\;}\\
 &  & \mathcal{X}\ar[uu]_{F}\ar[uull]^{G} &  &  &  &  & T\mathcal{X}\ar[uu]^{TF}\ar[uull]^{TG}\ar[uur]_{z\cdot TF}
}
\]
the pasting diagram on the right exhibits $\left(z\cdot TH,z\cdot T\varphi\right)$
as a left extension. 
\end{defn}

\begin{rem}
Given a pseudo $T$-algebra $\left(\mathcal{C},T\mathcal{C}\overset{z}{\rightarrow}\mathcal{C}\right)$
if we ask that the underlying object $\mathcal{C}$ is cocomplete
in the sense that all left extensions (along a chosen class of maps)
into $\mathcal{C}$ exist, and moreover that the algebra structure
map $z$ $T$-preserves these left extensions, then this is (essentially)
the notion of algebraic cocompleteness as given by Weber \cite[Definition 2.3.1]{markextension}
(except that we are not using pointwise left extensions here). In
the setting monoidal categories, this condition of $z$ (when $z$
is an algebra structure map) $T$-preserving the left extensions is
the analogue of asking the tensor product be separately cocontinuous;
see \cite[Prop. 2.3.2]{markextension}.
\end{rem}

We now recall a result for algebraic left extensions mostly due to
Koudenburg \cite{roald2015} (though we avoid working in a double
categorical setting). We will include some details of the proof as
we will need them later.
\begin{prop}
\label{docleftext} Suppose we are given a diagram
\[
\xymatrix@=1em{\mathcal{B}\ar[rr]^{R} &  & \mathcal{C}\ar@{}[ld]|-{\stackrel{\eta}{\Longleftarrow}}\\
 & \;\\
 &  & \mathcal{A}\ar[uu]_{I}\ar[uull]^{L}
}
\]
which exhibits $R$ as a left extension in a 2-category $\mathscr{C}$
equipped with a pseudomonad $\left(T,u,m\right)$. Suppose further
that 
\[
\left(\mathcal{A},T\mathcal{A}\stackrel{x}{\longrightarrow}\mathcal{A}\right),\quad\left(\mathcal{B},T\mathcal{B}\stackrel{y}{\longrightarrow}\mathcal{B}\right),\quad\left(\mathcal{C},T\mathcal{C}\stackrel{z}{\longrightarrow}\mathcal{C}\right)
\]
are pseudo $T\text{-algebras}$. Suppose even further that the left
extension $\left(R,\eta\right)$ is $T$-preserved by $z$, and the
resulting left extension $\left(z\cdot TR,z\cdot T\eta\right)$ is
itself $T$-preserved by $z$. Then given a lax $T\text{-morphism}$
structure $\sigma$ on $I$ and an oplax $T\text{-morphism}$ structure
$\alpha$ on $L$, there exists a unique lax $T\text{-morphism}$
structure $\beta$ on $R$ for which $\eta$ is a $T$-transformation.
Moreover, this left extension is then lifted to the $T$-left extension
\[
\xymatrix@=1em{\left(\mathcal{B},T\mathcal{B}\overset{y}{\rightarrow}\mathcal{B}\right)\ar[rr]^{\left(R,\beta\right)} &  & \left(\mathcal{C},T\mathcal{C}\overset{z}{\rightarrow}\mathcal{C}\right)\ar@{}[ld]|-{\stackrel{\eta}{\Longleftarrow}}\\
 & \;\\
 &  & \left(\mathcal{A},T\mathcal{A}\overset{x}{\rightarrow}\mathcal{A}\right)\ar[uu]_{\left(I,\sigma\right)}\ar[uull]^{\left(L,\alpha\right)}
}
\]
\end{prop}

\begin{proof}
Given our structure cells $\sigma$ and $\alpha$ as below
\[
\xymatrix{T\mathcal{A}\ar[d]_{TI}\ar[r]^{x}\ar@{}[rd]|-{\Uparrow\sigma} & \mathcal{A}\ar[d]^{I} &  & T\mathcal{A}\ar[d]_{TL}\ar[r]^{x}\ar@{}[rd]|-{\Downarrow\alpha} & \mathcal{A}\ar[d]^{L}\\
T\mathcal{C}\ar[r]_{z} & \mathcal{C} &  & T\mathcal{B}\ar[r]_{y} & \mathcal{B}
}
\]
 our lax constraint cell for $R$ is given as the unique $\beta$
such that $\eta$ is a $T$-transformation, that is the unique 2-cell
such that

\[
\xymatrix{ & T\mathcal{B}\ar[r]^{y} & \mathcal{B}\ar[dd]^{R} &  &  & T\mathcal{B}\ar[r]^{y}\ar[dd]|-{TR} & \mathcal{B}\ar[dd]^{R}\\
T\mathcal{A}\ar[ru]^{TL}\ar[r]^{x}\ar[rd]_{TI} & \mathcal{A}\ar[dr]_{I}\ar[ur]^{L}\ar@{}[r]|-{\Uparrow\eta}\ar@{}[u]|-{\Uparrow\alpha}\ar@{}[d]|-{\Uparrow\sigma} & \; & = & T\mathcal{A}\ar[ru]^{TL}\ar[rd]_{TI}\ar@{}[r]|-{\Uparrow T\eta\quad} & \;\ar@{}[r]|-{\Uparrow\beta} & \;\\
 & T\mathcal{C}\ar[r]_{z} & \mathcal{C} &  &  & T\mathcal{C}\ar[r]_{z} & \mathcal{C}
}
\]
as $z\cdot T\eta$ exhibits $z\cdot TR$ as a left extension. From
here, the proof of the coherence axioms for $\beta$ being a lax $T$-morphism
structure on $R$ is the same as in \cite[Theorem 2.4.4]{markextension}\footnote{The assumptions of \cite[Theorem 2.4.4]{markextension} concerning comma objects are not required for the proof of the coherence axioms.}.
Checking that the lax $T$-morphism $\left(R,\beta\right)$ is then
a $T$-left extension is a straightforward exercise, of which we omit
the details.
\end{proof}

\subsection{KZ Pseudomonads and KZ Doctrines}

A KZ pseudomonad is a special type of pseudomonad for which the algebra
structure maps are adjoint to units; with typical examples including
the cocompletion of a category under some class of colimits $\Phi$.
For this paper, we will use two different (but equivalent) characterizations
of KZ pseudomonads. The first characterization we will use is a well
known algebraic description of a KZ pseudomonad, described via conditions
on a ``KZ structure cell'' (similar to \cite{kock1972}), the second
characterization is in terms of left extensions, and will be referred
to as a KZ doctrine. 
\begin{rem}
Note that there are other (still equivalent) characterizations which
may be referred to as KZ pseudomonads or KZ doctrines. For example
the characterization through adjoint strings \cite{marm1997}, or
the characterization as lax idempotent pseudomonads \cite{lack1997}.
\end{rem}

\begin{defn}
\label{defkzpseudomonad} A \emph{KZ pseudomonad $\left(P,y,\mu\right)$
}on a 2-category $\mathscr{C}$ consists of a pseudomonad $\left(P,y,\mu\right)$
on $\mathscr{C}$ along with a modification $\theta\colon Py\to yP$
for which
\begin{equation}
\xymatrix@=1em{ &  &  &  &  &  &  &  & P\ar[rrd]^{yP}\\
1_{\mathscr{C}}\ar[rr]^{y} &  & P\ar@/^{1pc}/[rr]^{yP}\ar@/_{1pc}/[rr]_{Py}\ar@{}[rr]|-{\Uparrow\theta} &  & P^{2} & = & 1_{\mathscr{C}}\ar[rru]^{y}\ar[rrd]_{y} &  & \ar@{}[]|-{\Uparrow y_{y}} &  & P^{2}\\
 &  &  &  &  &  &  &  & P\ar[rru]_{Py}
}
\label{kzcoh1}
\end{equation}
and
\begin{equation}
\xymatrix@=1em{ &  & \ar@{}[d]|-{\Uparrow\alpha}\\
P\ar@/^{0.7pc}/[rr]^{yP}\ar@/_{0.7pc}/[rr]_{Py}\ar@{}[rr]|-{\Uparrow\theta}\ar@/^{2.5pc}/[rrrr]^{\textnormal{id}_{P}}\ar@/_{2.5pc}/[rrrr]_{\textnormal{id}_{P}} &  & P^{2}\ar[rr]^{\mu} &  & P & = & \textnormal{id}_{\textnormal{id}_{P}}\\
 &  & \ar@{}[u]|-{\Uparrow\beta}
}
\label{kzcoh2}
\end{equation}
\end{defn}

\begin{rem}
It is shown in \cite[Prop. 3.1, Lemma 3.2]{marm1997} that given the
adjoint string characterization we recover the definition given above,
and conversely given the above definition it is not hard to recover
the adjoint string definition, especially since it suffices to give
only one adjunction \cite[Theorem 11.1]{marm1997}.
\end{rem}

The above is an algebraic description of a KZ pseudomonad; however
there is another description in terms of left Kan extensions given
by Marmolejo and Wood \cite{marm2012} which we refer to as a KZ doctrine.
\begin{defn}
\label{defkzdoctrine}\cite[Definition 3.1]{marm2012} A \emph{KZ doctrine
$\left(P,y\right)$ }on a 2-category $\mathscr{C}$ consists of 

(i) An assignation on objects $P\colon\textnormal{ob}\mathscr{C}\to\textnormal{ob}\mathscr{C}$;

(ii) For every object $\mathcal{A}\in\mathscr{C}$, a 1-cell $y_{\mathcal{A}}\colon\mathcal{A}\to P\mathcal{A}$;

(iii) For every pair of objects $\mathcal{A}\text{ and }\mathcal{B}$
and 1-cell $F\colon\mathcal{A}\to P\mathcal{B}$, a left extension
\begin{equation}
\xymatrix@=1em{P\mathcal{A}\ar@{->}[rr]^{\overline{F}}\ar@{}[rd]|-{\stackrel{c_{F}}{\Longleftarrow}} &  & P\mathcal{B}\\
 & \;\\
\mathcal{A}\ar[rruu]_{F}\ar[uu]^{y_{\mathcal{A}}}
}
\label{kzdef2.1}
\end{equation}
of $F$ along $y_{\mathcal{A}}$ exhibited by an isomorphism $c_{F}$
as above. 

Moreover, we require that:

(a) For every object $\mathcal{A}\in\mathscr{C}$, the left extension
of $y_{\mathcal{A}}$ as in \ref{kzdef2.1} is given by
\[
\xymatrix@=1em{P\mathcal{A}\ar@{->}[rr]^{\textnormal{id}_{P\mathcal{A}}} &  & P\mathcal{A}\ar@{}[ld]|-{\stackrel{\textnormal{id}}{\Longleftarrow}}\\
 & \;\\
 &  & \mathcal{A}\ar[uu]_{y_{\mathcal{A}}}\ar[ulul]^{y_{\mathcal{A}}}
}
\]

Note that this means $c_{y_{\mathcal{A}}}$ is equal to the identity
2-cell on $y_{\mathcal{A}}$. 

(b) For any 1-cell $G\colon\mathcal{B}\to P\mathcal{C}$, the corresponding
left extension $\overline{G}\colon P\mathcal{B}\to P\mathcal{C}$
preserves the left extension $\overline{F}$ in \ref{kzdef2.1}.
\end{defn}

\begin{rem}
These two descriptions are equivalent in the sense that each gives
rise to the other \cite{marm2012,marm1997}. In Section \ref{consequencesandexamples}
we will express this relationship as a biequivalence between the 2-category
of KZ pseudomonads and the preorder of KZ doctrines.
\end{rem}

The following definitions in terms of left extensions are equivalent
to the preceding notions of pseudo $P$-algebra and $P$-homomorphism,
in the sense that we have an equivalence between the two resulting
2-categories of pseudo $P$-algebras arising from the two different
definitions \cite[Theorems 5.1,5.2]{marm2012}.
\begin{defn}
[\cite{marm2012}] Given a KZ doctrine $\left(P,y\right)$ on a 2-category
$\mathscr{C}$, we say an object $\mathcal{X}\in\mathscr{C}$ is \emph{$P$-cocomplete}
if for every $G\colon\mathcal{B}\to\mathcal{X}$ 
\[
\xymatrix@=1em{P\mathcal{B}\ar@{->}[rr]^{\overline{G}}\ar@{}[rd]|-{\stackrel{c_{G}}{\Longleftarrow}} &  & \mathcal{X} &  &  & P\mathcal{A}\ar@{->}[rr]^{\overline{F}}\ar@{}[rd]|-{\stackrel{c_{F}}{\Longleftarrow}} &  & P\mathcal{B}\ar[rr]^{\overline{G}} &  & \mathcal{X}\\
 & \; &  &  &  &  & \;\\
\mathcal{B}\ar[rruu]_{G}\ar[uu]^{y_{\mathcal{B}}} &  &  &  &  & \mathcal{A}\ar[rruu]_{F}\ar[uu]^{y_{\mathcal{A}}}
}
\]
there exists a left extension $\overline{G}$ as on the left exhibited
by an isomorphism $c_{G}$, and moreover this left extension respects
the left extensions $\overline{F}$ as in the diagram on the right.
We say a 1-cell $E\colon\mathcal{X}\to\mathcal{Y}$ between $P$-cocomplete
objects $\mathcal{X}$ and $\mathcal{Y}$ is a \emph{$P$-homomorphism
}(also called $P$\emph{-cocontinuous}) when it preserves all left
extensions along $y_{\mathcal{B}}$ into $\mathcal{X}$ for every
object $\mathcal{B}$. 
\end{defn}

\begin{rem}
It is clear that $P\mathcal{A}$ is $P$-cocomplete for every $\mathcal{A}\in\mathscr{C}$. 
\end{rem}

We now recall the notion of $P$-admissibility in the setting of a
KZ doctrine $P$. This notion of admissibility is useful for showing
that certain left extensions exist, and moreover are preserved. Note
that this notion will be used regularly throughout the paper.
\begin{defn}
\label{admequiv} Given a KZ doctrine $\left(P,y\right)$ on a 2-category
$\mathscr{C}$, we say a 1-cell $L\colon\mathcal{A}\to\mathcal{B}$
is \emph{$P$-admissible} if any of the following equivalent conditions
are met:
\end{defn}

\begin{enumerate}
\item In the left diagram below
\[
\xymatrix@=1em{\mathcal{B}\ar@{->}[rr]^{R_{L}} &  & P\mathcal{A}\ar@{}[ld]|-{\stackrel{\varphi_{L}}{\Longleftarrow}} &  &  & \mathcal{B}\ar@{->}[rr]^{R_{L}} &  & P\mathcal{A}\ar@{}[ld]|-{\stackrel{\varphi_{L}}{\Longleftarrow}}\ar[rr]^{\overline{H}}\ar@{}[rd]|-{\stackrel{c_{H}}{\Longleftarrow}} &  & \mathcal{X}\\
 & \; &  &  &  &  & \; &  & \;\\
 &  & \mathcal{A}\ar[uu]_{y_{\mathcal{A}}}\ar[ulul]^{L} &  &  &  &  & \mathcal{A}\ar[uu]_{y_{\mathcal{A}}}\ar[ulul]^{L}\ar[rruu]_{H}
}
\]
 there exists a left extension $\left(R_{L},\varphi_{L}\right)$ of
$y_{\mathcal{A}}$ along $L$, and moreover the left extension is
preserved by any $\overline{H}$ as in the right diagram where $\mathcal{X}$
is $P$-cocomplete;
\item Every $P$-cocomplete object $\mathcal{X}\in\mathscr{C}$ admits,
and $P$-homomorphism preserves, left extensions along $L$. This
says that for any given 1-cell $K:\mathcal{A}\to\mathcal{X}$, where
$\mathcal{X}$ is $P$-cocomplete, there exists a 1-cell $J$ and
2-cell $\delta$ as in the left diagram below 
\[
\xymatrix@=1em{\mathcal{B}\ar@{->}[rr]^{J} &  & \mathcal{X}\ar@{}[ld]|-{\stackrel{\delta}{\Longleftarrow}} &  &  & \mathcal{B}\ar@{->}[rr]^{J} &  & \mathcal{X}\ar@{}[ld]|-{\stackrel{\delta}{\Longleftarrow}}\ar[rr]^{E} &  & \mathcal{Y}\\
 & \; &  &  &  &  & \; &  & \;\\
 &  & \mathcal{A}\ar[uu]_{K}\ar[ulul]^{L} &  &  &  &  & \mathcal{A}\ar[uu]_{K}\ar[ulul]^{L}
}
\]
exhibiting $J$ as a left extension, and moreover this left extension
is preserved by any $P$-homomorphism $E\colon\mathcal{X}\to\mathcal{Y}$
for $P$-cocomplete $\mathcal{Y}$ as in the right diagram;
\item $PL:=\textnormal{lan}_{L}$ given as the left extension
\[
\xymatrix@=1em{P\mathcal{A}\ar@{->}[rr]^{PL}\ar@{}[rdrd]|-{\stackrel{c_{y_{\mathcal{B}\cdot L}}}{\Longleftarrow}} &  & P\mathcal{B}\\
\\
\mathcal{A}\ar[rr]_{L}\ar[uu]^{y_{\mathcal{A}}} &  & \mathcal{B}\ar[uu]_{y_{\mathcal{B}}}
}
\]
has a right adjoint. 
\end{enumerate}
\begin{rem}
For a proof that the descriptions (1), (2) and (3) above are equivalent,
we refer the reader to \cite{bungefunk} or \cite{yonedakz}.
\end{rem}

It is well known that pointwise left extensions along fully faithful
maps are exhibited by invertible 2-cells; in the following definition
we give an analogue of this fact for KZ doctrines.
\begin{defn}
Given a KZ doctrine $\left(P,y\right)$ on a 2-category $\mathscr{C}$,
we say a 1-cell $L\colon\mathcal{A}\to\mathcal{B}$ is \emph{$P$-fully
faithful} if $PL$ is fully faithful.
\end{defn}

\begin{rem}
The importance of the \emph{$P$-fully faithful }maps stems from the
fact that for a $P$-admissible map $L\colon\mathcal{A}\to\mathcal{B},$
this $L$ is $P$-fully faithful if and only if every left extension
along $L$ into a $P$-cocomplete object is exhibited by an isomorphism
\cite[Remark 24]{yonedakz}. Clearly each $y_{\mathcal{A}}$ is both
$P$-admissible and $P$-fully faithful.
\end{rem}

For any given KZ doctrine $P$ on a 2-category $\mathscr{C}$ a natural
question to ask is: what are the $P$-cocomplete objects; $P$-homomorphisms;
$P$-admissible maps and $P$-fully faithful maps? Let us consider
a couple of examples.
\begin{example}
A well known example of a KZ doctrine is the free small cocompletion
operation on locally small categories, which sends a locally small
category $\mathcal{A}$ to its category of small presheaves. In particular,
when $\mathcal{A}$ is small the free small cocompletion is $P\mathcal{A}=\left[\mathcal{A}^{\textnormal{op}},\mathbf{Set}\right]$.
In this example, the $P$-cocomplete objects are those locally small
categories which are small cocomplete and the $P$-homomorphisms are
those functors between such categories preserving small colimits.
The $P$-admissible maps are those functors $L\colon\mathcal{A}\to\mathcal{B}$
for which $\mathcal{B}\left(L-,-\right)\colon\mathcal{B}\to\left[\mathcal{A}^{\textnormal{op}},\mathbf{Set}\right]$
factors through $P\mathcal{A}$. Of these $P$-admissible maps, the
$P$-fully faithful maps are precisely the fully faithful functors.
\end{example}

Another example is the free \emph{large} cocompletion KZ doctrine
on locally small categories. The reader should keep in mind a theorem
of Freyd showing that any locally small category which admits all
large colimits is a preorder. Consequently, a locally small category
is large cocomplete precisely when it is a preorder with all large
joins. This KZ doctrine has some unusual properties. For example it
is a cocompletion KZ doctrine (in the simple sense that its algebras
are described as categories admitting a certain class of colimits)
with unit components not always fully faithful. Moreover, every functor
is admissible against the large cocompletion. We define this KZ doctrine
$P\colon\mathbf{Cat}\to\mathbf{Cat}$ by the assignment
\[
\begin{aligned}P\colon\textnormal{ob}\,\mathbf{Cat}\to\textnormal{ob}\,\mathbf{Cat}\colon & \mathcal{A}\mapsto\left[\mathcal{A}^{\textnormal{op}},\mathbbm2\right]\end{aligned}
\]
with unit maps for each $\mathcal{A}\in\mathbf{Cat}$ given by
\[
y_{\mathcal{A}}\colon\mathcal{A}\to\left[\mathcal{A}^{\textnormal{op}},\mathbbm2\right]\colon X\mapsto\mathcal{A}\left\langle -,X\right\rangle 
\]
with each $\mathcal{A}\left\langle -,X\right\rangle $ is defined
as
\[
\mathcal{A}\left\langle -,X\right\rangle \colon\mathcal{A}^{\textnormal{op}}\to\mathbbm2\colon S\mapsto\begin{cases}
1, & \exists\;S\overset{f}{\longrightarrow}X\textnormal{ in }\mathcal{A}\\
0, & \textnormal{otherwise}.
\end{cases}
\]
For any functor $F\colon\mathcal{A}\to\mathcal{D}$ where $\mathcal{D}$
is a preordered category with all large joins (such as $P\mathcal{B}$
for any $\mathcal{B}$) we may define a left extension $\overline{F}\colon\left[\mathcal{A}^{\textnormal{op}},\mathbbm2\right]\to\mathcal{D}$
as in the left diagram
\[
\xymatrix@=0.5em{\left[\mathcal{A}^{\textnormal{op}},\mathbbm2\right]\myar{\overline{F}}{rr} &  & \mathcal{D}\\
 & \;\ar@{}[ul]|-{\overset{\textnormal{id}}{\Longleftarrow}\quad} &  &  &  &  & {\displaystyle \overline{F}\left(H\right)=\sup_{X\in\mathcal{A}\colon HX=1}FX}\\
\mathcal{A}\ar[uu]^{y_{\mathcal{A}}}\ar[rruu]_{F}
}
\]
by the assignment on the right. Hence for this KZ doctrine, the $P$-cocomplete
objects are the large cocomplete categories, and the $P$-homomorphisms
are the order and join preserving maps between such categories. Every
map is $P$-admissible, and it is easily checked that a map $L\colon\mathcal{A}\to\mathcal{B}$
is $P$-fully faithful precisely when there exists a map $X\to Y$
in $\mathcal{A}$ if and only if there exists a map $LX\to LY$ in
$\mathcal{B}$.
\begin{rem}
For a set $X$ seen as a discrete category, the large cocompletion
of $X$ is $\left(\mathscr{P}X,\supseteq\right)$; and dually, the
large completion is $\left(\mathscr{P}X,\subseteq\right)$, where
$\mathscr{P}X$ is the powerset of $X$.
\end{rem}

\section{Pseudo-Distributive Laws over KZ Doctrines\label{liftingkzdoctrines}}

It was shown by Marmolejo that pseudo-distributive laws of a (co)KZ
doctrine over a KZ doctrine have a particularly simple form \cite[Definition 11.4]{marm1999}.
Here we show that one can give a description which is both simpler
(in that less coherence axioms are required) and more general (in
that the assumption of the former pseudomonad being (co)KZ may be
dropped). Hence the problem of lifting a cocompletion operation to
the 2-category of pseudo algebras may be more easily understood. 

Part of the motivation of our method comes from the observation that
if a KZ doctrine lifts to a pseudomonad on the 2-category of pseudo
algebras, then this pseudomonad is a KZ doctrine automatically\footnote{A fact perhaps most easily seen from the adjoint string definition \cite{marm1997}, in view of doctrinal adjunction \cite{doctrinal}.}.
Indeed, this fact means we may consider the problem of lifting a KZ
pseudomonad in terms of algebraic left extensions. 

In the proof we will make regular use of the admissibility perspective;
in fact, the preservation of admissible maps is crucial here, and
it is the main goal of this paper to describe such pseudo-distributive
laws in terms of this admissibility property. 

The proof of these results is quite technical, though the results
are summarized in Theorem \ref{liftkzequiv}.

\subsection{Notions of Pseudo-Distributive Laws}

Beck \cite{beckdist} defined a distributive law of a monad $\left(T,u,m\right)$
over another monad $\left(P,y,\mu\right)$ on a category $\mathcal{C}$
to be a natural transformation $\lambda\colon TP\to PT$ rendering
commutative the four diagrams
\[
\xymatrix@=1em{ & TP\ar[rr]^{\lambda} &  & PT &  &  &  & TP\ar[rr]^{\lambda} &  & PT\\
 &  & \;\ar@{}[ru]|-{=} &  &  &  &  &  & \;\ar@{}[ru]|-{=}\\
 &  &  & P\ar[uu]_{Pu}\ar[lluu]^{uP} &  &  &  &  &  & T\ar[uu]_{yT}\ar[lluu]^{Ty}\\
TTP\ar[dd]_{mP}\ar[rr]^{T\lambda}\;\ar@{}[rddrrr]|-{=} &  & TPT\ar[rr]^{\lambda T} &  & PTT\ar[dd]^{Pm} &  & TPP\ar@{}[rddrrr]|-{=}\ar[dd]_{T\mu}\ar[rr]^{\lambda P} &  & PTP\ar[rr]^{P\lambda} &  & PPT\ar[dd]^{\mu T}\\
\\
TP\ar[rrrr]_{\lambda} &  &  &  & PT &  & TP\ar[rrrr]_{\lambda} &  &  &  & PT
}
\]

A well known example on $\mathbf{Set}$ is the canonical distributive
law of the monad for monoids over the monad for abelian groups (whose
composite is the monad for rings). 

More generally, one may talk about a pseudo-distributive law of a
pseudomonad over another pseudomonad on a 2-category \cite{marm1999,kellycoherence,thesisTanaka,cheng2003}.
In this generalization the four conditions above are replaced by four
pieces of data (four invertible modifications) which are then required
to satisfy multiple coherence axioms, which we will omit here.
\begin{defn}
\label{dist} A \emph{pseudo-distributive law }of a pseudomonad $\left(T,u,m\right)$
over a pseudomonad $\left(P,y,\mu\right)$ on a 2-category $\mathscr{C}$
consists of a pseudonatural transformation $\lambda\colon TP\to PT$,
along with four invertible modifications $\omega_{1},\omega_{2},\omega_{3}$
and $\omega_{4}$ in place of the four equalities above. These four
modifications are subject to eight coherence axioms; see \cite{marm2008,marm1999}.

As a convention, we choose the direction of these four modifications
to be from right to left in the above four diagrams. 
\end{defn}

In this section, as in the background, we differentiate between ``KZ
doctrine'' defined in terms of left extensions, and ``KZ pseudomonad''
defined algebraically. 

We now define a pseudo-distributive law over such a KZ pseudomonad,
though showing this data and these coherence conditions suffice will
take some work. 
\begin{defn}
\label{distkzpseudomonad} Suppose we are given a 2-category $\mathscr{C}$
equipped with a pseudomonad $\left(T,u,m\right)$ and a KZ pseudomonad
$\left(P,y,\mu\right)$. Then a \emph{pseudo-distributive law over
a KZ pseudomonad} $\lambda\colon TP\to PT$ consists of a pseudonatural
transformation $\lambda\colon TP\to PT$ along with three invertible
modifications\footnote{Note the direction of the modifications are different in \cite{marm1999}. We use here the direction in which they will naturally arise from left extension and admissiblilty properties. Our direction agrees with that of \cite[Section 4]{tholen}.}
\[
\xymatrix@=1em{TP\ar[rr]^{\lambda} &  & PT &  & TP\ar[rr]^{\lambda} &  & PT &  & TTP\ar[dd]_{mP}\ar[rr]^{T\lambda}\;\ar@{}[rddrrr]|-{\overset{\omega_{3}}{\Longleftarrow}} &  & TPT\ar[rr]^{\lambda T} &  & PTT\ar[dd]^{Pm}\\
 & \;\ar@{}[ru]|-{\overset{\omega_{1}}{\Longleftarrow}} &  &  &  & \;\ar@{}[ru]|-{\overset{\omega_{2}}{\Longleftarrow}}\\
 &  & P\ar[uu]_{Pu}\ar[lluu]^{uP} &  &  &  & T\ar[uu]_{yT}\ar[lluu]^{Ty} &  & TP\ar[rrrr]_{\lambda} &  &  &  & PT
}
\]
subject to the three coherence axioms: 
\[
\xymatrix@=1em{ &  &  &  &  &  &  &  & TP\ar[ddll]_{TyP}\ar[rr]^{\lambda}\ar[dd]_{yTP}\ar@{}[ddrr]|-{\stackrel{y_{\lambda}^{-1}}{\Longleftarrow}\qquad\;\;} &  & PT\ar@/^{1pc}/[dd]^{PyT}\ar@/_{1pc}/[dd]_{yPT}\ar@{}[dd]|-{\stackrel{\theta T}{\Longleftarrow}}\\
TP\ar@/^{1pc}/[dd]^{TPy}\ar@/_{1pc}/[dd]_{TyP}\ar@{}[dd]|-{\stackrel{T\theta}{\Longleftarrow}}\ar[rr]^{\lambda}\ar@{}[drdr]|-{\qquad\quad\stackrel{\lambda_{y}}{\Longleftarrow}} &  & PT\ar[dd]^{PTy}\ar[rrdd]^{PyT} &  & \overset{\textnormal{coh }1}{=} &  &  & \;\ar@{}[dr]|-{\overset{\omega_{2}P}{\Longleftarrow}}\\
 &  &  & \;\ar@{}[dl]|-{\overset{P\omega_{2}}{\Longleftarrow}} &  & \; & TPP\ar[rr]_{\lambda P} &  & PTP\ar[rr]_{P\lambda} &  & PPT\ar[rr]_{\mu T} &  & PT\\
TPP\ar[rr]_{\lambda P} &  & PTP\ar[rr]_{P\lambda} &  & PPT\ar[rr]_{\mu T} &  & PT
}
\]
\[
\xymatrix@=1em{ &  &  &  &  &  &  & TP\ar[rd]^{\lambda}\\
P\ar[rr]^{uP}\ar@{}[drdr]|-{\stackrel{u_{y}}{\Longleftarrow}} &  & TP\ar[rr]^{\lambda}\ar@{}[dr]|-{\overset{\omega_{2}}{\Longleftarrow}} &  & PT &  & P\ar[rr]_{Pu}\ar[ru]^{uP}\ar@{}[drdr]|-{\stackrel{y_{u}^{-1}}{\Longleftarrow}} & \;\ar@{}[u]|-{\;\Uparrow\omega_{1}} & PT\\
 &  &  & \; &  & \overset{\textnormal{coh }2}{=}\\
1\ar[rr]_{u}\ar[uu]^{y} &  & T\ar[uu]_{Ty}\ar[uurr]_{yT} &  &  &  & 1\ar[rr]_{u}\ar[uu]^{y} &  & T\ar[uu]_{yT}
}
\]
\[
\xymatrix@=1em{ &  &  &  &  &  &  &  &  & TP\ar[rddrr]^{\lambda}\\
\\
TTP\ar[rr]^{mP} &  & TP\ar@{}[dldl]|-{\stackrel{m_{y}}{\Longleftarrow}}\ar[rr]^{\lambda}\ar@{}[dr]|-{\overset{\omega_{2}}{\Longleftarrow}} &  & PT &  & TTP\ar[rruru]^{mP}\ar[rr]^{T\lambda} &  & TPT\ar[rr]^{\lambda T} & \;\ar@{}[uu]|-{\Uparrow\omega_{3}} & PTT\ar[rr]^{Pm}\ar@{}[dd]|-{\stackrel{y_{m}^{-1}}{\Longleftarrow}} &  & PT\\
 &  &  & \; &  & \overset{\textnormal{coh }3}{=} &  & \;\ar@{}[ur]|-{\overset{T\omega_{2}}{\Longleftarrow}} &  & \;\ar@{}[ul]|-{\overset{\omega_{2}T}{\Longleftarrow}}\\
TT\ar[rr]_{m}\ar[uu]^{T^{2}y} &  & T\ar[uu]_{Ty}\ar[uurr]_{yT} &  &  &  &  &  & TT\ar[rr]_{m}\ar[uull]^{T^{2}y}\ar[rruu]_{yT^{2}}\ar[uu]^{TyT} &  & T\ar[uurr]_{yT}
}
\]
\end{defn}

\begin{rem}
(1) We will see later that $\omega_{1}$ and $\omega_{3}$ are uniquely
determined by $\omega_{2}$, due to the last two axioms and left extension
properties. (2) Actually, even the naturality cells of $\lambda$
may be determined given $\omega_{2}$ and the first coherence axiom.
(3) With the 2-cells $\omega_{1}$ and $\omega_{3}$ and the last
two coherence axioms omitted, we still have sufficient data to lift
$P$ to lax $T$-algebras. (4) These last two axioms may be seen as
invertibility conditions on $\omega_{1}$ and $\omega_{3}$, analogous
to those in \cite[Definition 11.4]{marm1999}. (5) During the proof,
we will see that each component $\omega_{2}^{\mathcal{A}}$ necessarily
exhibits each component $\lambda_{\mathcal{A}}$ as a left extension.
As $\omega_{2}$ uniquely determines the rest of the data, this will
show that such pseudo-distributive laws are essentially unique. (6)
In fact, the first coherence axiom above is equivalent to preservation
of admissible maps, in the presence of such a pseudonatural transformation
$\lambda$ and invertible modification $\omega_{2}$.
\end{rem}

We will need a notion of separately cocontinuous in the context of
KZ doctrines, and so we define the following. 
\begin{defn}
\label{defTadmccts} Suppose we are given a 2-category $\mathscr{C}$
equipped with a pseudomonad $\left(T,u,m\right)$ and a KZ doctrine
$\left(P,y\right)$. We define a 1-cell $z\colon T\mathcal{X}\to\mathcal{C}$
where $\mathcal{X}$ and $\mathcal{C}$ are $P$-cocomplete objects
to be:

\begin{enumerate}
\item \emph{$T_{P}$-cocontinuous} when every left extension along a unit
component $y_{\mathcal{A}}\colon\mathcal{A}\to P\mathcal{A}$ into
$\mathcal{X}$ is $T$-preserved by $z$;
\item \emph{$T_{P}$-adm-cocontinuous} when every left extension along a
$P$-admissible map $L\colon\mathcal{A}\to\mathcal{B}$ into $\mathcal{X}$
is $T$-preserved by $z$;
\end{enumerate}
\end{defn}

\begin{rem}
We will see later in Proposition \ref{admcocontinuousequiv} that
these two notions are equivalent in the presence of a pseudo-distributive
law of $T$ over $P$.
\end{rem}

We are now ready to give the definition of a pseudo-distributive law
over a KZ doctrine in terms of admissibility and left extensions.
\begin{defn}
\label{distkzdoctrine} Suppose we are given a 2-category $\mathscr{C}$
equipped with a pseudomonad $\left(T,u,m\right)$ and a KZ doctrine
$\left(P,y\right)$. Then a \emph{pseudo-distributive law over a KZ
doctrine} $\lambda\colon TP\to PT$ consists of the following assertions:
\begin{enumerate}
\item $T$ preserves $P$-admissible maps;
\end{enumerate}
\noindent and for every $\mathcal{A}\in\mathscr{C}$, \begin{enumerate}\setcounter{enumi}{1}

\item the exhibiting 2-cell $\omega_{2}^{\mathcal{A}}$ of the left
extension $\lambda_{\mathcal{A}}$\footnote{The left extension is unique up to coherent isomorphism, and exists since $Ty_{\mathcal{A}}$ is $P$-admissible.}
in
\[
\xymatrix@=1em{TP\mathcal{A}\ar[rr]^{\lambda_{\mathcal{A}}} &  & PT\mathcal{A}\ar@{}[dl]|-{\stackrel{\omega_{2}^{\mathcal{A}}}{\Longleftarrow}}\\
 & \;\\
 &  & T\mathcal{A}\ar[uu]_{y_{T\mathcal{A}}}\ar[uull]^{Ty_{\mathcal{A}}}
}
\]
is invertible\footnote{Equivalently one could ask that each $Ty_{\mathcal{A}}$  is $P$-fully faithful \cite[Prop. 23]{yonedakz}.};

\item the 1-cell $\lambda_{\mathcal{A}}$ above is $T_{P}$-cocontinuous\footnote{Equivalently one could ask that each $\lambda_{\mathcal{A}}$ is $T_{P}$-adm-cocontinuous.}; 

\item the respective diagrams
\[
\xymatrix@=1em{P\mathcal{A}\ar[rr]^{u_{P\mathcal{A}}} &  & TP\mathcal{A}\ar@{}[dldl]|-{\stackrel{u_{y_{\mathcal{A}}}}{\Longleftarrow}}\ar[rr]^{\lambda_{\mathcal{A}}}\ar@{}[dr]|-{\stackrel{\omega_{2}^{\mathcal{A}}}{\Longleftarrow}} &  & PT\mathcal{A} &  & T^{2}P\mathcal{A}\ar[rr]^{m_{P\mathcal{A}}} &  & TP\mathcal{A}\ar@{}[dldl]|-{\stackrel{m_{y_{\mathcal{A}}}}{\Longleftarrow}}\ar[rr]^{\lambda_{\mathcal{A}}}\ar@{}[dr]|-{\stackrel{\omega_{2}^{\mathcal{A}}}{\Longleftarrow}} &  & PT\mathcal{A}\\
 &  &  & \; &  &  &  &  &  & \;\\
\mathcal{A}\ar[rr]_{u_{\mathcal{A}}}\ar[uu]^{y_{\mathcal{A}}} &  & T\mathcal{A}\ar[uu]|-{Ty_{\mathcal{A}}}\ar[uurr]_{y_{T\mathcal{A}}} &  &  &  & T^{2}\mathcal{A}\ar[rr]_{m_{\mathcal{A}}}\ar[uu]^{T^{2}y_{\mathcal{A}}} &  & T\mathcal{A}\ar[uu]|-{Ty_{\mathcal{A}}}\ar[uurr]_{y_{T\mathcal{A}}}
}
\]
exhibit both $\lambda_{\mathcal{A}}\cdot u_{P\mathcal{A}}$ and $\lambda_{\mathcal{A}}\cdot m_{P\mathcal{A}}$
as left extensions. \end{enumerate}

\end{defn}

\begin{rem}
Note that a pseudo-distributive law as defined above is unique, as
it contains only assertions, and these assertions are invariant under
the choice of left left extension (unique up to coherent isomorphism).
\end{rem}

\subsection{The Main Theorem}

We are now ready to state the main result of this section (and this
paper), justifying our definitions above.
\begin{thm}
\label{liftkzequiv} Suppose we are given a 2-category $\mathscr{C}$
equipped with a pseudomonad $\left(T,u,m\right)$ and a KZ pseudomonad
$\left(P,y,\mu\right)$. Then the following are equivalent:

(a) $P$ lifts to a KZ doctrine $\widetilde{P}$ on $\text{ps-}T\text{-alg}$;

(b) $P$ lifts to a KZ pseudomonad $\widetilde{P}$ on $\text{ps-}T\text{-alg}$;

(c) $P$ lifts to a pseudomonad $\widetilde{P}$ on $\text{ps-}T\text{-alg}$;

(d) There exists a pseudo-distributive law over a KZ doctrine $\lambda\colon TP\to PT$;

(e) There exists a pseudo-distributive law over a KZ pseudomonad $\lambda\colon TP\to PT$;

(f) There exists a pseudo-distributive law $\lambda\colon TP\to PT$.
\end{thm}

The proof of this theorem is lengthy, and so we will leave the more
difficult aspects of the proof for subsequent subsections. Before
moving on to these subsections, we give the remainder of the proof.
\begin{proof}
[Proof of Theorem \ref{liftkzequiv}] In order to prove this theorem,
we will complete the cycle of implications
\[
\xymatrix@R=0.1em{ & \left(a\right)\ar@{=>}[r] & \left(b\right)\ar@{=>}[rd]\\
\left(d\right)\ar@{::>}[ru] &  &  & \left(c\right)\ar@{=>}[ld]\\
 & \left(e\right)\ar@{::>}[lu] & \left(f\right)\ar@{=>}[l]
}
\]
where the more difficult implications left to later sections are dotted
above.

$\left(a\right)\Longrightarrow\left(b\right)\colon$ A KZ doctrine
gives rise to a pseudomonad whose structure forms a fully faithful
adjoint string by \cite[Theorem 4.1]{marm2012}, and this in turn
gives rise to a KZ pseudomonad by \cite[Prop. 3.1, Lemma 3.2]{marm1997}.

$\left(b\right)\Longrightarrow\left(c\right)\colon$ This implication
is trivial.

$\left(c\right)\Longrightarrow\left(f\right)\colon$ For the correspondence
between pseudo-distributive laws and liftings to pseudo $T$-algebras
see \cite[Theorem 5.4]{cheng2003}.

$\left(f\right)\Longrightarrow\left(e\right)\colon$ Given a pseudo-distributive
law $\lambda\colon TP\to PT$ where $P$ is a KZ pseudomonad, to check
that we then have a pseudo-distributive law over a KZ pseudomonad
in the sense of Definition \ref{distkzpseudomonad} we need only check
the first axiom. But this axiom follows from coherences $7$ and $8$
as given in \cite[Section 4]{marm1999} along with the KZ pseudomonad
coherence axiom \ref{kzcoh2}.

$\left(e\right)\Longrightarrow\left(d\right)\colon$ This is shown
later in Theorem \ref{eimpliesd}.

$\left(d\right)\Longrightarrow\left(a\right)\colon$ This is shown
later in Theorem \ref{dimpliesa}. 
\end{proof}

\subsection{Distributive Laws over KZ Monads to those over KZ Doctrines}

We will devote this entire subsection to showing that a pseudo-distributive
law over a KZ pseudomonad, as in Definition \ref{distkzpseudomonad},
gives rise to a pseudo-distributive law over a KZ doctrine, as in
Definition \ref{distkzdoctrine}. This is $\left(e\right)\Longrightarrow\left(d\right)$
of Theorem \ref{liftkzequiv}. As this is the most difficult implication
to show, we will break the proof up into a number of propositions
and lemmata, starting with the following proposition. 

\textbf{Note for reader.} During this subsection and the next, the
reader will keep the three equivalent characterizations of $P$-admissible
maps (given in Definition \ref{admequiv}) in mind. Indeed, all three
characterizations are to be used repeatedly throughout these two subsections. 
\begin{rem}
Most of our diagrams are constructed from the following 2-cells, where
$P$ is a KZ doctrine and $T$ a pseudomonad on a bicategory $\mathscr{C}$:
\begin{enumerate}
\item As noted in Definition \ref{admequiv}, for any $P$-admissible 1-cell
$L\colon\mathcal{A}\to\mathcal{B}$ we have a left extension $\left(R_{L},\varphi_{L}\right)$
of $y_{\mathcal{A}}$ along $L$. In particular if $L=Ty_{\mathcal{A}}$
is $P$-admissible, we will denote this left extension by $\left(\lambda_{\mathcal{A}},\omega_{2}^{\mathcal{A}}\right)$.
Moreover, by \cite[Remark 16]{yonedakz}, if we are given a chosen
right adjoint $\textnormal{res}_{L}$ to $PL$, then the canonical
way to define $\left(R_{L},\varphi_{L}\right)$ is by
\[
\xymatrix@=1em{\mathcal{B}\ar@{->}[rr]^{R_{L}} &  & P\mathcal{A}\ar@{}[ld]|-{\stackrel{\varphi_{L}}{\Longleftarrow}} &  &  &  & \mathcal{B}\ar[rr]^{y_{\mathcal{B}}} & \; & P\mathcal{B}\ar[r]^{\textnormal{res}_{L}} & P\mathcal{A}\\
 & \; &  &  & := &  &  &  & \ar@{}[ul]|-{\stackrel{y_{L}}{\Longleftarrow}} & P\mathcal{A}\ar@/^{0.5pc}/[ul]^{\textnormal{lan}_{L}}\ar@{}[ul]|-{\quad\stackrel{\eta}{\Longleftarrow}}\ar[u]_{\textnormal{id}_{P\mathcal{A}}}\\
 &  & \mathcal{A}\ar[uu]_{y_{\mathcal{A}}}\ar[ulul]^{L} &  &  &  &  &  &  & \mathcal{A}\ar@/^{1pc}/[ulull]^{L}\ar[u]_{y_{\mathcal{A}}}
}
\]
\item As noted in Definition \ref{defkzdoctrine}, for any 1-cell $F\colon\mathcal{A}\to P\mathcal{B}$
we have a left extension $\left(\overline{F},c_{F}\right)$ of $F$
along $y_{\mathcal{A}}$ with $c_{F}$ invertible. If $F=R_{L}$ for
a $P$-admissible $L$, we will denote this left extension by $\left(\textnormal{res}_{L},c_{R_{L}}\right)$,
and note that $\textnormal{res}_{L}$ defined this way is right adjoint
to $PL$ \cite[Lemma 13]{yonedakz}.
\end{enumerate}
\end{rem}

\begin{prop}
\label{lambdabeck} Suppose we are given a 2-category $\mathscr{C}$
equipped with a pseudomonad $\left(T,u,m\right)$ and a KZ doctrine
$\left(P,y\right)$. Further suppose that for each object $\mathcal{A}\in\mathscr{C}$,
$Ty_{\mathcal{A}}$ is \emph{$P$-}admissible, and the left extension\footnote{This left extension exists since $Ty_{\mathcal{A}}$ is $P$-admissible.}
which we denote $\lambda_{\mathcal{A}}$ in
\[
\xymatrix@=1em{TP\mathcal{A}\ar[rr]^{\lambda_{\mathcal{A}}} &  & PT\mathcal{A}\ar@{}[dl]|-{\stackrel{\omega_{2}^{\mathcal{A}}}{\Longleftarrow}}\\
 & \;\\
 &  & T\mathcal{A}\ar[uu]_{y_{T\mathcal{A}}}\ar[uull]^{Ty_{\mathcal{A}}}
}
\]
is exhibited by an isomorphism denoted $\omega_{2}^{\mathcal{A}}$.
Then for every \emph{$P$-}admissible 1-cell $L\colon\mathcal{A}\to\mathcal{B}$
such that $TL\colon T\mathcal{A}\to T\mathcal{B}$ is also \emph{$P$-}admissible,
the respective pastings
\begin{equation}
\xymatrix@=1em{PT\mathcal{A} &  & TP\mathcal{A}\ar[ll]_{\lambda_{\mathcal{A}}}\ar@{}[dl]|-{\stackrel{\omega_{2}^{\mathcal{A}}}{\implies}} &  & TP\mathcal{B}\ar[ll]_{T\textnormal{res}_{L}} &  & PT\mathcal{A} &  & PT\mathcal{B}\myard{\textnormal{res}_{TL}}{ll}\ar@{}[d]|-{\stackrel{c_{R_{TL}}}{\implies}} &  & TP\mathcal{B}\ar[ll]_{\lambda_{\mathcal{B}}}\\
 & \; &  & \;\ar@{}[ur]|-{\quad\stackrel{Tc_{R_{L}}}{\implies}}\ar@{}[dl]|-{\stackrel{T\varphi_{L}}{\implies}} &  &  &  &  & \;\ar@{}[d]|-{\stackrel{\varphi_{TL}}{\implies}} & \;\ar@{}[ur]|-{\;\stackrel{\omega_{2}^{\mathcal{B}}}{\implies}}\\
 &  & T\mathcal{A}\ar[uu]|-{Ty_{\mathcal{A}}}\ar[rr]_{TL}\ar[uull]^{y_{T\mathcal{A}}} &  & T\mathcal{B}\ar[uu]_{Ty_{\mathcal{B}}}\ar[uull]|-{TR_{L}} &  &  &  & T\mathcal{A}\ar[rr]_{TL}\ar[uull]^{y_{T\mathcal{A}}} &  & T\mathcal{B}\ar[uu]_{Ty_{\mathcal{B}}}\ar[uull]|-{y_{T\mathcal{B}}}\ar[uullll]|-{R_{TL}}
}
\label{defgamma}
\end{equation}
exhibit $\lambda_{\mathcal{A}}\cdot T\textnormal{res}_{L}$ and $\textnormal{res}_{TL}\cdot\lambda_{\mathcal{B}}$
as left extensions of $y_{T\mathcal{A}}$ along $Ty_{\mathcal{B}}\cdot TL$;
yielding an isomorphism of left extensions: 
\[
\xymatrix@=1em{TP\mathcal{B}\ar[rr]^{\lambda_{\mathcal{B}}}\ar[dd]_{T\textnormal{res}_{L}}\ar@{}[rdrd]|-{\Uparrow\gamma_{L}} &  & PT\mathcal{B}\ar[dd]^{\textnormal{res}_{TL}}\\
\\
TP\mathcal{A}\ar[rr]_{\lambda_{\mathcal{A}}} &  & PT\mathcal{A}
}
\]
 Moreover, if the left diagram below exhibits $R_{L}$ as a left extension
\[
\xymatrix@=1em{\mathcal{B}\ar[rr]^{R_{L}} &  & P\mathcal{A}\ar@{}[dl]|-{\stackrel{\varphi_{L}}{\Longleftarrow}} &  &  &  & T\mathcal{B}\ar[rr]^{TR_{L}} &  & TP\mathcal{A}\ar@{}[dl]|-{\stackrel{T\varphi_{L}}{\Longleftarrow}}\ar[rr]^{\lambda_{\mathcal{A}}}\ar@{}[rd]|-{\stackrel{\omega_{2}^{\mathcal{A}}}{\Longleftarrow}} &  & PT\mathcal{A}\\
 & \; &  &  &  &  &  & \; &  & \;\\
 &  & \mathcal{A}\ar[uull]^{L}\ar[uu]_{y_{\mathcal{A}}} &  &  &  &  &  & T\mathcal{A}\ar[uull]^{TL}\ar[uu]|-{Ty_{\mathcal{A}}}\ar[uurr]_{y_{T\mathcal{A}}}
}
\]
then the right diagram exhibits $\lambda_{\mathcal{A}}\cdot TR_{L}$
as a left extension. 
\end{prop}

\begin{proof}
Firstly, we consider the diagram
\[
\xymatrix@=1em{PT\mathcal{A} &  & TP\mathcal{A}\ar[ll]_{\lambda_{\mathcal{A}}}\ar@{}[dl]|-{\stackrel{\omega_{2}^{\mathcal{A}}}{\implies}} &  & TP\mathcal{B}\ar[ll]_{T\textnormal{res}_{L}}\\
 & \; &  & \;\ar@{}[ur]|-{\quad\stackrel{Tc_{R_{L}}}{\implies}}\ar@{}[dl]|-{\stackrel{T\varphi_{L}}{\implies}}\\
 &  & T\mathcal{A}\ar[uu]|-{Ty_{\mathcal{A}}}\ar[rr]_{TL}\ar[uull]^{y_{T\mathcal{A}}} &  & T\mathcal{B}\ar[uu]_{Ty_{\mathcal{B}}}\ar[uull]|-{TR_{L}}
}
\]
and note that $\lambda_{\mathcal{A}}\cdot T\textnormal{res}_{L}$
is a left extension since for any 1-cell $H\colon TP\mathcal{B}\to PT\mathcal{A}$
we have the natural bijections \def\fCenter{\ \rightarrow\ } 
\def\ScoreOverhang{2pt}
\settowidth{\rhs}{$H\cdot T\textnormal{lan}_{L}\cdot Ty_{\mathcal{A}}$} 
\settowidth{\lhs}{$\lambda_{\mathcal{A}}\cdot T\textnormal{res}_{L}$}
\begin{prooftree}
\Axiom$\makebox[\lhs][r]{$\lambda_{\mathcal{A}}\cdot T\textnormal{res}_{L}$} \fCenter \makebox[\rhs][l]{$H$}$
\RightLabel{$\qquad \textnormal{mates correspondence}$} 
\UnaryInf$\makebox[\lhs][r]{$\lambda_{\mathcal{A}}$} \fCenter \makebox[\rhs][l]{$H\cdot T\textnormal{lan}_{L}$}$ 
\RightLabel{$\qquad \text{since \ensuremath{\lambda_{\mathcal{A}}} is a left extension}$}
\UnaryInf$\makebox[\lhs][r]{$y_{T\mathcal{A}}$} \fCenter \makebox[\rhs][l]{$H\cdot T\textnormal{lan}_{L}\cdot Ty_{\mathcal{A}}$}$ 
\RightLabel{$\qquad PL\cdot y_{\mathcal{A}}\cong y_{\mathcal{B}}\cdot L$}
\UnaryInf$\makebox[\lhs][r]{$y_{T\mathcal{A}}$} \fCenter \makebox[\rhs][l]{$H\cdot Ty_{\mathcal{B}}\cdot TL$}$ 
\end{prooftree}and one may check this is the correct exhibiting 2-cell using \cite[Remark 16]{yonedakz}.
We may also consider the diagram
\[
\xymatrix@=1em{PT\mathcal{A} &  & PT\mathcal{B}\myard{\textnormal{res}_{TL}}{ll}\ar@{}[d]|-{\stackrel{c_{R_{TL}}}{\implies}} &  & TP\mathcal{B}\ar[ll]_{\lambda_{\mathcal{B}}}\\
 &  & \;\ar@{}[d]|-{\stackrel{\varphi_{TL}}{\implies}} & \;\ar@{}[ur]|-{\;\stackrel{\omega_{2}^{\mathcal{B}}}{\implies}}\\
 &  & T\mathcal{A}\ar[rr]_{TL}\ar[uull]^{y_{T\mathcal{A}}} &  & T\mathcal{B}\ar[uu]_{Ty_{\mathcal{B}}}\ar[uull]|-{y_{T\mathcal{B}}}\ar[uullll]|-{R_{TL}}
}
\]
and note that since $Ty_{\mathcal{B}}$ is \emph{$P$-}admissible
the left extension $\lambda_{\mathcal{B}}$ is preserved by $\textnormal{res}_{TL}$.
Noting $c_{R_{TL}}$ is invertible, we then apply the pasting lemma
for left extensions (the dual of \cite[Prop. 1]{yonedastructures})
to see the outside diagram exhibits $\textnormal{res}_{TL}\cdot\lambda_{\mathcal{B}}$
as a left extension. By uniqueness of left extensions, we derive our
desired isomorphism $\gamma_{L}\colon\lambda_{\mathcal{A}}\cdot T\textnormal{res}_{L}\cong\textnormal{res}_{TL}\cdot\lambda_{\mathcal{B}}$.
Now, to show that
\begin{equation}
\xymatrix@=1em{T\mathcal{B}\ar[rr]^{TR_{L}} &  & TP\mathcal{A}\ar@{}[dl]|-{\stackrel{T\varphi_{L}}{\Longleftarrow}}\ar[rr]^{\lambda_{\mathcal{A}}}\ar@{}[rd]|-{\stackrel{\omega_{2}^{\mathcal{A}}}{\Longleftarrow}} &  & PT\mathcal{A}\\
 & \; &  & \;\\
 &  & T\mathcal{A}\ar[uull]^{TL}\ar[uu]|-{Ty_{\mathcal{A}}}\ar[uurr]_{y_{T\mathcal{A}}}
}
\label{TpresR}
\end{equation}
exhibits $\lambda_{\mathcal{A}}\cdot TR_{L}$ as a left extension,
it suffices to show that we have an isomorphism $\lambda_{\mathcal{A}}\cdot TR_{L}\cong R_{TL}$
and that pasting the left extension $\left(R_{TL},\varphi_{TL}\right)$
with this isomorphism yields the above. This is the case since all
regions in the following diagram commute up to isomorphism
\[
\xymatrix@=1em{ &  & \ar@{}[d]|-{\cong Tc_{R_{L}}} &  & TP\mathcal{A}\ar[drr]^{\lambda_{\mathcal{A}}}\ar@{}[d]|-{\cong\gamma_{L}}\\
T\mathcal{B}\ar@/_{1.5pc}/[rrrr]_{y_{T\mathcal{B}}}\ar[rr]^{Ty_{\mathcal{B}}}\ar@/_{3pc}/[rrrrrr]_{R_{TL}}\ar@/^{1pc}/[rrrru]^{TR_{L}} &  & TP\mathcal{B}\ar[rr]^{\lambda_{\mathcal{B}}}\ar[urr]^{T\textnormal{res}_{L}} &  & PT\mathcal{B}\ar[rr]^{\textnormal{res}_{TL}} &  & PT\mathcal{A}\\
 &  & \ar@{}[u]|-{\cong\omega_{2}^{\mathcal{B}}} &  & \ar@{}[]|-{\cong c_{R_{TL}}}
}
\]
and it is easy to check $\varphi_{TL}$ pasted with this isomorphism
yields the pasting  \ref{TpresR} if one uses the definition of $\gamma_{L}$.
\end{proof}
\begin{rem}
Note that the above proposition tells us something about the components
of $\lambda$ being separately cocontinuous, without any assumptions
on pseudonaturality of $\lambda$. This may seem unusual in view of
the following lemma, in which we show pseudonaturality of $\lambda$
is precisely equivalent to the $T_{P}$-cocontinuity of its components.
\end{rem}

\begin{lem}
\label{lambdaccts} Suppose we are given a 2-category $\mathscr{C}$
equipped with a pseudomonad $\left(T,u,m\right)$ and a KZ doctrine
$\left(P,y\right)$. Further suppose that for each object $\mathcal{A}\in\mathscr{C}$,
$Ty_{\mathcal{A}}$ is \emph{$P$-}admissible and the left extension
which we call $\lambda_{\mathcal{A}}$ as on the left below
\[
\xymatrix@=1em{TP\mathcal{A}\ar[rr]^{\lambda_{\mathcal{A}}} &  & PT\mathcal{A}\ar@{}[dl]|-{\stackrel{\omega_{2}^{\mathcal{\mathcal{A}}}}{\Longleftarrow}} &  &  & TP\mathcal{B}\ar[rr]^{\lambda_{\mathcal{B}}} &  & PT\mathcal{B}\\
 & \;\\
 &  & T\mathcal{A}\ar[uu]_{y_{T\mathcal{A}}}\ar[uull]^{Ty_{\mathcal{A}}} &  &  & TP\mathcal{A}\ar[rr]_{\lambda_{\mathcal{A}}}\ar[uu]^{TPL}\ar@{}[urur]|-{\Uparrow\lambda_{L}} &  & PT\mathcal{A}\ar[uu]_{PTL}
}
\]
is exhibited by an isomorphism $\omega_{2}^{\mathcal{\mathcal{A}}}$.
Then for all $L\colon\mathcal{A}\to\mathcal{B}$ the naturality squares
for $\lambda$ as on the right above commute up to a coherent isomorphism
$\lambda_{L}$, with coherent meaning
\[
\xymatrix@=1em{ &  & TP\mathcal{B}\ar[rrd]^{\lambda_{\mathcal{B}}}\\
TP\mathcal{A}\ar[rr]^{\lambda_{\mathcal{A}}}\ar[rru]^{TPL} &  & PT\mathcal{A}\ar@{}[dl]|-{\stackrel{\omega_{2}^{\mathcal{\mathcal{A}}}}{\Longleftarrow}}\ar[rr]^{PTL}\ar@{}[rdrd]|-{\stackrel{y_{TL}^{-1}}{\Longleftarrow}}\ar@{}[u]|-{\Uparrow\lambda_{L}} &  & PT\mathcal{B} &  & TP\mathcal{A}\ar[rr]^{TPL}\ar@{}[ddrr]|-{\stackrel{Ty_{L}^{-1}}{\Longleftarrow}} &  & TP\mathcal{B}\ar[rr]^{\lambda_{\mathcal{B}}} &  & PT\mathcal{B}\\
 & \; &  &  &  & = &  & \; &  & \;\ar@{}[ul]|-{\stackrel{\omega_{2}^{\mathcal{B}}}{\Longleftarrow}}\\
 &  & T\mathcal{A}\ar[uu]|-{y_{T\mathcal{A}}}\ar[uull]^{Ty_{\mathcal{A}}}\ar[rr]_{TL} &  & T\mathcal{B}\ar[uu]_{y_{T\mathcal{B}}} &  & T\mathcal{A}\ar[uu]^{Ty_{\mathcal{A}}}\ar[rr]_{TL} &  & T\mathcal{B}\ar[uurr]_{y_{T\mathcal{B}}}\ar[uu]|-{Ty_{\mathcal{B}}}
}
\]
(the condition for $\omega_{2}$ to be a modification), if and only
if each $\lambda_{\mathcal{A}}$ is $T_{P}$-cocontinuous.
\end{lem}

\begin{proof}
The following implications prove the logical equivalence.

$\left(\Longrightarrow\right)\colon$ Suppose that for each $L\colon\mathcal{A}\to\mathcal{B}$
the naturality square of $\lambda$ commutes up to a coherent isomorphism
$\lambda_{L}$. Then noting that $\textnormal{id}_{P\mathcal{B}}=R_{y_{\mathcal{B}}}$,
we see that for any left extension as on the left (which is isomorphic
to $\left(\overline{F},c_{F}\right)$ by uniqueness)
\[
\xymatrix@=1em{P\mathcal{A}\ar[rr]^{PF}\ar@{}[rrdd]|-{\stackrel{y_{F}^{-1}}{\Longleftarrow}} &  & P^{2}\mathcal{B}\ar[rr]^{\textnormal{res}_{y_{\mathcal{B}}}} &  & P\mathcal{B} &  & TP\mathcal{A}\ar[rr]^{TPF}\ar@{}[rrdd]|-{\stackrel{Ty_{F}^{-1}}{\Longleftarrow}} &  & TP^{2}\mathcal{B}\ar[rr]^{T\textnormal{res}_{y_{\mathcal{B}}}} &  & TP\mathcal{B}\ar[rr]^{\lambda_{\mathcal{B}}} &  & PT\mathcal{B}\\
 &  &  & \;\ar@{}[ul]|-{\stackrel{c_{\textnormal{id}_{P\mathcal{B}}}}{\Longleftarrow}} &  &  &  &  &  & \;\ar@{}[ul]|-{\stackrel{Tc_{\textnormal{id}_{P\mathcal{B}}}}{\Longleftarrow}}\\
\mathcal{A}\ar[uu]^{y_{\mathcal{A}}}\ar[rr]_{F} &  & P\mathcal{B}\ar[uu]|-{y_{P\mathcal{B}}}\ar[uurr]_{\textnormal{id}_{P\mathcal{B}}} &  &  &  & T\mathcal{A}\ar[uu]^{Ty_{\mathcal{A}}}\ar[rr]_{TF} &  & TP\mathcal{B}\ar[uu]|-{Ty_{P\mathcal{B}}}\ar[uurr]_{T\textnormal{id}_{P\mathcal{B}}}
}
\]
it suffices to check that the right diagram above exhibits $\lambda_{\mathcal{B}}\cdot T\textnormal{res}_{y_{\mathcal{B}}}\cdot TPF$
as a left extension. To see this we note that the pasting
\[
\xymatrix@=1em{ &  & PT\mathcal{A}\ar[rr]^{PTF}\ar@{}[dd]|-{\Uparrow\lambda_{F}^{-1}} &  & PTP\mathcal{B}\ar[ddrr]^{\textnormal{res}_{Ty_{\mathcal{B}}}}\ar@{}[dd]|-{\Uparrow\gamma_{y_{\mathcal{B}}}}\\
\\
TP\mathcal{A}\ar[rr]^{TPF}\ar@{}[rrdd]|-{\stackrel{Ty_{F}^{-1}}{\Longleftarrow}}\ar[uurr]^{\lambda_{\mathcal{A}}} &  & TP^{2}\mathcal{B}\ar[rr]^{T\textnormal{res}_{y_{\mathcal{B}}}}\ar[uurr]^{\lambda_{P\mathcal{B}}} &  & TP\mathcal{B}\ar[rr]^{\lambda_{\mathcal{B}}} &  & PT\mathcal{B}\\
 &  &  & \;\ar@{}[ul]|-{\stackrel{Tc_{\textnormal{id}_{P\mathcal{B}}}}{\Longleftarrow}}\\
T\mathcal{A}\ar[uu]^{Ty_{\mathcal{A}}}\ar[rr]_{TF} &  & TP\mathcal{B}\ar[uu]|-{Ty_{P\mathcal{B}}}\ar[uurr]_{\textnormal{id}_{TP\mathcal{B}}}
}
\]
is equal to the pasting (using $\lambda_{\mathcal{B}}=R_{Ty_{\mathcal{B}}}$)
\[
\xymatrix@=1em{TP\mathcal{A}\ar[rr]^{\lambda_{\mathcal{A}}} &  & PT\mathcal{A}\ar[rr]^{PTF}\ar@{}[rdrd]|-{\stackrel{y_{TF}^{-1}}{\Longleftarrow}} &  & PTP\mathcal{B}\ar[rr]^{\textnormal{res}_{Ty_{\mathcal{B}}}} &  & PT\mathcal{B}\\
 & \;\ar@{}[ur]|-{\stackrel{\omega_{2}^{\mathcal{A}}}{\Longleftarrow}} &  &  &  & \;\ar@{}[ul]|-{\stackrel{c_{\lambda_{\mathcal{B}}}}{\Longleftarrow}}\\
 &  & T\mathcal{A}\ar[rr]_{TF}\ar[uu]|-{y_{T\mathcal{A}}}\ar[uull]^{Ty_{\mathcal{A}}} &  & TP\mathcal{B}\ar[uu]|-{y_{TP\mathcal{B}}}\ar[rruu]_{\lambda_{\mathcal{B}}}
}
\]
This is shown by first using the coherence condition on $\lambda_{F}^{-1}$,
and then using the definition of $\gamma_{y_{\mathcal{B}}}$ from
Proposition \ref{lambdabeck}. Note also this last diagram exhibits
$\textnormal{res}_{Ty_{\mathcal{B}}}\cdot PTF\cdot\lambda_{\mathcal{A}}$
as a left extension since $Ty_{\mathcal{A}}$ is \emph{$P$-}admissible
(using preservation of the left extension $\lambda_{\mathcal{A}}$
by $P$-homomorphisms). 

$\left(\Longleftarrow\right)\colon$ For any $L\colon\mathcal{A}\to\mathcal{B}$,
we know $PTL\cdot\lambda_{\mathcal{A}}$ is a left extension of $y_{T\mathcal{B}}\cdot TL$
along $Ty_{\mathcal{A}}$ since $Ty_{\mathcal{A}}$ is \emph{$P$-}admissible.
Also $\lambda_{\mathcal{B}}\cdot TPL$ is such a left extension as
$\lambda_{\mathcal{B}}$ is $T_{P}$-cocontinuous, giving us an isomorphism
of left extensions $\lambda_{F}$ coherent as in the statement of
this lemma.
\end{proof}
\begin{rem}
A Beck condition is satisfied here. Indeed, the 2-cell $\gamma_{L}$
as in Proposition \ref{lambdabeck} is the mate of $\lambda_{L}$
as in the above lemma. This may be seen by pasting the left diagram
of \ref{defgamma} with the mate of $\lambda_{L}$ and then recovering
the right diagram (making use of \cite[Remark 16]{yonedakz} and the
coherence condition on $\lambda_{L}$).
\end{rem}

In the following lemma we see that for a pseudo-distributive law over
a KZ pseudomonad as in Definition \ref{distkzpseudomonad}, the modification
components $\omega_{2}^{\mathcal{A}}$ necessarily exhibit each $\lambda_{\mathcal{A}}$
as a left extension, and from this we deduce the existence of invertible
components $\omega_{4}^{\mathcal{A}}$.
\begin{lem}
\label{w2isext} Suppose we are given a 2-category $\mathscr{C}$
equipped with a pseudomonad $\left(T,u,m\right)$ and a KZ pseudomonad
$\left(P,y,\mu\right)$. Suppose further that we are given a pseudo-distributive
law over a KZ pseudomonad $\lambda\colon TP\to PT$. Then for each
$\mathcal{A}\in\mathscr{C}$, $Ty_{\mathcal{A}}$ is \emph{$P$-}admissible,
exhibited by an adjunction
\[
PTy_{\mathcal{A}}\dashv\mu_{T\mathcal{A}}\cdot P\lambda_{\mathcal{A}}
\]
Moreover, the diagrams as on the left exhibit each $\lambda_{\mathcal{A}}$
as a left extension, 
\[
\xymatrix@=1em{TP\mathcal{A}\ar[rr]^{\lambda_{\mathcal{A}}} &  & PT\mathcal{A}\ar@{}[dl]|-{\stackrel{\omega_{2}^{\mathcal{\mathcal{A}}}}{\Longleftarrow}} &  &  &  & TPP\mathcal{A}\ar[rr]^{\lambda_{P\mathcal{A}}}\ar@{}[rddrrr]|-{\stackrel{\omega_{4}^{\mathcal{\mathcal{A}}}}{\Longleftarrow}}\ar[dd]_{T\mu_{\mathcal{A}}} &  & PTP\mathcal{A}\ar[rr]^{P\lambda_{\mathcal{A}}} &  & PPT\mathcal{A}\ar[dd]^{\mu_{T\mathcal{A}}}\\
 & \;\\
 &  & T\mathcal{A}\ar[uu]_{y_{T\mathcal{A}}}\ar[uull]^{Ty_{\mathcal{A}}} &  &  &  & TP\mathcal{A}\ar[rrrr]_{\lambda_{\mathcal{A}}} &  &  &  & PT\mathcal{A}
}
\]
and there exists canonical isomorphisms as on the right for each $\mathcal{A}$.
\end{lem}

\begin{proof}
We now prove the three assertions of the above lemma.

\noun{Each $Ty_{\mathcal{A}}$ is $P$-admissible.} Firstly, we note
that the below diagram exhibits $\mu_{T\mathcal{A}}\cdot P\lambda_{\mathcal{A}}$
as a left extension
\begin{equation}
\xymatrix@=1em{PTP\mathcal{A}\ar[rr]^{P\lambda_{\mathcal{A}}}\ar@{}[rdrd]|-{\stackrel{y_{\lambda_{\mathcal{A}}}^{-1}}{\Longleftarrow}} &  & P^{2}T\mathcal{A}\ar[rr]^{\mu_{T\mathcal{A}}}\ar@{}[rd]|-{\cong} &  & PT\mathcal{A}\\
 &  &  & \;\\
TP\mathcal{A}\ar[rr]_{\lambda_{\mathcal{A}}}\ar[uu]^{y_{TP\mathcal{A}}} &  & PT\mathcal{A}\ar[uu]|-{y_{PT\mathcal{A}}}\ar[rruu]_{\textnormal{id}_{PT\mathcal{A}}}
}
\label{defrightadj}
\end{equation}
Indeed, the construction of a KZ doctrine from a pseudomonad whose
structure forms a fully faithful adjoint string is outlined in \cite{marm2012},
and the above is an instance of this construction. Noting $\overline{\lambda_{\mathcal{A}}}=\mu_{T\mathcal{A}}\cdot P\lambda_{\mathcal{A}}$,
we define our unit $\eta$ as the unique solution to the left extension
problem
\[
\xymatrix@=1em{ & PTP\mathcal{A}\ar[rd]^{\overline{\lambda_{\mathcal{A}}}} &  &  & PT\mathcal{A}\ar[rr]^{PTy_{\mathcal{A}}}\ar@{}[rdrd]|-{\Uparrow y_{Ty_{\mathcal{A}}}^{-1}} &  & PTP\mathcal{A}\ar[rr]^{P\lambda_{\mathcal{A}}}\ar@{}[rdrd]|-{\Uparrow y_{\lambda_{\mathcal{A}}}^{-1}} &  & P^{2}T\mathcal{A}\ar[rr]^{\mu_{T\mathcal{A}}}\ar@{}[rd]|-{\cong} &  & PT\mathcal{A}\\
PT\mathcal{A}\ar[rr]_{\textnormal{id}}\ar[ru]^{PTy_{\mathcal{A}}} & \;\ar@{}[u]|-{\Uparrow\eta}\ar@{}[dld]|-{\quad\Uparrow\textnormal{id}} & PT\mathcal{A} & = &  &  &  &  &  & \;\\
 &  &  &  & T\mathcal{A}\ar[rr]_{Ty_{\mathcal{A}}}\ar[uu]^{y_{T\mathcal{A}}}\ar@/_{2pc}/[rrrr]_{y_{T\mathcal{A}}} &  & TP\mathcal{A}\ar[rr]_{\lambda_{\mathcal{A}}}\ar[uu]|-{y_{TP\mathcal{A}}}\ar@{}[d]|-{\Uparrow\omega_{2}^{\mathcal{A}}} &  & PT\mathcal{A}\ar[uu]|-{y_{PT\mathcal{A}}}\ar[rruu]_{\textnormal{id}}\\
T\mathcal{A}\ar[uu]^{y_{T\mathcal{A}}}\ar@/_{1pc}/[uurr]_{y_{T\mathcal{A}}} &  &  &  &  &  &  & \;
}
\]
Note that the unit $\eta$ must then be given by 
\[
\xymatrix@=1em{ &  & \;\ar@{}[d]|-{\Downarrow P\omega_{2}^{\mathcal{A}}} &  & \ar@{}[]|-{\quad\cong}\\
PT\mathcal{A}\ar[rr]^{PTy_{\mathcal{A}}}\ar@/^{1.7pc}/[rrrr]^{Py_{T\mathcal{A}}}\ar@/^{3.3pc}/[rrrrrr]^{\textnormal{id}_{PT\mathcal{A}}} &  & PTP\mathcal{A}\ar[rr]^{P\lambda_{\mathcal{A}}} &  & P^{2}T\mathcal{A}\ar[rr]^{\mu_{T\mathcal{A}}} &  & PT\mathcal{A}
}
\]
as $y\colon1\to P$ is a pseudonatural transformation. We define our
counit $\epsilon$ as the unique solution to the left extension problem
\[
\xymatrix@=1em{ &  & \;\ar@{}[d]|-{\Uparrow\epsilon} &  &  &  &  &  & \ar@{}[]|-{\quad\Uparrow\left(\omega_{2}^{P\mathcal{A}}\right)^{-1}}\\
PTP\mathcal{A}\ar[rr]^{\overline{\lambda_{\mathcal{A}}}}\ar@/^{2pc}/[rrrr]^{\textnormal{id}_{PTP\mathcal{A}}} &  & PT\mathcal{A}\ar[rr]^{PTy_{\mathcal{A}}} &  & PTP\mathcal{A} & = & TP\mathcal{A}\ar@/^{1pc}/[rr]^{Ty_{P\mathcal{A}}}\ar@/_{1pc}/[rr]_{TPy_{\mathcal{A}}}\ar@{}[rr]|-{\Uparrow T\theta}\ar@/^{3.8pc}/[rrrr]^{y_{TP\mathcal{A}}}\ar@/_{1pc}/[rrdd]_{\lambda_{\mathcal{A}}} &  & TP^{2}\mathcal{A}\ar[rr]^{\lambda_{P\mathcal{A}}} &  & PTP\mathcal{A}\\
 & \ar@{}[]|-{\cong\qquad}\\
TPA\ar[uu]^{y_{TP\mathcal{A}}}\ar@/_{1pc}/[rruu]_{\lambda_{\mathcal{A}}} &  &  &  &  &  &  &  & PTA\ar@/_{1pc}/[rruu]_{PTy_{\mathcal{A}}}\ar@{}[uu]|-{\Uparrow\lambda_{y_{\mathcal{A}}}}
}
\]
where the unnamed isomorphism above is \ref{defrightadj}. One could
also define $\epsilon$ directly in terms of $\theta$, but that would
result in a more complicated proof. Of the triangle identities:
\[
\xymatrix@=1em{PTy_{\mathcal{A}}\ar[ddrrr]_{\textnormal{id}_{PTy_{\mathcal{A}}}}\myar{PTy_{\mathcal{A}}\cdot\eta}{rrr} &  &  & PTy_{\mathcal{A}}\cdot\overline{\lambda_{\mathcal{A}}}\cdot PTy_{\mathcal{A}}\ar[dd]^{\epsilon\cdot PTy_{\mathcal{A}}} &  & \overline{\lambda_{\mathcal{A}}}\ar[ddrrr]_{\textnormal{id}_{\overline{\lambda_{\mathcal{A}}}}}\myar{\eta\cdot\overline{\lambda_{\mathcal{A}}}}{rrr} &  &  & \overline{\lambda_{\mathcal{A}}}\cdot PTy_{\mathcal{A}}\cdot\overline{\lambda_{\mathcal{A}}}\ar[dd]^{\overline{\lambda_{\mathcal{A}}}\cdot\epsilon}\\
\\
 &  &  & PTy_{\mathcal{A}} &  &  &  &  & \overline{\lambda_{\mathcal{A}}}
}
\]
the left identity, which is equivalent to asking for equality when
whiskered by $y_{T\mathcal{A}}$, can be proven using that $\omega_{2}$
is a modification. The right triangle identity, which is equivalent
to asking for equality when whiskered by $y_{TP\mathcal{A}}$, amounts
to asking that the pasting

\[
\xymatrix@=1em{ &  &  &  & PT\mathcal{A}\ar@/^{1pc}/[ddddrr]^{y_{PT\mathcal{A}}}\ar@/^{1.5pc}/[ddddrrrr]^{\textnormal{id}_{PT\mathcal{A}}}\\
\\
 &  & \;\\
 &  & \ar@{}[u]|-{\Uparrow\left(\omega_{2}^{P\mathcal{A}}\right)^{-1}} &  &  &  & \ar@{}[r]|-{\cong} & \;\\
TP\mathcal{A}\ar@/^{1pc}/[rr]^{Ty_{P\mathcal{A}}}\ar@/_{1pc}/[rr]_{TPy_{\mathcal{A}}}\ar@{}[rr]|-{\Uparrow T\theta^{\mathcal{A}}}\ar@/^{3.8pc}/[rrrr]^{y_{TP\mathcal{A}}}\ar@/_{1pc}/[rrdd]_{\lambda_{\mathcal{A}}}\ar@/^{2.3pc}/[rrrruuuu]^{\lambda_{\mathcal{A}}} &  & TP^{2}\mathcal{A}\ar[rr]^{\lambda_{P\mathcal{A}}} &  & PTP\mathcal{A}\ar[rr]^{P\lambda_{\mathcal{A}}}\ar@{}[uuuu]|-{\Uparrow y_{\lambda_{\mathcal{A}}}} &  & P^{2}T\mathcal{A}\ar[rr]^{\mu_{T\mathcal{A}}} &  & PT\mathcal{A}\\
 &  & \; &  & \;\ar@{}[rru]|-{\Uparrow P\omega_{2}^{\mathcal{A}}\quad}\\
 &  & PT\mathcal{A}\ar@/_{1pc}/[rruu]^{PTy_{\mathcal{A}}}\ar@{}[uu]|-{\Uparrow\lambda_{y_{\mathcal{A}}}}\ar@/_{1pc}/[rruurr]_{Py_{T\mathcal{A}}}\ar@/_{1.6pc}/[rruurrrr]_{\textnormal{id}_{PT\mathcal{A}}} &  &  & \ar@{}[rruur]|-{\cong\quad\;\;}
}
\]
is the identity. This is where the first axiom for a pseudo-distributive
law over a KZ pseudomonad is used, in addition to the second coherence
axiom \ref{kzcoh2} of a KZ pseudomonad.

\noun{Each $\omega_{2}^{\mathcal{A}}$ exhibits $\lambda_{\mathcal{A}}$
as a left extension.} As $Ty_{\mathcal{A}}$ is \emph{$P$-}admissible,
we know by \cite[Remark 16]{yonedakz} that the pasting
\[
\xymatrix@=1em{TP\mathcal{A}\ar[rr]^{y_{TP\mathcal{A}}} &  & PTP\mathcal{A}\ar[rr]^{\textnormal{res}_{Ty_{\mathcal{A}}}} &  & PT\mathcal{A}\\
\;\ar@{}[rr]|-{\quad\quad\;\;\stackrel{y_{Ty_{\mathcal{A}}}}{\Longleftarrow}} &  & PT\mathcal{A}\ar[u]^{PTy_{\mathcal{A}}}\ar@{}[rru]|-{\stackrel{\eta\cdot y_{T\mathcal{A}}}{\Longleftarrow}}\\
 &  &  &  & T\mathcal{A}\ar[uu]_{y_{T\mathcal{A}}}\ar@/^{2pc}/[ululll]^{Ty_{\mathcal{A}}}\ar[ull]^{y_{T\mathcal{A}}}
}
\]
exhibits $\textnormal{res}_{Ty_{\mathcal{A}}}\cdot y_{TP\mathcal{A}}$
as a left extension, where $\textnormal{res}_{Ty_{\mathcal{A}}}=\overline{\lambda_{\mathcal{A}}}=\mu_{T\mathcal{A}}\cdot P\lambda_{\mathcal{A}}$,
and $\eta$ is the unit of $PTy_{\mathcal{A}}\dashv\textnormal{res}_{Ty_{\mathcal{A}}}$
as just defined. From a substitution of the definition of $\eta$
(and pasting with a couple of isomorphisms) we see that the pasting
\[
\xymatrix@=1em{ &  & \;\ar@{}[rd]|-{\cong}\\
 & PT\mathcal{A}\ar[rr]^{y_{PT\mathcal{A}}}\ar@/^{3pc}/[rrrd]^{\textnormal{id}_{PT\mathcal{A}}}\ar@{}[rd]|-{\Uparrow y_{\lambda_{\mathcal{A}}}} &  & P^{2}T\mathcal{A}\ar[rd]|-{\mu_{T\mathcal{A}}}\\
TP\mathcal{A}\ar[rr]^{y_{TP\mathcal{A}}}\ar[ru]^{\lambda_{\mathcal{A}}} &  & PTP\mathcal{A}\ar[ur]^{P\lambda_{\mathcal{A}}} & \;\ar@{}[l]|-{\stackrel{P\omega_{2}^{\mathcal{A}}}{\Longleftarrow}\;\;}\ar@{}[r]|-{\cong} & PT\mathcal{A}\\
 &  & \; & \;\\
 &  &  & PT\mathcal{A}\ar[uul]^{PTy_{\mathcal{A}}}\ar[rruul]_{\textnormal{id}}\ar[uuu]|-{Py_{T\mathcal{A}}} & \;\ar@{}[l]|-{=}\\
 &  & \;\ar@{}[uu]|-{\stackrel{y_{Ty_{\mathcal{A}}}}{\Longleftarrow}}\\
 &  &  &  & T\mathcal{A}\ar[uuuu]_{y_{T\mathcal{A}}}\ar@/^{2pc}/[uluuulll]^{Ty_{\mathcal{A}}}\ar[uul]^{y_{T\mathcal{A}}}
}
\]
exhibits $\lambda_{\mathcal{A}}$ as a left extension. Note that this
pasting is equal to $\omega_{2}$ as a consequence of $\omega_{2}$
being a modification as well as the coherence condition \ref{thirdaxiom}
satisfied by $P$.

\noun{There exists canonical isomorphisms $\omega_{4}^{\mathcal{A}}$.}
We have the left extension
\[
\xymatrix@=1em{TP^{2}\mathcal{A}\ar[rr]^{\lambda_{P\mathcal{A}}} &  & PTP\mathcal{A}\ar[rr]^{P\lambda_{\mathcal{A}}}\ar@{}[rddr]|-{\stackrel{y_{\lambda_{\mathcal{A}}}^{-1}}{\Longleftarrow}} &  & PPT\mathcal{A}\ar[rr]^{\mu_{T\mathcal{A}}} &  & PT\mathcal{A}\\
 & \;\ar@{}[ur]|-{\stackrel{\;\;\omega_{2}^{P\mathcal{A}}}{\Longleftarrow}} &  &  &  & \ar@{}[ul]|-{\cong}\\
 &  & TP\mathcal{A}\ar[uu]|-{y_{TP\mathcal{A}}}\ar[uull]^{Ty_{P\mathcal{A}}}\ar[rr]_{\lambda_{\mathcal{A}}} &  & PT\mathcal{A}\ar[uu]|-{y_{PT\mathcal{A}}}\ar[rruu]_{\textnormal{id}_{PT\mathcal{A}}}
}
\]
since $Ty_{P\mathcal{A}}$ is \emph{$P$-}admissible, and also the
left extension
\[
\xymatrix@=1em{TP^{2}\mathcal{A}\ar[rr]^{T\mu_{\mathcal{A}}} &  & TP\mathcal{A}\ar[rr]^{\lambda_{\mathcal{A}}} &  & PT\mathcal{A}\\
 & \;\ar@{}[ul]|-{\Uparrow Tc_{\textnormal{id}}}\\
TP\mathcal{A}\ar[rruu]_{T\textnormal{id}}\ar[uu]^{Ty_{P\mathcal{A}}}
}
\]
since $\lambda_{\mathcal{A}}$ is $T_{P}$-cocontinuous by Lemma \ref{lambdaccts},
giving us our isomorphism of left extensions $\omega_{4}^{\mathcal{A}}$.
Note that this means $\omega_{4}$ satisfies coherence axiom 7 of
\cite{marm1999}.
\end{proof}
In the following proposition we show that the admissible maps are
preserved. Note that the proof relies on the existence of isomorphisms
$\omega_{4}^{\mathcal{A}}$ as above, which in turn relies on the
the admissibility of $y_{\mathcal{A}}$ being preserved (also shown
above).
\begin{prop}
\label{preserveadm} Suppose we are given a 2-category $\mathscr{C}$
equipped with a pseudomonad $\left(T,u,m\right)$ and a KZ pseudomonad
$\left(P,y,\mu\right)$. Suppose further that we are given a pseudo-distributive
law over a KZ pseudomonad $\lambda\colon TP\to PT$. Then $T$ preserves
$P$-admissible maps.
\end{prop}

\begin{proof}
Suppose we are given a 1-cell $L\colon\mathcal{A}\to\mathcal{B}$
which is $P$-admissible, meaning that we have an adjunction $PL\dashv\textnormal{res}_{L}$
with unit and counit denoted $\eta$ and $\epsilon$ respectively.
We show that $TL\colon T\mathcal{A}\to T\mathcal{B}$ must then be
$P$-admissible, with the admissibility exhibited by an adjunction
\[
PTL\dashv\mu_{T\mathcal{A}}\cdot P\lambda_{\mathcal{A}}\cdot PT\textnormal{res}_{L}\cdot PTy_{\mathcal{B}}
\]
Firstly, we note that this right adjoint is exhibited as the left
extension
\begin{equation}
\xymatrix@=1em{PT\mathcal{B}\ar[rr]^{PTy_{\mathcal{B}}}\ar@{}[rdrd]|-{\stackrel{y_{Ty_{\mathcal{B}}}^{-1}}{\Longleftarrow}} &  & PTP\mathcal{B}\ar[rr]^{PT\textnormal{res}_{L}}\ar@{}[rdrd]|-{\stackrel{y_{T\textnormal{res}_{L}}^{-1}}{\Longleftarrow}} &  & PTP\mathcal{A}\ar[rr]^{P\lambda_{\mathcal{A}}}\ar@{}[rdrd]|-{\stackrel{y_{\lambda_{\mathcal{A}}}^{-1}}{\Longleftarrow}} &  & P^{2}T\mathcal{A}\ar[rr]^{\mu_{T\mathcal{A}}}\ar@{}[rd]|-{\cong} &  & PT\mathcal{A}\\
 &  &  &  &  &  &  & \;\\
T\mathcal{B}\ar[uu]^{y_{T\mathcal{B}}}\ar[rr]_{Ty_{\mathcal{B}}} &  & TP\mathcal{B}\ar[uu]|-{y_{TP\mathcal{B}}}\ar[rr]_{T\textnormal{res}_{L}} &  & TP\mathcal{A}\ar[rr]_{\lambda_{\mathcal{A}}}\ar[uu]|-{y_{TP\mathcal{A}}} &  & PT\mathcal{A}\ar[uu]|-{y_{PT\mathcal{A}}}\ar[rruu]_{\textnormal{id}_{PT\mathcal{A}}}
}
\label{defrightadj2}
\end{equation}
and denote it $\mathbf{R}_{L}$ for convenience. We then define our
unit $n$ as the unique 2-cell rendering
\[
\xymatrix@=1em{ &  &  & PT\mathcal{B}\ar[rd]^{\mathbf{R}_{L}}\\
T\mathcal{A}\ar[rr]_{y_{T\mathcal{A}}} &  & PT\mathcal{A}\ar[rr]_{\textnormal{id}_{PT\mathcal{A}}}\ar[ru]^{PTL} & \;\ar@{}[u]|-{\Uparrow n} & PT\mathcal{A}
}
\]
equal to
\[
\xymatrix@=1em{PT\mathcal{A}\ar[rr]^{PTL}\ar@{}[rdrd]|-{\stackrel{y_{TL}^{-1}}{\Longleftarrow}} &  & PT\mathcal{B}\ar[rr]^{PTy_{\mathcal{B}}}\ar@{}[rdrd]|-{\stackrel{y_{Ty_{\mathcal{B}}}^{-1}}{\Longleftarrow}} &  & PTP\mathcal{B}\ar[rr]^{PT\textnormal{res}_{L}}\ar@{}[rdrd]|-{\stackrel{y_{T\textnormal{res}_{L}}^{-1}}{\Longleftarrow}} &  & PTP\mathcal{A}\ar[rr]^{P\lambda_{\mathcal{A}}}\ar@{}[rdrd]|-{\stackrel{y_{\lambda_{\mathcal{A}}}^{-1}}{\Longleftarrow}} &  & P^{2}T\mathcal{A}\ar[rr]^{\mu_{T\mathcal{A}}}\ar@{}[rd]|-{\cong} &  & PT\mathcal{A}\\
 &  &  &  &  &  &  &  &  & \;\\
T\mathcal{A}\ar[rr]_{TL}\ar[uu]^{y_{T\mathcal{A}}}\ar@/_{0.7pc}/[rrrrd]|-{Ty_{\mathcal{A}}}\ar@/_{3.5pc}/[rrrrrrrr]_{y_{T\mathcal{A}}} &  & T\mathcal{B}\ar[uu]|-{y_{T\mathcal{B}}}\ar[rr]_{Ty_{\mathcal{B}}} & \; & TP\mathcal{B}\ar[uu]|-{y_{TP\mathcal{B}}}\ar[rr]_{T\textnormal{res}_{L}} & \ar@{}[d]|-{\Uparrow T\eta\quad\quad} & TP\mathcal{A}\ar[rr]_{\lambda_{\mathcal{A}}}\ar[uu]|-{y_{TP\mathcal{A}}} &  & PT\mathcal{A}\ar[uu]|-{y_{PT\mathcal{A}}}\ar[rruu]_{\textnormal{id}_{PT\mathcal{A}}}\\
 &  & \ar@{}[ru]|-{\;\;\Uparrow Ty_{L}} &  & TP\mathcal{A}\ar[u]^{TPL}\ar@/_{0.5pc}/[rru]_{T\textnormal{id}} & \;\\
 &  &  &  & \ar@{}[u]|-{\Uparrow\omega_{2}^{\mathcal{A}}}
}
\]
and note that the unit $n$ is then given by
\begin{equation}
\xymatrix@=1em{PT\mathcal{A}\ar[rr]^{PTL}\ar@/_{1pc}/[rrrrd]|-{PTy_{\mathcal{A}}}\ar@/_{4pc}/[rrrrrrrr]|-{Py_{T\mathcal{A}}}\ar@/_{5pc}/[rrrrrrrrrr]_{\textnormal{id}_{PT\mathcal{A}}} &  & PT\mathcal{B}\ar[rr]^{PTy_{\mathcal{B}}} &  & PTP\mathcal{B}\ar[rr]^{PT\textnormal{res}_{L}}\ar@{}[rrd]|-{\Uparrow PT\eta\quad\quad} &  & PTP\mathcal{A}\ar[rr]^{P\lambda_{\mathcal{A}}} &  & P^{2}T\mathcal{A}\ar[rr]^{\mu_{T\mathcal{A}}} &  & PT\mathcal{A}\\
 &  &  & \ar@{}[ul]|-{\Uparrow PTy_{L}} & PTP\mathcal{A}\ar[u]^{PTPL}\ar@/_{0.7pc}/[rru]_{PT\textnormal{id}} &  & \; & \;\\
 &  &  &  & \ar@{}[u]|-{\Uparrow P\omega_{2}^{\mathcal{A}}} &  & \;\\
 &  &  &  &  & \; & \ar@{}[uur]|-{\cong}
}
\label{nformula}
\end{equation}
as $y\colon1\to P$ is a pseudonatural transformation. We define our
counit $e$ as the unique solution to the left extension problem
\[
\xymatrix@=1em{ &  &  &  & \;\ar@{}[d]|-{\Uparrow e} &  &  &  &  &  & \ar@{}[d]|-{\Uparrow\left(\omega_{2}^{\mathcal{B}}\right)^{-1}}\\
T\mathcal{B}\ar[rr]^{y_{T\mathcal{B}}}\ar@/_{2pc}/[rrrr]_{\lambda_{\mathcal{A}}\cdot T\textnormal{res}_{L}\cdot Ty_{\mathcal{B}}} &  & PT\mathcal{B}\ar[rr]^{\mathbf{R}_{L}}\ar@/^{1.8pc}/[rrrr]^{\textnormal{id}_{PT\mathcal{B}}} &  & PT\mathcal{A}\ar[rr]^{PTL} &  & PT\mathcal{B} & = & T\mathcal{B}\ar@/_{0pc}/[dd]_{Ty_{\mathcal{B}}}\ar@/^{0pc}/[rr]^{Ty_{\mathcal{B}}}\ar@/^{1.8pc}/[rrrr]^{y_{T\mathcal{B}}} &  & TP\mathcal{B}\ar[rr]^{\lambda_{\mathcal{B}}} &  & PT\mathcal{B}\\
 &  & \;\ar@{}[u]|-{\cong} &  &  &  &  &  & \ar@{}[r]|-{=} & \ar@{}[rd]|-{\Uparrow T\epsilon} &  & \ar@{}[]|-{\Uparrow\lambda_{L}}\\
 &  &  &  &  &  &  &  & TP\mathcal{B}\ar[rruu]|-{T\textnormal{id}}\ar@/_{0pc}/[rr]_{T\textnormal{res}_{L}} &  & TP\mathcal{A}\ar[uu]|-{TPL}\ar[rr]_{\lambda_{\mathcal{A}}} &  & PT\mathcal{A}\ar[uu]_{PTL}
}
\]
where the unlabeled isomorphism is \ref{defrightadj2}. Of the triangle
identities:
\[
\xymatrix@=1em{PTL\ar[ddrrr]_{\textnormal{id}_{PTL}}\myar{PTL\cdot n}{rrr} &  &  & PTL\cdot\mathbf{R}_{L}\cdot PTL\ar[dd]^{e\cdot PTL} &  & \mathbf{R}_{L}\ar[ddrrr]_{\textnormal{id}_{\mathbf{R}_{L}}}\myar{n\cdot\mathbf{R}_{L}}{rrr} &  &  & \mathbf{R}_{L}\cdot PTL\cdot\mathbf{R}_{L}\ar[dd]^{\mathbf{R}_{L}\cdot e}\\
\\
 &  &  & PTL &  &  &  &  & \mathbf{R}_{L}
}
\]
the left identity (or equivalently the left triangle identity whiskered
by $y_{T\mathcal{A}}$) easily follows from the whiskered definitions
of $n$ and $e$ as well as the corresponding triangle identity for
$PL\dashv\textnormal{res}_{L}$, and $\omega_{2}$ being a modification.
The right triangle identity (or that whiskered by $y_{T\mathcal{B}}$)
is more complicated. This identity amounts to checking that
\begin{equation}
\xymatrix@=1em{ &  & \ar@{}[d]|-{\Uparrow\left(\omega_{2}^{\mathcal{B}}\right)^{-1}}\\
T\mathcal{B}\ar@/_{0pc}/[dd]_{Ty_{\mathcal{B}}}\ar@/^{0pc}/[rr]^{Ty_{\mathcal{B}}}\ar@/^{1.8pc}/[rrrr]^{y_{T\mathcal{B}}} &  & TP\mathcal{B}\ar[rr]^{\lambda_{\mathcal{B}}} &  & PT\mathcal{B}\ar[rr]^{PTy_{\mathcal{B}}}\ar@{}[rrdd]|-{\Uparrow PTy_{L}} &  & PTP\mathcal{B}\ar[rr]^{PT\textnormal{res}_{L}} &  & PTP\mathcal{A}\ar[rr]^{P\lambda_{\mathcal{A}}} &  & P^{2}T\mathcal{A}\ar[rr]^{\mu_{T\mathcal{A}}} &  & PT\mathcal{A}\\
 & \ar@{}[rd]|-{\Uparrow T\epsilon}\ar@{}[lu]|-{=} &  & \ar@{}[]|-{\Uparrow\lambda_{L}} &  &  &  & \;\ar@{}[ul]|-{\Uparrow PT\eta} &  & \ar@{}[rd]|-{\cong\quad}\\
TP\mathcal{B}\ar[rruu]|-{T\textnormal{id}}\ar@/_{0pc}/[rr]_{T\textnormal{res}_{L}} &  & TP\mathcal{A}\ar[uu]|-{TPL}\ar[rr]_{\lambda_{\mathcal{A}}} &  & PT\mathcal{A}\ar[uu]|-{PTL}\ar[rr]^{PTy_{\mathcal{A}}}\ar@/_{2pc}/[rrrrrruu]|-{Py_{T\mathcal{A}}}\ar@/_{2.5pc}/[rrrrrrrruu]_{\textnormal{id}_{PT\mathcal{A}}} &  & PTP\mathcal{A}\ar[uu]|-{PTPL}\ar[uurr]|-{PT\textnormal{id}} & \; & \ar@{}[lu]|-{\Uparrow P\omega_{2}^{\mathcal{A}}} &  & \;
}
\label{diag1}
\end{equation}
is just the isomorphism \ref{defrightadj2}. The first step here is
to make our diagrams more like the first axiom of a pseudo-distributive
law over a KZ doctrine. Upon using that $\omega_{2}$ is a modification
and the coherence axiom \ref{kzcoh2}, the problem reduces to showing
that \ref{diag1} with the unnamed isomorphism and cell $\omega_{2}^{\mathcal{B}}$
removed is equal to
\begin{equation}
\xymatrix@=1em{ & PT\mathcal{B}\ar[rr]^{PTy_{\mathcal{B}}}\ar@{}[rd]|-{\Uparrow\lambda_{y_{\mathcal{B}}}^{-1}} &  & PTP\mathcal{B}\ar[rr]^{PT\textnormal{res}_{L}}\ar@{}[rd]|-{\Uparrow\lambda_{\textnormal{res}_{L}}^{-1}} &  & PTP\mathcal{A}\ar[rr]^{P\lambda_{\mathcal{A}}}\ar@{}[dr]|-{\overset{y_{\lambda_{\mathcal{A}}}^{-1}}{\Longleftarrow}} &  & P^{2}T\mathcal{A}\ar[rr]^{\mu_{T\mathcal{A}}} &  & PT\mathcal{A}\\
TP\mathcal{B}\ar[rr]^{TPy_{\mathcal{B}}}\ar[ur]^{\lambda_{\mathcal{B}}} &  & TP^{2}\mathcal{B}\ar[rr]^{TP\textnormal{res}_{L}}\ar[ur]^{\lambda_{P\mathcal{B}}} &  & TP^{2}\mathcal{A}\ar@{}[r]|-{\overset{\;\omega_{2}^{P\mathcal{A}}}{\Longleftarrow}\;\;\;}\ar[ur]^{\lambda_{P\mathcal{A}}} & \; & \; & \;\ar@{}[]|-{\overset{\theta^{T\mathcal{A}}}{\Longleftarrow}}\\
 & T\mathcal{B}\ar[rr]_{Ty_{\mathcal{B}}}\ar[ul]^{Ty_{\mathcal{B}}}\ar@{}[ru]|-{\Uparrow Ty_{y_{\mathcal{B}}}^{-1}} &  & TP\mathcal{B}\ar[rr]_{T\textnormal{res}_{L}}\ar[ul]^{Ty_{P\mathcal{B}}}\ar@{}[ru]|-{\Uparrow Ty_{\textnormal{res}_{L}}^{-1}} &  & TP\mathcal{A}\ar[rr]_{\lambda_{\mathcal{A}}}\ar[uu]|-{y_{TP\mathcal{A}}}\ar[ul]^{Ty_{P\mathcal{A}}} &  & PT\mathcal{A}\ar@/^{1pc}/[uu]^{y_{PT\mathcal{A}}}\ar@/_{1pc}/[uu]_{Py_{T\mathcal{A}}}
}
\label{diag2}
\end{equation}
We then simplify \ref{diag2} using the first axiom of a pseudo-distributive
law over a KZ doctrine, canceling $P\omega_{2}^{\mathcal{A}}$, and
pasting $\lambda_{y_{\mathcal{B}}}$, $\lambda_{y_{\mathcal{A}}}$
and $\lambda_{\textnormal{res}_{L}}$ to the other side of the desired
equation. By pseudonaturality of $\lambda$, the problem may then
be reduced to showing that 
\[
\xymatrix@=1em{T\mathcal{B}\ar@/_{0pc}/[dd]_{Ty_{\mathcal{B}}}\ar@/^{0pc}/[rr]^{Ty_{\mathcal{B}}} &  & TP\mathcal{B}\ar[rr]^{TPy_{\mathcal{B}}} &  & TP^{2}\mathcal{B}\ar[rr]^{TP\textnormal{res}_{L}} &  & TP^{2}\mathcal{A}\ar[rr]^{\lambda_{P\mathcal{A}}} &  & PTP\mathcal{A}\ar[rr]^{P\lambda_{\mathcal{A}}} &  & P^{2}T\mathcal{A}\ar[rr]^{\mu_{T\mathcal{A}}} &  & PT\mathcal{A}\\
 & \ar@{}[rd]|-{\Uparrow T\epsilon}\ar@{}[lu]|-{=} &  & \ar@{}[]|-{\Uparrow TPy_{L}} &  & \ar@{}[ul]|-{\;\;\Uparrow TP\eta}\\
TP\mathcal{B}\ar[rruu]|-{T\textnormal{id}}\ar@/_{0pc}/[rr]_{T\textnormal{res}_{L}} &  & TP\mathcal{A}\ar[uu]|-{TPL}\ar[rr]_{TPy_{\mathcal{A}}} &  & TP^{2}\mathcal{A}\ar[rruu]_{TP\textnormal{id}}\ar[uu]|-{TP^{2}L}
}
\]
is equal to
\[
\xymatrix@=1em{TP\mathcal{B}\ar[rr]^{TPy_{\mathcal{B}}} &  & TP^{2}\mathcal{B}\ar[rr]^{TP\textnormal{res}_{L}} &  & TP^{2}\mathcal{A}\ar[rr]^{\lambda_{P\mathcal{A}}} &  & PTP\mathcal{A}\ar[rr]^{P\lambda_{\mathcal{A}}} &  & P^{2}T\mathcal{A}\ar[rr]^{\mu_{T\mathcal{A}}} &  & PT\mathcal{A}\\
\; & \ar@{}[]|-{\Uparrow Ty_{y_{\mathcal{B}}}^{-1}} & \; & \ar@{}[ld]|-{\Uparrow Ty_{\textnormal{res}_{L}}^{-1}} & \;\ar@{}[]|-{\overset{T\theta^{\mathcal{A}}}{\Longleftarrow}}\\
T\mathcal{B}\ar[rr]_{Ty_{\mathcal{B}}}\ar[uu]^{Ty_{\mathcal{B}}} &  & TP\mathcal{B}\ar[rr]_{T\textnormal{res}_{L}}\ar[uu]|-{Ty_{P\mathcal{B}}} &  & TP\mathcal{A}\ar@/^{0.8pc}/[uu]^{Ty_{P\mathcal{A}}}\ar@/_{0.8pc}/[uu]_{TPy_{\mathcal{A}}}
}
\]
Since we have the isomorphism $\omega_{4}^{\mathcal{A}}$ as in Lemma
\ref{w2isext}, and as $T\theta^{\mathcal{B}}\cdot Ty_{\mathcal{B}}$
is invertible (meaning we can paste with $T\theta^{\mathcal{B}}$
and maintain the logical equivalence), we may reduce the problem to
showing that
\[
\xymatrix@=1em{T\mathcal{B}\ar@/_{0pc}/[dd]_{Ty_{\mathcal{B}}}\ar@/^{0pc}/[rr]^{Ty_{\mathcal{B}}} &  & TP\mathcal{B}\ar@/_{0.5pc}/[rr]_{TPy_{\mathcal{B}}}\ar@/^{0.5pc}/[rr]^{Ty_{P\mathcal{B}}} & \ar@{}[]|-{\Uparrow T\theta^{\mathcal{B}}} & TP^{2}\mathcal{B}\ar[rr]^{TP\textnormal{res}_{L}} &  & TP^{2}\mathcal{A}\ar[rr]^{T\mu_{\mathcal{A}}} &  & TP\mathcal{A}\ar[rr]^{\lambda_{\mathcal{A}}} &  & PT\mathcal{A}\\
 & \ar@{}[rd]|-{\Uparrow T\epsilon}\ar@{}[lu]|-{=} &  & \ar@{}[]|-{\stackrel{\;}{\Uparrow TPy_{L}}} &  & \ar@{}[ul]|-{\;\;\Uparrow TP\eta}\\
TP\mathcal{B}\ar[rruu]|-{T\textnormal{id}}\ar@/_{0pc}/[rr]_{T\textnormal{res}_{L}} &  & TP\mathcal{A}\ar[uu]|-{TPL}\ar[rr]_{TPy_{\mathcal{A}}} & \; & TP^{2}\mathcal{A}\ar[rruu]_{TP\textnormal{id}_{P\mathcal{A}}}\ar[uu]|-{TP^{2}L}
}
\]
is equal to 
\[
\xymatrix@=1em{TP\mathcal{B}\ar@/_{0.5pc}/[rr]_{TPy_{\mathcal{B}}}\ar@/^{0.5pc}/[rr]^{Ty_{P\mathcal{B}}} & \ar@{}[]|-{\Uparrow T\theta^{\mathcal{B}}} & TP^{2}\mathcal{B}\ar[rr]^{TP\textnormal{res}_{L}} &  & TP^{2}\mathcal{A}\ar[rr]^{T\mu_{\mathcal{A}}} &  & TP\mathcal{A}\ar[rr]^{\lambda_{\mathcal{A}}} &  & PT\mathcal{A}\\
\; & \ar@{}[d]|-{\underset{\;}{\Uparrow Ty_{y_{\mathcal{B}}}}} & \; & \ar@{}[d]|-{\Uparrow Ty_{\textnormal{res}_{L}}^{-1}} & \;\ar@{}[]|-{\overset{T\theta^{\mathcal{A}}}{\Longleftarrow}}\\
T\mathcal{B}\ar[rr]_{Ty_{\mathcal{B}}}\ar[uu]^{Ty_{\mathcal{B}}} & \; & TP\mathcal{B}\ar[rr]_{T\textnormal{res}_{L}}\ar[uu]|-{Ty_{P\mathcal{B}}} & \; & TP\mathcal{A}\ar@/^{0.8pc}/[uu]^{Ty_{P\mathcal{A}}}\ar@/_{0.8pc}/[uu]_{TPy_{\mathcal{A}}}
}
\]
From here, use that $\theta$ is a modification, the axioms \ref{kzcoh1}
and \ref{kzcoh2}, and pseudonaturality of $y$ to deduce the triangle
identity from that of the adjunction $PL\dashv\textnormal{res}_{L}$.
\end{proof}
\begin{rem}
Note that here, as well as in the preceding lemma, we only used that
$\omega_{2}$ is an invertible modification and the first axiom for
a pseudo-distributive law over a KZ doctrine, along with pseudo naturality
of $\lambda$.
\end{rem}

We are now ready to prove the main result of this subsection.
\begin{thm}
\label{eimpliesd} In the statement of Theorem \ref{liftkzequiv}
(e) implies (d).
\end{thm}

\begin{proof}
We first note by Proposition \ref{preserveadm} that $T$ preserves
\emph{$P$-}admissible maps. Also, we know by Lemma \ref{w2isext}
that each $\lambda_{\mathcal{A}}$ is a left extension exhibited by
the distributive law data as in
\[
\xymatrix@=1em{TP\mathcal{A}\ar[rr]^{\lambda_{\mathcal{A}}} &  & PT\mathcal{A}\ar@{}[dl]|-{\stackrel{\omega_{2}^{\mathcal{A}}}{\Longleftarrow}}\\
 & \;\\
 &  & T\mathcal{A}\ar[uu]_{y_{T\mathcal{A}}}\ar[uull]^{Ty_{\mathcal{A}}}
}
\]
with $\omega_{2}^{\mathcal{A}}$ invertible by assumption. That each
$\lambda_{\mathcal{A}}$ is $T_{P}$-cocontinuous is a consequence
of Lemma \ref{lambdaccts} and $\omega_{2}$ being a modification.
Finally, that the diagrams
\[
\xymatrix@=1em{P\mathcal{A}\ar[rr]^{u_{P\mathcal{A}}} &  & TP\mathcal{A}\ar@{}[dldl]|-{\stackrel{u_{y_{\mathcal{A}}}}{\Longleftarrow}}\ar[rr]^{\lambda_{\mathcal{A}}}\ar@{}[dr]|-{\stackrel{\omega_{2}^{\mathcal{A}}}{\Longleftarrow}} &  & PT\mathcal{A} &  & T^{2}P\mathcal{A}\ar[rr]^{m_{P\mathcal{A}}} &  & TP\mathcal{A}\ar@{}[dldl]|-{\stackrel{m_{y_{\mathcal{A}}}}{\Longleftarrow}}\ar[rr]^{\lambda_{\mathcal{A}}}\ar@{}[dr]|-{\stackrel{\omega_{2}^{\mathcal{A}}}{\Longleftarrow}} &  & PT\mathcal{A}\\
 &  &  & \; &  &  &  &  &  & \;\\
\mathcal{A}\ar[rr]_{u_{\mathcal{A}}}\ar[uu]^{y_{\mathcal{A}}} &  & T\mathcal{A}\ar[uu]|-{Ty_{\mathcal{A}}}\ar[uurr]_{y_{T\mathcal{A}}} &  &  &  & T^{2}\mathcal{A}\ar[rr]_{m_{\mathcal{A}}}\ar[uu]^{T^{2}y_{\mathcal{A}}} &  & T\mathcal{A}\ar[uu]|-{Ty_{\mathcal{A}}}\ar[uurr]_{y_{T\mathcal{A}}}
}
\]
exhibit both $\lambda_{\mathcal{A}}\cdot u_{P\mathcal{A}}$ and $\lambda_{\mathcal{A}}\cdot m_{P\mathcal{A}}$
as left extensions is due to the last two axioms for a pseudo-distributive
law over a KZ pseudomonad (as pasting a left extension with an isomorphism
$\omega_{1}$ or $\omega_{3}$ will preserve the left extension property).
Indeed, it is clear the left diagram below exhibits $Pu_{\mathcal{A}}$
as a left extension.
\[
\xymatrix@=1em{P\mathcal{A}\ar[rr]^{Pu_{\mathcal{A}}}\ar@{}[rrdd]|-{\stackrel{\;y_{u_{\mathcal{A}}}^{-1}}{\Longleftarrow}} &  & PT\mathcal{A} &  &  & T^{2}P\mathcal{A}\ar[rr]^{T\lambda_{\mathcal{A}}} &  & TPT\mathcal{A}\ar@{}[dl]|-{\stackrel{T\omega_{2}^{\mathcal{A}}}{\Longleftarrow}}\ar[rr]^{\lambda_{T\mathcal{A}}}\ar@{}[dr]|-{\stackrel{\;\;\omega_{2}T^{\mathcal{A}}}{\Longleftarrow}} &  & PT^{2}\mathcal{A}\ar[rr]^{Pm_{\mathcal{A}}}\ar@{}[dd]|-{\stackrel{\;y_{m_{\mathcal{A}}}^{-1}}{\Longleftarrow}} &  & PT\mathcal{A}\\
 &  &  &  &  &  & \; &  & \;\\
\mathcal{A}\ar[rr]_{u_{\mathcal{A}}}\ar[uu]^{y_{\mathcal{A}}} &  & T\mathcal{A}\ar[uu]_{y_{T\mathcal{A}}} &  &  &  &  & T^{2}\mathcal{A}\ar[uu]|-{Ty_{T\mathcal{A}}}\ar[uull]^{T^{2}y_{\mathcal{A}}}\ar[uurr]_{y_{T^{2}\mathcal{A}}}\ar[rr]_{m_{\mathcal{A}}} &  & T\mathcal{A}\ar[uurr]_{y_{T\mathcal{A}}}
}
A
\]
To see that the composite $Pm_{\mathcal{A}}\cdot\lambda_{T\mathcal{A}}\cdot T\lambda_{\mathcal{A}}$
on the right is a left extension, note that Proposition \ref{lambdabeck}
shows $\lambda_{T\mathcal{A}}\cdot T\lambda_{\mathcal{A}}$ is a left
extension above, and since $T^{2}y_{\mathcal{A}}$ is \emph{$P$-}admissible
by Proposition \ref{preserveadm}, the left extension property is
respected upon whiskering by $Pm_{\mathcal{A}}$. 
\end{proof}

\subsection{Lifting a KZ Doctrine to Algebras via a Distributive Law}

In this subsection we show that given a pseudo-distributive law of
a pseudomonad $T$ over a KZ doctrine $P$, we may lift $P$ to a
KZ doctrine $\widetilde{P}$ on the 2-category of pseudo $T$-algebras.
This is $\left(d\right)\Longrightarrow\left(a\right)$ of Theorem
\ref{liftkzequiv}. However, before we show this implication we will
first need to verify the following proposition.
\begin{prop}
\label{claim} Suppose we are given statement $\left(d\right)$ of
Theorem \ref{liftkzequiv}. It then follows that:
\begin{enumerate}
\item $T$ preserves $P$-admissible maps;
\end{enumerate}
\noindent and for every pseudo $T$-algebra $\left(\mathcal{A},T\mathcal{A}\overset{x}{\rightarrow}\mathcal{A}\right)$,
\begin{enumerate}\setcounter{enumi}{1}

\item there exists a 1-cell $z_{x}$ given as the left extension
via an isomorphism $\xi_{x}$ 
\[
\xymatrix{TP\mathcal{A}\ar[r]^{z_{x}} & P\mathcal{A}\\
T\mathcal{A}\ar[r]_{x}\ar[u]^{Ty_{\mathcal{A}}}\ar@{}[ur]|-{\Uparrow\xi_{x}} & \mathcal{A}\ar[u]_{y_{\mathcal{A}}}
}
\]
which we call the Day convolution at $x$;

\item each $z_{x}$ is $T_{P}$-cocontinuous;

\item the respective diagrams
\[
\xymatrix@=1em{P\mathcal{A}\ar[rr]^{u_{P\mathcal{A}}} &  & TP\mathcal{A}\ar@{}[dldl]|-{\Uparrow u_{y_{\mathcal{A}}}}\ar[rr]^{z_{x}}\ar@{}[ddrr]|-{\Uparrow\xi_{x}} &  & P\mathcal{A} &  & T^{2}P\mathcal{A}\ar[rr]^{m_{P\mathcal{A}}} &  & TP\mathcal{A}\ar@{}[dldl]|-{\Uparrow m_{y_{\mathcal{A}}}}\ar[rr]^{z_{x}}\ar@{}[ddrr]|-{\Uparrow\xi_{x}} &  & P\mathcal{A}\\
 &  &  & \; &  &  &  &  &  & \;\\
\mathcal{A}\ar[rr]_{u_{\mathcal{A}}}\ar[uu]^{y_{\mathcal{A}}} &  & T\mathcal{A}\ar[uu]|-{Ty_{\mathcal{A}}}\ar[rr]_{x} &  & \mathcal{A}\ar[uu]_{y_{\mathcal{A}}} &  & T^{2}\mathcal{A}\ar[rr]_{m_{\mathcal{A}}}\ar[uu]^{T^{2}y_{\mathcal{A}}} &  & T\mathcal{A}\ar[uu]|-{Ty_{\mathcal{A}}}\ar[rr]_{x} &  & \mathcal{A}\ar[uu]_{y_{\mathcal{A}}}
}
\]
exhibit $z_{x}\cdot u_{P\mathcal{A}}$ and $z_{x}\cdot m_{P\mathcal{A}}$
as left extensions.

\end{enumerate}

\end{prop}

\begin{proof}
(1) This property is straight from the definition. We include this
property here so that this proposition may be taken as one the equivalent
conditions of Theorem \ref{liftkzequiv}. We will remark about this
later in this subsection. Now, let a pseudo $T$-algebra $\left(\mathcal{A},T\mathcal{A}\overset{x}{\rightarrow}\mathcal{A}\right)$
be given. (2) The left extension $\left(z_{x},\xi_{x}\right)$ is
given by the diagram
\[
\xymatrix@=1em{TP\mathcal{A}\ar[rr]^{\lambda_{\mathcal{A}}} &  & PT\mathcal{A}\ar@{}[dl]|-{\stackrel{\omega_{2}^{\mathcal{A}}}{\Longleftarrow}}\ar[rr]^{Px}\ar@{}[drdr]|-{\stackrel{y_{x}^{-1}}{\Longleftarrow}} &  & P\mathcal{A}\\
 & \;\\
 &  & T\mathcal{A}\ar[uu]|-{y_{T\mathcal{A}}}\ar[uull]^{Ty_{\mathcal{A}}}\ar[rr]_{x} &  & \mathcal{A}\ar[uu]_{y_{\mathcal{A}}}
}
\]
where the left extension $\lambda_{\mathcal{A}}$ is preserved by
$Px$ as $Ty_{\mathcal{A}}$ is \emph{$P$-}admissible. (3) Suppose
we are given a left extension as on the left below.
\[
\xymatrix@=1em{P\mathcal{D}\ar[rr]^{\overline{F}} &  & P\mathcal{A} &  &  &  & TP\mathcal{D}\ar[rr]^{T\overline{F}} &  & TP\mathcal{A}\ar[rr]^{z_{x}} &  & P\mathcal{A}\\
 & \ar@{}[ul]|-{\overset{c_{F}}{\Longleftarrow}} &  &  &  &  &  & \ar@{}[ul]|-{\overset{Tc_{F}}{\Longleftarrow}}\\
\mathcal{D}\ar[uu]^{y_{\mathcal{D}}}\ar[uurr]_{F} &  &  &  &  &  & T\mathcal{D}\ar[uu]^{Ty_{\mathcal{D}}}\ar[uurr]_{TF}
}
\]
As this left extension is $T$-preserved by $\lambda_{\mathcal{A}}$,
which in turn is preserved by $Px$ as $Ty_{\mathcal{D}}$ is $P$-admissible,
the diagram on the right exhibits $z_{x}\cdot T\overline{F}=Px\cdot\lambda_{\mathcal{A}}\cdot T\overline{F}$
as a left extension. (4) Again noting each $Ty_{\mathcal{A}}$ is
\emph{$P$-}admissible, we see the left extensions
\[
\xymatrix@=1em{P\mathcal{A}\ar[rr]^{u_{P\mathcal{A}}} &  & TP\mathcal{A}\ar@{}[dldl]|-{\overset{u_{y}^{\mathcal{A}}}{\Longleftarrow}}\ar[rr]^{\lambda_{\mathcal{A}}}\ar@{}[dr]|-{\overset{\omega_{2}^{\mathcal{A}}}{\Longleftarrow}} &  & PT\mathcal{A} &  & T^{2}P\mathcal{A}\ar[rr]^{m_{P\mathcal{A}}} &  & TP\mathcal{A}\ar@{}[dldl]|-{\overset{m_{y}^{\mathcal{A}}}{\Longleftarrow}}\ar[rr]^{\lambda_{\mathcal{A}}}\ar@{}[dr]|-{\overset{\omega_{2}^{\mathcal{A}}}{\Longleftarrow}} &  & PT\mathcal{A}\\
 &  &  & \; &  &  &  &  &  & \;\\
\mathcal{A}\ar[rr]_{u_{\mathcal{A}}}\ar[uu]^{y_{\mathcal{A}}} &  & T\mathcal{A}\ar[uu]|-{Ty_{\mathcal{A}}}\ar[uurr]_{y_{T\mathcal{A}}} &  &  &  & T^{2}\mathcal{A}\ar[rr]_{m_{\mathcal{A}}}\ar[uu]^{T^{2}y_{\mathcal{A}}} &  & T\mathcal{A}\ar[uu]|-{Ty_{\mathcal{A}}}\ar[uurr]_{y_{T\mathcal{A}}}
}
\]
are preserved upon composing with $Px$. Trivially, these left extensions
are then preserved upon pasting with the isomorphism $y_{x}$.
\end{proof}
The following remark is not needed for the proof of Theorem \ref{liftkzequiv},
it merely explains why the consequences in the above proposition are
equivalent to the conditions (\emph{a}) through to (\emph{f}) of this
theorem.
\begin{rem}
Note that from this proposition one may recover statement \emph{(d)}
of Theorem \ref{liftkzequiv}. This is since given the data of this
proposition, one may recover a choice of each $\lambda_{\mathcal{A}}$
and its exhibiting invertible 2-cell $\omega_{2}^{\mathcal{A}}$ as
a left extension, by taking the pasting
\[
\xymatrix@=1em{TP\mathcal{A}\ar@{}[ddrr]|-{\Uparrow Ty_{u_{\mathcal{A}}}^{-1}}\myar{TPu_{\mathcal{A}}}{rr} &  & TPT\mathcal{A}\ar[rr]^{z_{m_{\mathcal{A}}}}\ar@{}[ddrr]|-{\Uparrow\xi_{m_{\mathcal{A}}}} &  & PT\mathcal{A}\\
\\
T\mathcal{A}\ar[uu]^{Ty_{\mathcal{A}}}\ar[rr]_{Tu_{\mathcal{A}}} &  & T^{2}\mathcal{A}\ar[uu]|-{Ty_{T\mathcal{A}}}\ar[rr]_{m_{\mathcal{A}}} &  & T\mathcal{A}\ar[uu]_{y_{T\mathcal{A}}}
}
\]

The condition of each $\lambda_{\mathcal{A}}$ being $T_{P}$-cocontinuous
is inherited from the corresponding condition on each $z_{m_{\mathcal{A}}}$.
Condition (4) of this proposition yields the corresponding conditions
on the maps $\lambda_{\mathcal{A}}$. We omit this last calculation,
as it is not required for the proof of the main theorem. We just note
that this last calculation relies on the pseudo-algebra structure
of the maps $z_{x}\colon TP\mathcal{A}\to P\mathcal{A}$ constructed
later on in this subsection. The construction of the algebra structure
may be done with all of the axioms for a pseudo-distributive law over
a KZ doctrine without the last (which we have recovered from the proposition),
in addition to the last condition of the proposition.
\end{rem}

The following proposition will be useful in the proof that $\left(d\right)$
implies $\left(a\right)$.
\begin{prop}
\label{admcocontinuousequiv}Suppose we are given a 2-category $\mathscr{C}$
equipped with a pseudomonad $\left(T,u,m\right)$ and a KZ doctrine
$\left(P,y\right)$. Further suppose that we are given a pseudo-distributive
law over a KZ doctrine $\lambda\colon TP\to PT$. Then for any two
$P$-cocomplete objects $\mathcal{C}$ and $\mathcal{D}$, a 1-cell
$u\colon T\mathcal{C}\to\mathcal{D}$ is $T_{P}$-cocontinuous if
and only if it is $T_{P}$-adm-cocontinuous. 
\end{prop}

\begin{proof}
Supposing that $u$ is $T_{P}$-cocontinuous we check that $u$ is
necessarily $T_{P}$-adm-cocontinuous. To see this, we first note
that we have an induced isomorphism of left extensions as a consequence
of having the two left extensions
\[
\xymatrix@=1em{TP\mathcal{C}\ar[rr]^{\lambda_{\mathcal{C}}} &  & PT\mathcal{C}\ar[rr]^{Pu} &  & P\mathcal{D}\ar[rr]^{\left(y_{\mathcal{D}}\right)_{\ast}} &  & \mathcal{D} &  & TP\mathcal{C}\ar[rr]^{T\left(y_{\mathcal{C}}\right)_{\ast}} &  & T\mathcal{C}\ar[rr]^{u} &  & \mathcal{D}\\
 & \;\ar@{}[ul]|-{\stackrel{\omega_{2}^{\mathcal{C}}}{\Longleftarrow}} &  &  &  &  &  &  &  & \;\ar@{}[ul]|-{\stackrel{Tc_{\textnormal{id}_{\mathcal{C}}}}{\Longleftarrow}} & \ar@{}[]|-{\cong}\\
T\mathcal{C}\ar[uu]^{Ty_{\mathcal{C}}}\ar[rruu]_{y_{T\mathcal{C}}}\ar[rr]_{u} &  & \mathcal{D}\ar[rruu]_{y_{\mathcal{D}}}\ar@{}[uu]|-{\stackrel{y_{u}^{-1}}{\Longleftarrow}}\ar@/_{1pc}/[rurrur]_{\textnormal{id}_{\mathcal{D}}} &  & \ar@{}[uu]|-{\stackrel{c_{\textnormal{id}_{\mathcal{D}}}}{\Longleftarrow}} &  &  &  & T\mathcal{C}\ar[uu]^{Ty_{\mathcal{C}}}\ar[rruu]_{T\textnormal{id}_{\mathcal{C}}}\ar@/_{1pc}/[rurrur]_{u}
}
\]
We must check that the left extension (where $L$ is $P$-admissible)
\[
\xymatrix@=1em{\mathcal{B}\ar[rr]^{R_{L}} &  & P\mathcal{A}\ar@{}[dl]|-{\stackrel{\varphi_{L}}{\Longleftarrow}}\ar[rr]^{PH} &  & P\mathcal{C}\ar[rr]^{\left(y_{\mathcal{C}}\right)_{\ast}} &  & \mathcal{C}\\
 & \; &  & \; &  & \;\ar@{}[ul]|-{\stackrel{c_{\textnormal{id}_{\mathcal{C}}}}{\Longleftarrow}}\\
 &  & \mathcal{A}\ar[uull]^{L}\ar[uu]|-{y_{\mathcal{A}}}\ar[rr]_{H}\ar@{}[uurr]|-{\stackrel{y_{H}^{-1}}{\Longleftarrow}} &  & \mathcal{C}\ar[uu]|-{y_{\mathcal{C}}}\ar[rruu]_{\textnormal{id}_{\mathcal{C}}}
}
\]
is $T$-preserved by $u$. Indeed, on applying $T$ and whiskering
by $u$, and then pasting with this isomorphism of left extensions
and a naturality isomorphism of $\lambda$ (which we have by Lemma
\ref{lambdaccts}), we obtain 
\[
\xymatrix@=1em{ &  & PT\mathcal{A}\ar[rr]^{PTH} &  & PT\mathcal{C}\ar[rr]^{Pu} &  & P\mathcal{D}\ar[rrdd]^{\left(y_{\mathcal{D}}\right)_{\ast}} & \;\\
\\
T\mathcal{B}\ar[rr]^{TR_{L}} &  & TP\mathcal{A}\ar@{}[dl]|-{\stackrel{T\varphi_{L}}{\Longleftarrow}}\ar[rr]^{TPH}\ar[uu]^{\lambda_{\mathcal{A}}}\ar@{}[uurr]|-{\stackrel{\lambda_{H}^{-1}}{\Longleftarrow}} &  & TP\mathcal{C}\ar[rr]^{T\left(y_{\mathcal{C}}\right)_{\ast}}\ar[uu]_{\lambda_{\mathcal{C}}}\ar@{}[uurrr]|-{\cong} &  & T\mathcal{C}\ar[rr]^{u} &  & \mathcal{D}\\
 & \; &  & \; &  & \;\ar@{}[ul]|-{\stackrel{Tc_{\textnormal{id}_{\mathcal{C}}}}{\Longleftarrow}}\\
 &  & T\mathcal{A}\ar[uull]^{TL}\ar[uu]|-{Ty_{\mathcal{A}}}\ar[rr]_{TH}\ar@{}[uurr]|-{\stackrel{Ty_{H}^{-1}}{\Longleftarrow}} &  & T\mathcal{C}\ar[uu]|-{Ty_{\mathcal{C}}}\ar[rruu]_{T\textnormal{id}_{\mathcal{C}}}
}
\]
Then noting that pasting with invertible 2-cells preserves left extensions
and that
\[
\xymatrix@=1em{T\mathcal{B}\ar[rr]^{TR_{L}} &  & TP\mathcal{A}\ar@{}[dl]|-{\stackrel{T\varphi_{L}}{\Longleftarrow}}\ar[rr]^{\lambda_{\mathcal{A}}}\ar@{}[dr]|-{\stackrel{\omega_{2}}{\Longleftarrow}} &  & PT\mathcal{A}\ar[rr]^{PTH} &  & PT\mathcal{C}\ar[rr]^{Pu} &  & P\mathcal{D}\ar[rr]^{\left(y_{\mathcal{D}}\right)_{\ast}} &  & \mathcal{D}\\
 & \; &  & \;\\
 &  & T\mathcal{A}\ar[uull]^{TL}\ar[uu]|-{Ty_{\mathcal{A}}}\ar[rruu]_{y_{T\mathcal{A}}}
}
\]
is a left extension as a consequence of $TL$ being $P$-admissible
(thus the left extension $\lambda_{\mathcal{A}}\cdot TR_{L}$ in Proposition
\ref{lambdabeck} being preserved), we have the result.
\end{proof}
We now have everything required to complete the proof of the main
theorem.
\begin{thm}
\label{dimpliesa} In the statement of Theorem \ref{liftkzequiv}
(d) implies (a).
\end{thm}

\begin{proof}
Firstly, we observe that each $z_{x}$ is $T_{P}$-adm-cocontinuous
as a consequence of Proposition \ref{admcocontinuousequiv}. It follows
that we have the left extensions
\[
\xymatrix@=1em{T^{2}P\mathcal{A}\ar[rr]^{Tz_{x}} &  & TP\mathcal{A}\ar@{}[dldl]|-{\Uparrow T\xi_{x}}\ar[rr]^{z_{x}}\ar@{}[ddrr]|-{\Uparrow\xi_{x}} &  & P\mathcal{A} &  & T^{3}P\mathcal{A}\ar[rr]^{T^{2}z_{x}} &  & T^{2}P\mathcal{A}\ar[rr]^{Tz_{x}}\ar@{}[dldl]|-{\Uparrow T^{2}\xi_{x}} &  & TP\mathcal{A}\ar@{}[dldl]|-{\Uparrow T\xi_{x}}\ar[rr]^{z_{x}}\ar@{}[ddrr]|-{\Uparrow\xi_{x}} &  & P\mathcal{A}\\
 &  &  & \; &  &  &  &  &  &  &  & \;\\
T^{2}\mathcal{A}\ar[rr]_{Tx}\ar[uu]^{T^{2}y_{\mathcal{A}}} &  & T\mathcal{A}\ar[uu]|-{Ty_{\mathcal{A}}}\ar[rr]_{x} &  & \mathcal{A}\ar[uu]_{y_{\mathcal{A}}} &  & T^{3}\mathcal{A}\ar[rr]_{T^{2}x}\ar[uu]^{T^{3}y_{\mathcal{A}}} &  & T^{2}\mathcal{A}\ar[rr]_{Tx}\ar[uu]|-{T^{2}y_{\mathcal{A}}} &  & T\mathcal{A}\ar[uu]|-{Ty_{\mathcal{A}}}\ar[rr]_{x} &  & \mathcal{A}\ar[uu]_{y_{\mathcal{A}}}
}
\]
upon noting that each $T^{2}y_{\mathcal{A}}$ and $T^{3}y_{\mathcal{A}}$
is $P$-admissible. 

Secondly, we check that each $\left(P\mathcal{A},z_{x}\right)$ is
a pseudo $T$-algebra. We define our algebra structure maps as the
unique solutions to the following left extension problems (and note
they are invertible as they are isomorphisms of left extensions by
Proposition \ref{claim})
\[
\xymatrix{ & TP\mathcal{A}\ar[dr]^{z_{x}} &  &  &  & TP\mathcal{A}\ar[r]^{z_{x}}\ar@{}[rd]|-{\Uparrow\xi_{x}} & P\mathcal{A}\\
P\mathcal{A}\ar[rr]^{\textnormal{id}_{P\mathcal{A}}}\ar[ur]^{u_{P\mathcal{A}}} & \;\ar@{}[u]|-{\Uparrow\sigma_{x}} & P\mathcal{A} & = & P\mathcal{A}\ar[ur]^{u_{P\mathcal{A}}}\ar@{}[r]|-{\Uparrow u_{y_{\mathcal{A}}}} & T\mathcal{A}\ar[r]_{x}\ar[u]^{Ty_{\mathcal{A}}}\ar@{}[d]|-{\cong} & \mathcal{A}\ar[u]_{y_{\mathcal{A}}}\\
 &  & \mathcal{A}\ar[u]_{y_{\mathcal{A}}}\ar[ull]_{\quad\;\;\stackrel{\textnormal{id}}{\Longleftarrow}}^{y_{\mathcal{A}}} &  & \mathcal{A}\ar[ur]_{u_{\mathcal{A}}}\ar@/_{1pc}/[rru]_{\textnormal{id}_{\mathcal{A}}}\ar[u]^{y_{\mathcal{A}}} & \;\\
 & TP\mathcal{A}\ar[dr]^{z_{x}} &  &  &  & TP\mathcal{A}\ar[r]^{z_{x}}\ar@{}[rd]|-{\Uparrow\xi_{x}} & P\mathcal{A}\\
T^{2}P\mathcal{A}\ar[r]^{Tz_{x}}\ar[ur]^{m_{P\mathcal{A}}}\ar@{}[rd]|-{\Uparrow T\xi_{x}} & TP\mathcal{A}\ar[r]^{z_{x}}\ar@{}[rd]|-{\Uparrow\xi_{x}}\ar@{}[u]|-{\Uparrow\delta_{x}} & P\mathcal{A} & = & T^{2}P\mathcal{A}\ar[ur]^{m_{P\mathcal{A}}}\ar@{}[r]|-{\Uparrow m_{y_{\mathcal{A}}}} & T\mathcal{A}\ar[r]_{x}\ar[u]^{Ty_{\mathcal{A}}}\ar@{}[d]|-{\cong} & \mathcal{A}\ar[u]_{y_{\mathcal{A}}}\\
T^{2}\mathcal{A}\ar[r]_{Tx}\ar[u]^{T^{2}y_{\mathcal{A}}} & T\mathcal{A}\ar[r]_{x}\ar[u]|-{Ty_{\mathcal{A}}} & \mathcal{A}\ar[u]_{y_{\mathcal{A}}} &  & T^{2}\mathcal{A}\ar[ur]_{m_{\mathcal{A}}}\ar[u]^{T^{2}y_{\mathcal{A}}}\ar[r]_{Tx} & T\mathcal{A}\ar[ur]_{x}
}
\]
Note that these are the axioms for $\xi_{x}$ to exhibit $y_{\mathcal{A}}$
as a pseudo $T$-morphism. To check that the algebra structure coherence
axioms are satisfied, we note that the equalities
\[
\xymatrix@=1em{ & \;\\
 & T^{2}P\mathcal{A}\ar[rr]^{m_{P\mathcal{A}}}\ar[rd]^{Tz_{x}}\ar@{}[u]|-{\cong} &  & TP\mathcal{A}\ar[rd]^{z_{x}}\\
TP\mathcal{A}\ar[rr]|-{\textnormal{id}}\ar[ur]|-{Tu_{P\mathcal{A}}}\ar@/^{3pc}/[rrur]^{\textnormal{id}_{TP\mathcal{A}}} & \;\ar@{}[u]|-{\Uparrow T\sigma_{x}} & TP\mathcal{A}\ar[rr]^{z_{x}}\ar@{}[rdrd]|-{\Uparrow\xi_{x}}\ar@{}[ur]|-{\Uparrow\delta_{x}} &  & P\mathcal{A} &  & TP\mathcal{A}\ar[rr]^{z_{x}}\ar@{}[rdrd]|-{\Uparrow\xi_{x}} &  & P\mathcal{A}\\
 & \;\ar@{}[ul]|-{=} &  &  &  & =\\
T\mathcal{A}\ar[rrrr]_{x}\ar[uu]^{Ty_{\mathcal{A}}}\ar[uurr]_{Ty_{\mathcal{A}}} &  &  &  & \mathcal{A}\ar[uu]_{y_{\mathcal{A}}} &  & T\mathcal{A}\ar[uu]^{Ty_{\mathcal{A}}}\ar[rr]_{x} &  & \mathcal{A}\ar[uu]_{y_{\mathcal{A}}}\\
 & \; & \; & TP\mathcal{A}\ar@/^{1pc}/[rdd]^{z_{x}}\\
 & T^{2}P\mathcal{A}\ar[rr]^{Tz_{x}}\ar[rru]^{m_{P\mathcal{A}}}\ar@{}[u]|-{\quad\;\cong} &  & TP\mathcal{A}\ar[rd]^{z_{x}}\ar@{}[u]|-{\Uparrow\delta_{x}}\\
TP\mathcal{A}\ar[rr]_{z_{x}}\ar[ur]|-{u_{TP\mathcal{A}}}\ar@/^{2pc}/[rrruu]^{\textnormal{id}_{TP\mathcal{A}}} &  & P\mathcal{A}\ar[rr]_{\textnormal{id}_{P\mathcal{A}}}\ar[ur]^{u_{P\mathcal{A}}}\ar@{}[ul]|-{\Uparrow u_{z_{x}}^{-1}} & \;\ar@{}[u]|-{\Uparrow\sigma_{x}} & P\mathcal{A} &  & TP\mathcal{A}\ar[rr]^{z_{x}}\ar@{}[rdrd]|-{\Uparrow\xi_{x}} &  & P\mathcal{A}\\
 &  &  & \;\ar@{}[ru]|-{=} &  & =\\
T\mathcal{A}\ar[rrrr]_{x}\ar[uu]^{Ty_{\mathcal{A}}}\ar@{}[rrruu]|-{\Uparrow\xi_{x}} &  &  &  & \mathcal{A}\ar[uu]_{y_{\mathcal{A}}}\ar[uull]^{y_{\mathcal{A}}} &  & T\mathcal{A}\ar[uu]^{Ty_{\mathcal{A}}}\ar[rr]_{x} &  & \mathcal{A}\ar[uu]_{y_{\mathcal{A}}}
}
\]
and the equality between

\[
\xymatrix@=1em{T^{2}P\mathcal{A}\ar@/^{1.5pc}/[rrrr]^{m_{P\mathcal{A}}} & \ar@{}[]|-{\cong} & T^{2}P\mathcal{A}\ar[rr]^{m_{P\mathcal{A}}}\ar[ddrr]^{Tz_{x}} &  & TP\mathcal{A}\ar[rddr]^{z_{x}}\\
\; & \;\\
T^{3}P\mathcal{A}\ar[rr]^{T^{2}z_{x}}\ar[ruru]^{Tm_{P\mathcal{A}}}\ar[uu]^{m_{TP\mathcal{A}}} &  & T^{2}P\mathcal{A}\ar[rr]^{Tz_{x}}\ar@{}[uu]|-{\Uparrow T\delta_{x}} &  & TP\mathcal{A}\ar[rr]^{z_{x}}\ar@{}[uu]|-{\Uparrow\delta_{x}} &  & P\mathcal{A}\\
\\
T^{3}\mathcal{A}\ar[rr]_{T^{2}x}\ar[uu]^{T^{3}y_{\mathcal{A}}}\ar@{}[rruu]|-{\Uparrow T^{2}\xi_{x}} &  & T^{2}\mathcal{A}\ar[rr]_{Tx}\ar[uu]|-{T^{2}y_{\mathcal{A}}}\ar@{}[rruu]|-{\Uparrow T\xi_{x}} &  & T\mathcal{A}\ar[rr]_{x}\ar[uu]|-{Ty_{\mathcal{A}}}\ar@{}[rruu]|-{\Uparrow\xi_{x}} &  & \mathcal{A}\ar[uu]_{y_{\mathcal{A}}}
}
\]
and

\[
\xymatrix@=1em{ &  &  &  & TP\mathcal{A}\ar@/^{1pc}/[rdrddd]^{z_{x}}\\
\\
 &  & T^{2}P\mathcal{A}\ar[rr]^{Tz_{x}}\ar[rruu]^{m_{P\mathcal{A}}} &  & TP\mathcal{A}\ar[rddr]^{z_{x}}\ar@{}[uu]|-{\Uparrow\delta_{x}}\\
\\
T^{3}P\mathcal{A}\ar[rr]^{T^{2}z_{x}}\ar[rruu]^{m_{TP\mathcal{A}}} &  & T^{2}P\mathcal{A}\ar[rr]^{Tz_{x}}\ar[uurr]^{m_{P\mathcal{A}}}\ar@{}[uu]|-{\Uparrow m_{z_{x}}^{-1}\ } &  & TP\mathcal{A}\ar[rr]^{z_{x}}\ar@{}[uu]|-{\Uparrow\delta_{x}} &  & P\mathcal{A}\\
\\
T^{3}\mathcal{A}\ar[rr]_{T^{2}x}\ar[uu]^{T^{3}y_{\mathcal{A}}}\ar@{}[rruu]|-{\Uparrow T^{2}\xi_{x}} &  & T^{2}\mathcal{A}\ar[rr]_{Tx}\ar[uu]|-{T^{2}y_{\mathcal{A}}}\ar@{}[rruu]|-{\Uparrow T\xi_{x}} &  & T\mathcal{A}\ar[rr]_{x}\ar[uu]|-{Ty_{\mathcal{A}}}\ar@{}[rruu]|-{\Uparrow\xi_{x}} &  & \mathcal{A}\ar[uu]_{y_{\mathcal{A}}}
}
\]
easily follow from the respective conditions on $\left(\mathcal{A},x\right)$
being a pseudo $T$-algebra and the definitions of $\delta_{x}$ and
$\sigma_{x}$.

We now use the above to define our KZ doctrine 
\[
\widetilde{P}\colon\text{ps-\ensuremath{T}-alg}\to\text{ps-\ensuremath{T}-alg}
\]
We use the assignment on objects $\left(\mathcal{A},x\right)\mapsto\left(P\mathcal{A},z_{x}\right)$.
We take our units as the pseudo $T$-morphisms $\left(y_{\mathcal{A}},\xi_{x}\right)\colon\left(\mathcal{A},x\right)\to\left(P\mathcal{A},z_{x}\right)$.
Now suppose that we are given a pseudo $T$-morphism $\left(F,\phi\right)\colon\left(\mathcal{A},x\right)\to\left(P\mathcal{B},z_{r}\right)$,
where $\left(P\mathcal{B},z_{r}\right)=\widetilde{P}\left(\mathcal{B},r\right)$,
as in the diagram
\[
\xymatrix@=1em{\left(P\mathcal{A},z_{x}\right)\ar[rr]^{\left(\overline{F},\overline{\phi}\right)}\ar@{}[dr]|-{\stackrel{c_{F}}{\Longleftarrow}} &  & \left(P\mathcal{B},z_{r}\right)\\
 & \;\\
\left(\mathcal{A},x\right)\ar[uurr]_{\left(F,\phi\right)}\ar[uu]^{\left(y_{\mathcal{A}},\xi_{x}\right)}
}
\]
Since $z_{r}$ is $T_{P}$-cocontinuous, we may apply Proposition
\ref{docleftext} to find a lax $T$-morphism $\left(\overline{F},\overline{\phi}\right)$
as above. Indeed the lax structure map $\overline{\phi}$ is given
as the unique solution to 
\[
\xymatrix{ & TP\mathcal{A}\ar[r]^{z_{x}} & P\mathcal{A}\ar[dd]^{\overline{F}} &  &  & TP\mathcal{A}\ar[r]^{z_{x}}\ar[dd]_{T\overline{F}} & P\mathcal{A}\ar[dd]^{\overline{F}}\\
T\mathcal{A}\ar[r]^{x}\ar[ur]^{Ty_{\mathcal{A}}}\ar[rd]_{TF} & \mathcal{A}\ar[dr]^{F}\ar[ur]^{y_{\mathcal{A}}}\ar@{}[r]|-{\Uparrow c_{F}}\ar@{}[u]|-{\Uparrow\xi_{x}} & \; & = & T\mathcal{A}\ar[ru]^{Ty_{\mathcal{A}}}\ar[rd]_{TF}\ar@{}[r]|-{\Uparrow Tc_{F}\quad} & \;\ar@{}[r]|-{\Uparrow\overline{\phi}} & \;\\
 & TP\mathcal{B}\ar[r]_{z_{r}}\ar@{}[u]|-{\Uparrow\phi} & P\mathcal{B} &  &  & TP\mathcal{B}\ar[r]_{z_{r}} & P\mathcal{B}
}
\]
But we notice that
\[
\xymatrix@=1em{TP\mathcal{A}\ar[rr]^{T\overline{F}}\ar@{}[dr]|-{\stackrel{Tc_{F}}{\Longleftarrow}} &  & TP\mathcal{B}\ar[rr]^{z_{r}} &  & P\mathcal{B} &  & TP\mathcal{A}\ar[rr]^{z_{x}} &  & P\mathcal{A}\ar[rr]^{\overline{F}} &  & P\mathcal{B}\\
 & \; &  &  &  &  &  &  &  & \;\ar@{}[ul]|-{\stackrel{c_{F}}{\Longleftarrow}}\\
T\mathcal{A}\ar[uu]^{Ty_{\mathcal{A}}}\ar[rruu]_{TF} &  &  &  &  &  & T\mathcal{A}\ar[rr]_{x}\ar[uu]^{Ty_{\mathcal{A}}}\ar@{}[urur]|-{\Uparrow\xi_{x}}\ar@/_{0.5pc}/[rrrd]_{TF} &  & \mathcal{A}\ar[uu]_{y_{\mathcal{A}}}\ar[uurr]_{F}\ar@{}[dr]|-{\cong\phi}\\
 &  &  &  &  &  &  &  &  & TP\mathcal{B}\ar[uuru]_{z_{r}}
}
\]
are both left extensions since $z_{r}$ is $T_{P}$-cocontinuous and
$Ty_{\mathcal{A}}$ is \emph{$P$-}admissible respectively. It follows
that the lax $T$-morphism structure map $\overline{\phi}$ is an
isomorphism of left extensions, making $\left(\overline{F},\overline{\phi}\right)$
a pseudo $T$-morphism. Of course, if we only assume $\left(F,\phi\right)$
to be a lax $T$-morphism then we can only expect $\overline{F}$
to admit a lax $T$-morphism structure.

We now check that such left extensions are preserved by other left
extensions of this form. Suppose we are given two left extensions
of pseudo $T$-algebras and pseudo $T$-morphisms
\[
\xymatrix@=1em{\left(P\mathcal{A},z_{x}\right)\ar[rr]^{\left(\overline{F},\overline{\phi}\right)}\ar@{}[dr]|-{\stackrel{c_{F}}{\Longleftarrow}} &  & \left(P\mathcal{B},z_{r}\right) &  &  & \left(P\mathcal{B},z_{r}\right)\ar[rr]^{\left(\overline{G},\overline{\sigma}\right)}\ar@{}[dr]|-{\stackrel{c_{G}}{\Longleftarrow}} &  & \left(P\mathcal{C},z_{h}\right)\\
 & \; &  &  &  &  & \;\\
\left(\mathcal{A},x\right)\ar[uurr]_{\left(F,\phi\right)}\ar[uu]^{\left(y_{\mathcal{A}},\xi_{x}\right)} &  &  &  &  & \left(\mathcal{B},r\right)\ar[uurr]_{\left(G,\sigma\right)}\ar[uu]^{\left(y_{\mathcal{B}},\xi_{r}\right)}
}
\]
To see that
\[
\xymatrix@=1em{\left(P\mathcal{A},z_{x}\right)\ar[rr]^{\left(\overline{F},\overline{\phi}\right)} &  & \left(P\mathcal{B},z_{r}\right)\ar[rr]^{\left(\overline{G},\overline{\sigma}\right)} &  & \left(P\mathcal{C},z_{h}\right)\\
\ar@{}[rr]|-{\stackrel{\left(\overline{G},\overline{\sigma}\right)c_{F}}{\Longleftarrow}} & \; & \left(P\mathcal{B},z_{r}\right)\ar[rur]_{\left(\overline{G},\overline{\sigma}\right)}\\
\left(\mathcal{A},x\right)\ar[urr]_{\left(F,\phi\right)}\ar[uu]^{\left(y_{\mathcal{A}},\xi_{x}\right)}
}
\]
is a left extension we need only observe that the $T$-morphism structure
on $\overline{G}\overline{F}$ resulting from an application of Proposition
\ref{docleftext} (on the outside diagram) is given by composing $\overline{\phi}$
and $\overline{\sigma}$ as above. This is shown by pasting the defining
diagram for $\overline{\phi}$ with $\overline{\sigma}$ which gives

\begin{equation}
\xymatrix{ & TP\mathcal{A}\ar[r]^{z_{x}} & P\mathcal{A}\ar[dd]^{\overline{F}} &  &  & TP\mathcal{A}\ar[r]^{z_{x}}\ar[dd]_{T\overline{F}} & P\mathcal{A}\ar[dd]^{\overline{F}}\\
T\mathcal{A}\ar[r]^{x}\ar[ur]^{Ty_{\mathcal{A}}}\ar[rd]_{TF} & \mathcal{A}\ar[dr]^{F}\ar[ur]^{y_{\mathcal{A}}}\ar@{}[r]|-{\Uparrow c_{F}}\ar@{}[u]|-{\Uparrow\xi_{x}} & \; &  & T\mathcal{A}\ar[ru]^{Ty_{\mathcal{A}}}\ar[rd]_{TF}\ar@{}[r]|-{\Uparrow Tc_{F}\quad} & \;\ar@{}[r]|-{\Uparrow\overline{\phi}} & \;\\
 & TP\mathcal{B}\ar[r]_{z_{r}}\ar@{}[u]|-{\Uparrow\phi}\ar[dd]_{T\overline{G}} & P\mathcal{B}\ar[dd]^{\overline{G}} & = &  & TP\mathcal{B}\ar[r]_{z_{r}}\ar[dd]_{T\overline{G}} & P\mathcal{B}\ar[dd]^{\overline{G}}\\
 & \;\ar@{}[r]|-{\Uparrow\overline{\sigma}} & \; &  &  & \;\ar@{}[r]|-{\Uparrow\overline{\sigma}} & \;\\
 & TP\mathcal{C}\ar[r]_{z_{h}} & P\mathcal{C} &  &  & TP\mathcal{C}\ar[r]_{z_{h}} & P\mathcal{C}
}
\label{prescoh}
\end{equation}
which is  the defining diagram for the induced lax structure on $\overline{G}\cdot\overline{F}$
from an application of Proposition \ref{docleftext}.

It is an easy consequence of Proposition \ref{docleftext} that each
$\left(y_{\mathcal{A}},\xi_{x}\right)$ is dense. Indeed since $z_{x}$
$T$-preserves the left extension
\[
\xymatrix@=1em{P\mathcal{A}\ar[rr]^{\textnormal{id}_{P\mathcal{A}}} &  & P\mathcal{A}\\
 & \ar@{}[ul]|-{=}\\
\mathcal{A}\ar[uu]^{y_{\mathcal{A}}}\ar[uurr]_{y_{\mathcal{A}}}
}
\]
(as well the resulting left extension) the density property may be
lifted to pseudo-$T$-algebras applying Proposition \ref{docleftext}.
\end{proof}

\section{Consequences and Examples\label{consequencesandexamples}}

In this section we point out some consequences of Theorem \ref{liftkzequiv}
proven in the previous section, and in particular some properties
of the lifted KZ doctrine $\widetilde{P}$ on $\textnormal{ps-}T\textnormal{-alg}$.
Before considering the properties of $\widetilde{P}$, we mention
two easy corollaries.
\begin{cor}
\label{distunique} Pseudo-distributive laws over KZ pseudomonads
are essentially unique.
\end{cor}

\begin{proof}
As shown in Lemma \ref{w2isext}, the modification components $\omega_{2}^{\mathcal{A}}$
exhibit $\lambda_{\mathcal{A}}$ as a left extension. The last two
coherence axioms of a pseudo-distributive law over a KZ pseudomonad
then define the components $\omega_{1}^{\mathcal{A}}$ and $\omega_{3}^{\mathcal{A}}$
as unique solutions to a left extension problem. Note that $\omega_{4}^{\mathcal{A}}$
is also defined as the unique solution to a left extension problem
(see the proof of \ref{w2isext}). The essential uniqueness of left
extensions then tells us these pseudo-distributive laws are essentially
unique.
\end{proof}
\begin{cor}
When the conditions of Theorem \ref{liftkzequiv} are met, the lifted
pseudomonad arising from the pseudo-distributive law is automatically
KZ.
\end{cor}

\begin{proof}
As a consequence of the essential uniqueness of pseudo-distributive
laws over KZ pseudomonads, any lifted pseudomonad must be equivalent
to the KZ pseudomonad whose existence is guaranteed by Theorem \ref{liftkzequiv}.
\end{proof}

\subsection{The Lifted KZ Doctrines}

We first check that in addition to having a lifting to $\textnormal{ps-}T\textnormal{-alg}$,
we have a lifting to the 2-category of pseudo-$T$-algebras, lax (or
oplax) $T$-morphisms, and $T$-transformations.
\begin{prop}
\noindent \label{oplaxlaxlift} Suppose any of the equivalent conditions
of Theorem \ref{liftkzequiv} are satisfied. Then

(a) $P$ lifts to a KZ doctrine $\widetilde{P}_{\textnormal{oplax}}$
on $\text{ps-}T\text{-alg}_{\textnormal{oplax}}$;

(b) $P$ lifts to a KZ doctrine $\widetilde{P}_{\textnormal{lax}}$
on $\text{ps-}T\text{-alg}_{\textnormal{lax}}$;
\end{prop}

\begin{proof}
$\left(a\right)\colon$$P$ lifts to a KZ doctrine $\widetilde{P}_{\textnormal{oplax}}$
on $\text{ps-}T\text{-alg}_{\textnormal{oplax}}$ since given any
oplax structure cell $\varphi$ on a map $F\colon\mathcal{A}\to P\mathcal{B}$
as below
\[
\xymatrix@=1em{\left(P\mathcal{A},z_{x}\right)\ar[rr]^{\left(\overline{F},\overline{\varphi}\right)}\ar@{}[dr]|-{\stackrel{c_{F}}{\Longleftarrow}} &  & \left(P\mathcal{B},z_{r}\right)\\
 & \;\\
\left(\mathcal{A},x\right)\ar[uurr]_{\left(F,\varphi\right)}\ar[uu]^{\left(y_{\mathcal{A}},\xi_{x}\right)}
}
\]
we get an oplax structure cell $\overline{\varphi}$ given as unique
the solution to 
\[
\xymatrix@=1em{TP\mathcal{B}\ar[rr]^{z_{r}}\ar@{}[ddrr]|-{\Uparrow\overline{\varphi}} &  & P\mathcal{B} &  &  &  & TP\mathcal{B}\ar[rr]^{z_{r}}\ar@{}[rrdddd]|-{\Uparrow\varphi} &  & P\mathcal{B}\\
\\
TP\mathcal{A}\ar[rr]^{z_{x}}\ar@{}[ddrr]|-{\Uparrow\xi_{x}}\ar[uu]^{T\overline{F}} &  & P\mathcal{A}\ar[uu]_{\overline{F}} & = & TP\mathcal{A}\ar[uurr]^{T\overline{F}} & \ar@{}[]|-{\;\;\stackrel{Tc_{F}^{-1}}{\Longrightarrow}} &  &  &  & \ar@{}[]|-{\stackrel{c_{F}}{\Longrightarrow}} & P\mathcal{A}\ar[lluu]_{\overline{F}}\\
 & \;\\
T\mathcal{A}\ar[uu]^{Ty_{\mathcal{A}}}\ar[rr]_{x} &  & \mathcal{A}\ar[uu]_{y_{\mathcal{A}}} &  &  &  & T\mathcal{A}\ar[uull]^{Ty_{\mathcal{A}}}\ar[rr]_{x}\ar[uuuu]|-{TF} &  & \mathcal{A}\ar[urru]_{y_{\mathcal{A}}}\ar[uuuu]|-{F}
}
\]
with the coherence conditions for $\overline{\varphi}$ being an oplax
$T$-morphism structure following from Proposition \ref{claim} (Part
4). Note that the induced oplax structure when composed by an oplax
$T$-morphism $\left(\overline{G},\overline{\tau}\right)$ as below
\[
\xymatrix@=1em{\left(P\mathcal{A},z_{x}\right)\ar[rr]^{\left(\overline{F},\overline{\varphi}\right)}\ar@{}[dr]|-{\stackrel{c_{F}}{\Longleftarrow}} &  & \left(P\mathcal{B},z_{r}\right)\ar[rr]^{\left(\overline{G},\overline{\tau}\right)} &  & \left(P\mathcal{C},z_{k}\right)\\
 & \;\\
\left(\mathcal{A},x\right)\ar[uurr]_{\left(F,\varphi\right)}\ar[uu]^{\left(y_{\mathcal{A}},\xi_{x}\right)}
}
\]
is still $\left(\overline{G},\overline{\tau}\right)\cdot\left(\overline{F},\overline{\varphi}\right)$.
To see that $\left(\overline{F},\overline{\varphi}\right)$ is a left
extension in the sense of transformations, suppose we are given a
transformation $\sigma\colon\left(F,\varphi\right)\to\left(H,\psi\right)\cdot\left(y_{\mathcal{A}},\xi_{x}\right)$,
then the induced cell $\overline{\sigma}\colon\overline{F}\to H$
is a transformation since

\[
\xymatrix@=1em{ & TP\mathcal{B}\ar[rr]^{z_{r}}\ar@{}[ddrr]|-{\Uparrow\overline{\varphi}} &  & P\mathcal{B} &  &  & TP\mathcal{B}\ar[rr]^{z_{r}}\ar@{}[dd]|-{\Uparrow\psi} &  & P\mathcal{B}\\
\;\ar@{}[r]|-{\overset{T\overline{\sigma}}{\Longleftarrow}} & \; &  &  &  &  &  & \;\ar@{}[r]|-{\overset{\overline{\sigma}}{\Longleftarrow}} & \;\\
 & TP\mathcal{A}\ar[rr]^{z_{x}}\ar@{}[ddrr]|-{\Uparrow\xi_{x}}\ar[uu]_{T\overline{F}}\ar@/^{2pc}/[uu]^{TH} &  & P\mathcal{A}\ar[uu]_{\overline{F}} & = &  & TP\mathcal{A}\ar[rr]^{z_{x}}\ar@{}[ddrr]|-{\Uparrow\xi_{x}}\ar@/^{2pc}/[uu]^{TH} &  & P\mathcal{A}\ar[uu]_{\overline{F}}\ar@/^{2pc}/[uu]^{H}\\
 &  & \; &  &  &  &  & \;\\
 & T\mathcal{A}\ar[uu]^{Ty_{\mathcal{A}}}\ar[rr]_{x} &  & \mathcal{A}\ar[uu]_{y_{\mathcal{A}}} &  &  & T\mathcal{A}\ar[uu]^{Ty_{\mathcal{A}}}\ar[rr]_{x} &  & \mathcal{A}\ar[uu]_{y_{\mathcal{A}}}
}
\]
as a consequence of $\sigma$ being a transformation. By Proposition
\ref{docleftext} the density property is still valid in the setting
of oplax $T$-morphisms; this being why we proved the general case
of Proposition \ref{docleftext} in terms of composites of lax and
oplax morphisms.

$\left(b\right)\colon$ The proof that $P$ lifts to a KZ doctrine
$\widetilde{P}_{\textnormal{lax}}$ on $\text{ps-}T\text{-alg}_{\textnormal{lax}}$
is essentially given in Theorem \ref{dimpliesa}.
\end{proof}
We now check that the KZ structure cell $\theta\colon Py\to yP$ remains
the same upon lifting to algebras. 
\begin{prop}
\noindent Suppose any of the equivalent conditions of Theorem \ref{liftkzequiv}
are satisfied. Then the KZ structure cell $\theta\colon Py\to yP$
for $P$ is also the KZ structure cell for $\widetilde{P}$.
\end{prop}

\begin{proof}
Recall that the components of $\theta$ are recovered as the induced
cells out of the left extensions $Py_{\mathcal{A}}$ as in the diagram
below
\[
\xymatrix@=1em{P\mathcal{A}\ar[rr]^{Py_{\mathcal{A}}}\ar@{}[rrdd]|-{\Uparrow y_{y_{\mathcal{A}}}^{-1}} &  & P^{2}\mathcal{A}\\
\\
\mathcal{A}\ar[rr]_{y_{\mathcal{A}}}\ar[uu]^{y_{\mathcal{A}}} &  & P\mathcal{A}\ar[uu]_{y_{P\mathcal{A}}}
}
\]
such that the composite with this diagram is an identity. Now apply
Proposition \ref{docleftext} to this naturality square noting that
each $y_{\mathcal{A}}$ extends to a pseudo $T$-morphism $\left(y_{\mathcal{A}},\xi_{x}\right)$
in order to recover the components of the KZ structure cell for $\widetilde{P}$.
\end{proof}
If we are to study the lifted KZ doctrine $\widetilde{P}$, we should
consider the $\widetilde{P}$-cocomplete objects and the $\widetilde{P}$-admissible
maps. We start with the former.

Algebraic cocompleteness is usually defined by asking that the underlying
object be cocomplete, and that the algebra structure map be separately
cocontinuous. The following proposition justifies this definition.
\begin{prop}
Suppose any of the equivalent conditions of Theorem \ref{liftkzequiv}
are satisfied. Then a pseudo $T$-algebra $\left(\mathcal{A},x\right)$
is

(a) $\widetilde{P}$-cocomplete iff $\mathcal{A}$ is $P$-cocomplete
and $x\colon T\mathcal{A}\to\mathcal{A}$ is $T_{P}$-cocontinuous;

(b) $\widetilde{P}_{\textnormal{lax}}$-cocomplete iff $\mathcal{A}$
is $P$-cocomplete and $x\colon T\mathcal{A}\to\mathcal{A}$ is $T_{P}$-cocontinuous;

(c) $\widetilde{P}_{\textnormal{oplax}}$-cocomplete iff $\mathcal{A}$
is $P$-cocomplete.

\noindent Moreover, the pseudo/lax/oplax $T$-morphisms $\left(F,\phi\right)$
which are $\widetilde{P}/\widetilde{P}_{\textnormal{lax}}/\widetilde{P}_{\textnormal{oplax}}$-cocontinuous
are all classified by those maps for which the underlying $F$ is
$P$-cocontinuous.
\end{prop}

\begin{proof}
We start off by proving part ($a$).

$\left(\Longrightarrow\right)\colon$ Suppose that $\left(\mathcal{A},x\right)$
is a $\widetilde{P}$-cocomplete pseudo $T$-algebra. Then, by doctrinal
adjunction \cite{doctrinal}, the pseudo $T$-morphism $\left(y_{\mathcal{A}},\xi_{x}\right)$
has a reflection left adjoint $\left(\left(y_{\mathcal{A}}\right)_{\ast},\left(\xi_{x}^{-1}\right)_{\ast}\right)$
for which $\left(\xi_{x}^{-1}\right)_{\ast}$ is defined by the mates
correspondence and is invertible. That is, we have isomorphisms 
\[
\xymatrix@=1em{TP\mathcal{A}\ar[rr]^{z_{x}}\ar@{}[ddrr]|-{\Downarrow\xi_{x}^{-1}} &  & P\mathcal{A} &  &  &  & TP\mathcal{A}\ar[rr]^{z_{x}}\ar@{}[ddrr]|-{\Downarrow\left(\xi_{x}^{-1}\right)_{\ast}}\ar[dd]_{T\left(y_{\mathcal{A}}\right)_{\ast}} &  & P\mathcal{A}\ar[dd]^{\left(y_{\mathcal{A}}\right)_{\ast}}\\
 & \; &  &  &  &  &  & \;\\
T\mathcal{A}\ar[uu]^{Ty_{\mathcal{A}}}\ar[rr]_{x} &  & \mathcal{A}\ar[uu]_{y_{\mathcal{A}}} &  &  &  & T\mathcal{A}\ar[rr]_{x} &  & \mathcal{A}
}
\]
Now $\left(y_{\mathcal{A}}\right)_{\ast}\dashv y_{\mathcal{A}}$ via
a reflection adjoint so $\mathcal{A}$ is $P$-cocomplete. We thus
check that $x\colon T\mathcal{A}\to\mathcal{A}$ is $T_{P}$-cocontinuous.
Suppose we are given a left extension as on the left
\[
\xymatrix@=1em{P\mathcal{D}\ar[rr]^{\overline{F}}\ar@{}[dr]|-{\Uparrow c_{F}} &  & \mathcal{A} &  &  & TP\mathcal{D}\ar[rr]^{\overline{F}}\ar@{}[dr]|-{\Uparrow Tc_{F}} &  & T\mathcal{A}\ar[rr]^{x} &  & \mathcal{A}\\
 & \; &  &  &  &  & \;\\
\mathcal{D}\ar[rruu]_{F}\ar[uu]^{y_{\mathcal{D}}} &  &  &  &  & T\mathcal{D}\ar[rruu]_{TF}\ar[uu]^{Ty_{\mathcal{D}}}
}
\]
We check that the right diagram is a left extension. We first note
this is equivalent to showing that $x$ $T$-preserves left extensions
as on the left below
\[
\xymatrix@=1em{P\mathcal{D}\ar[rr]^{PF}\ar@{}[drdr]|-{\Uparrow c_{y_{\mathcal{A}}\cdot F}} &  & P\mathcal{A}\ar[rr]^{\left(y_{\mathcal{A}}\right)_{\ast}}\ar@{}[dr]|-{\Uparrow c_{\textnormal{id}_{\mathcal{A}}}} &  & \mathcal{A} &  & TP\mathcal{D}\ar[rr]^{TPF}\ar@{}[drdr]|-{\Uparrow Tc_{y_{\mathcal{A}}\cdot F}} &  & TP\mathcal{A}\ar[rr]^{T\left(y_{\mathcal{A}}\right)_{\ast}}\ar@{}[dr]|-{\Uparrow Tc_{\textnormal{id}_{\mathcal{A}}}} &  & T\mathcal{A}\ar[rr]^{x} &  & \mathcal{A}\\
 &  &  & \; &  &  &  &  &  & \;\\
\mathcal{D}\ar[rr]_{F}\ar[uu]^{y_{\mathcal{D}}} &  & \mathcal{A}\ar[uu]|-{y_{\mathcal{A}}}\ar[uurr]_{\textnormal{id}_{\mathcal{A}}} &  &  &  & T\mathcal{D}\ar[rr]_{TF}\ar[uu]^{Ty_{\mathcal{D}}} &  & T\mathcal{A}\ar[uu]|-{Ty_{\mathcal{A}}}\ar[uurr]_{T\textnormal{id}_{\mathcal{A}}}
}
\]
and so it suffices to check the right diagram is a left extension.
This is seen upon pasting with the isomorphism $\left(\xi_{x}^{-1}\right)_{\ast}$
as $z_{x}$ is $T_{P}$-cocontinuous and $\left(y_{\mathcal{A}}\right)_{\ast}$
is a left adjoint (and hence preserves all left extensions).

$\left(\Longleftarrow\right)\colon$ Suppose that $\mathcal{A}$ is
$P$-cocomplete and $x$ is $T_{P}$-cocontinuous. Then $\left(\mathcal{A},x\right)$
is $\widetilde{P}$-cocomplete as $\left(\mathcal{A},x\right)$ admits
left extensions along $\left(y_{\mathcal{A}},\xi_{x}\right)$ by Proposition
\ref{docleftext}, and showing that such left extensions admit a pseudo
$T$-morphism structure and are preserved is a similar calculation
to that in the proof of Theorem \ref{dimpliesa}.

$\left(b\right)\colon$ The proof of the classification of $\widetilde{P}_{\textnormal{lax}}$-cocomplete
pseudo $P$-algebras is almost the same (as the reflection left adjoint
must again be pseudo by doctrinal adjunction \cite{doctrinal}), and
so we omit the details.

$\left(c\right)\colon$ The $\widetilde{P}_{\textnormal{oplax}}$-cocomplete
pseudo $P$-algebras are those with an underlying $P$-cocomplete
object, as a consequence of doctrinal adjunction \cite{doctrinal}.

That the $T$-morphisms $\left(F,\phi\right)$ which are $\widetilde{P}/\widetilde{P}_{\textnormal{lax}}/\widetilde{P}_{\textnormal{oplax}}$-cocontinuous
are all classified by those morphisms for which the underlying $F$
is $P$-cocontinuous is a straightforward calculation. Indeed, given
a pseudo $T$-morphism $\left(F,\phi\right)$ for which $F$ is $P$-cocontinuous,
checking that $\left(F,\phi\right)$ is then $\widetilde{P}$-cocontinuous
requires only checking a coherence condition (similar to \ref{prescoh}).
Conversely, given that $\left(F,\phi\right)$ is $\widetilde{P}$-cocontinuous,
that is, a pseudo $\widetilde{P}$-morphism on $\textnormal{ps-}T\textnormal{-alg}$,
we know the underlying $F$ must be a pseudo $P$-morphism on $\mathscr{C}$
(by forgetting that certain morphisms and 2-cells are $T$-algebraic),
so that $F$ is $P$-cocontinuous. The $\widetilde{P}_{\textnormal{lax}}\textnormal{ and }\widetilde{P}_{\textnormal{oplax}}$
case may be similarly seen.
\end{proof}
\begin{prop}
Suppose any of the equivalent conditions of Theorem \ref{liftkzequiv}
are satisfied. Assume $\left(L,\alpha\right)\colon\left(\mathcal{A},x\right)\to\left(\mathcal{B},y\right)$
is a pseudo $T$-morphism and $L\colon\mathcal{A}\to\mathcal{B}$
is $P$-admissible. Then $\left(L,\alpha\right)$ is $\widetilde{P}$-admissible
if and only if for every $\widetilde{P}$-cocomplete pseudo $T$-algebra
$\left(\mathcal{C},z\right)$ and pseudo $T$-morphism $\left(I,\xi\right)$
as in the diagram
\[
\xymatrix@=1em{\left(\mathcal{B},y\right)\ar[rr]^{\left(R,\beta\right)} &  & \left(\mathcal{C},z\right)\ar@{}[dl]|-{\stackrel{\delta}{\Longleftarrow}}\\
 & \;\\
 &  & \left(\mathcal{A},x\right)\ar[uu]_{\left(I,\xi\right)}\ar[uull]^{\left(L,\alpha\right)}
}
\]
the induced lax structure cell $\beta$ on the underlying left extension
$R$ as in Proposition \ref{docleftext} is invertible. Moreover,
for pseudo, lax and oplax $\left(L,\alpha\right)$ respectively,

1. $\left(L,\alpha\right)$ is $\widetilde{P}$-admissible iff $\widetilde{P}\left(L,\alpha\right)$
has a pseudo right adjoint;

2. $\left(L,\alpha\right)$ is $\widetilde{P}_{\textnormal{lax}}$-admissible
iff $\widetilde{P}\left(L,\alpha\right)$ is pseudo;

3. $\left(L,\alpha\right)$ is $\widetilde{P}_{\textnormal{oplax}}$-admissible
iff $\widetilde{P}\left(L,\alpha\right)$ has a pseudo right adjoint.
\end{prop}

\begin{proof}
The first part of this proposition follows an equivalent characterization
of $P$-admissibility as given by Bunge and Funk (discussed in \cite{bungefunk,yonedakz}),
along with Proposition \ref{docleftext}. The last three properties
are a direct consequence of doctrinal adjunction \cite{doctrinal}.
\end{proof}
\begin{rem}
Note that the conditions of $\widetilde{P}/\widetilde{P}_{\textnormal{oplax}}$-admissibility
are analogous to asking a Guitart exactness condition is satisfied
\cite{exactness} (in the presence of some additional structure, and
in the context of pointwise left extensions). However, we omit discussion
of this as it would take us beyond the scope of this paper.
\end{rem}

\begin{rem}
Note that if $P$ (and thus $\widetilde{P}$) is locally fully faithful,
and $\left(L,\alpha\right)$ is a lax $T$-morphism, then $\widetilde{P}\left(L,\alpha\right)$
being pseudo implies $\left(L,\alpha\right)$ is. Indeed, the lax
structure cell $\alpha$ when whiskered by $y_{\mathcal{A}}$ is invertible
(a direct consequence of how the structure cell of $\widetilde{P}\left(L,\alpha\right)$
is defined in Proposition \ref{docleftext}). As $y_{\mathcal{A}}$
is fully faithful, this means $\alpha$ is invertible. Hence, in this
case, Statement 2 of the above proposition is equivalent to saying
$\left(L,\alpha\right)$ is pseudo.
\end{rem}

Given a KZ doctrine $P$ on a 2-category $\mathscr{C}$ we have an
equivalence given by composition with the unit $y_{\mathcal{A}}$,
namely $\mathscr{C}_{\textnormal{ccts}}\left(P\mathcal{A},\mathcal{B}\right)\simeq\mathscr{C}\left(\mathcal{A},\mathcal{B}\right)$,
with $\mathscr{C}_{\textnormal{ccts}}\left(P\mathcal{A},\mathcal{B}\right)$
containing left extensions of maps $\mathcal{A}\to\mathcal{B}$ along
the unit $y_{\mathcal{A}}$. This is clearly essentially surjective
as for an $F\colon\mathcal{A}\to\mathcal{B}$ we may take $\overline{F}\colon P\mathcal{A}\to\mathcal{B}$,
and fully faithful as $y_{\mathcal{A}}$ is dense. We can thus recover
Im and Kelly's following result.
\begin{cor}
[Im-Kelly \cite{uniconvolution}] Suppose we are given a 2-category
$\mathscr{C}$ equipped with a pseudomonad $\left(T,u,m\right)$ and
a KZ doctrine $\left(P,y\right)$. Suppose any of the equivalent conditions
of Theorem \ref{liftkzequiv} are met. Then for every pair of pseudo
$T$-algebras $\left(\mathcal{A},x\right)$ and $\left(\mathcal{B},r\right)$
where $\mathcal{B}$ is $P$-cocomplete, composition with the unit
$\left(y_{\mathcal{A}},\xi_{x}\right)$ defines the equivalence 
\[
\mathbf{Oplax}\left[\left(\mathcal{A},x\right),\left(\mathcal{B},r\right)\right]\simeq\mathbf{Oplax}_{\textnormal{ccts}}\left[\left(P\mathcal{A},z_{x}\right),\left(\mathcal{B},r\right)\right]
\]
where a morphism of pseudo $T$-algebras is cocontinuous when the
underlying morphism is. Suppose further that $r$ is $T_{P}$-cocontinuous.
Then composition with the unit $\left(y_{\mathcal{A}},\xi_{x}\right)$
also defines the equivalences
\[
\begin{aligned}\mathbf{Lax}\left[\left(\mathcal{A},x\right),\left(\mathcal{B},r\right)\right] & \simeq\mathbf{Lax}_{\textnormal{ccts}}\left[\left(P\mathcal{A},z_{x}\right),\left(\mathcal{B},r\right)\right]\\
\mathbf{Pseudo}\left[\left(\mathcal{A},x\right),\left(\mathcal{B},r\right)\right] & \simeq\mathbf{Pseudo}_{\textnormal{ccts}}\left[\left(P\mathcal{A},z_{x}\right),\left(\mathcal{B},r\right)\right]
\end{aligned}
\]
Moreover, the above three equivalences restrict to \emph{$P$-}admissible
underlying morphisms.
\end{cor}

\begin{proof}
We need only check the restriction. Note that if $\overline{L}\colon P\mathcal{A}\to\mathcal{B}$
is \emph{$P$-}admissible then so is the composite $\overline{L}\cdot y_{\mathcal{A}}\cong L$
due to closure under composition. If $L$ is \emph{$P$-}admissible,
then $\overline{L}$ has a right adjoint by \cite[Lemma 12]{yonedakz},
and so $P\overline{L}$ also does.
\end{proof}

\subsection{The Preorder of KZ Doctrines on a 2-Category}

In the following discussion of morphisms between KZ pseudomonads and
doctrines we will omit most of the details, as this would take us
beyond the scope of this paper. Moreover, the calculations are quite
similar to those in Section \ref{liftingkzdoctrines}.

It is the goal of this subsection to show that the 2-category of KZ
pseudomonads on a 2-category $\mathscr{C}$ is biequivalent to a preorder.
This is a property one might expect given the ``property like structure''
viewpoint \cite{lack1997}; and the tools of admissible maps give
us a method of proving this result.
\begin{defn}
Given KZ pseudomonads $\left(P,y,\mu\right)$ and $\left(P',y',\mu'\right)$
on a 2-category $\mbox{\ensuremath{\mathscr{C}}}$, a \emph{morphism
of KZ pseudomonads} $P\Longrightarrow P'$ (corresponding to a lifting
of the identity on $\mathscr{C}$) consists of a pseudonatural transformation
$\alpha\colon P\to P'$ and an invertible modification 
\[
\xymatrix@=1em{P\ar[rr]^{\alpha} &  & P'\\
 & \ar@{}[ur]|-{\overset{\psi_{y}}{\Longleftarrow}}\\
 &  & 1_{\mathscr{C}}\ar[uu]_{y'}\ar[lluu]^{y}
}
\]
such that
\[
\xymatrix@=1em{ &  &  &  &  &  &  &  &  & P\ar[rr]^{\alpha}\ar[dd]_{y'P}\ar[dlld]_{yP} &  & P'\ar@/^{0.7pc}/[dd]^{P'y'}\ar@/_{0.7pc}/[dd]_{y'P'}\\
P\ar[rr]^{\alpha}\ar@/^{0.7pc}/[dd]^{Py}\ar@/_{0.7pc}/[dd]_{yP} &  & P'\ar[dd]^{P'y}\ar[rrdd]^{P'y'} &  &  & \ar@{}[rd]|-{=} &  &  & \ar@{}[rd]|-{\stackrel{\psi_{y}P}{\Longleftarrow}\quad\;} & \; & \ar@{}[l]|-{\stackrel{\left(y'\right)_{\alpha}^{-1}}{\Longleftarrow}} & \ar@{}[]|-{\stackrel{\theta'}{\Longleftarrow}}\\
\ar@{}[]|-{\stackrel{\theta}{\Longleftarrow}} & \ar@{}[r]|-{\stackrel{\alpha_{y}}{\Longleftarrow}\quad} & \; & \ar@{}[ld]|-{\;\quad\stackrel{P'\psi_{y}}{\Longleftarrow}} &  &  & \; & PP\ar[rr]_{\alpha P} &  & P'P\ar[rr]_{P'\alpha} &  & P'P'\ar[rr]_{\mu'} &  & P'\\
PP\ar[rr]_{\alpha P} &  & P'P\ar[rr]_{P'\alpha} &  & P'P'\ar[rr]_{\mu'} &  & P'
}
\]
\end{defn}

The reader will notice the following is similar to Lemma \ref{w2isext},
meaning we are justified in omitting most of the details.
\begin{lem}
Given a morphism of KZ pseudomonads as above, the 2-cell $\psi_{y}$
exhibits $\alpha$ as a left extension of $y'$ along $y$.
\end{lem}

\begin{proof}
We first observe that $P'y\dashv\mu'\cdot P'\alpha$ (note that this
right adjoint is $\overline{\alpha}$, similar to $\overline{\lambda}$
in Lemma \ref{w2isext}) with unit $\eta$ given by
\[
\xymatrix@=1em{ & P'P\myar{P'\alpha}{r} & P'P'\ar[rd]^{\mu'}\\
P'\ar[ru]^{P'y}\ar@/_{1.3pc}/[rur]_{P'y'}\ar@/_{1.5pc}/[rrr]_{\textnormal{id}_{P'}} & \ar@{}[u]|-{\quad\Uparrow P'\psi_{y}} & \ar@{}[]|-{\cong} & P'\\
 &  & \;
}
\]
We define the counit $\epsilon$ as the unique 2-cell for which
\[
\xymatrix@=1em{ &  &  &  & \;\ar@{}[drr]|-{\underset{}{\Uparrow\epsilon}} &  &  &  &  &  &  &  & \ar@{}[]|-{\;\;\Uparrow\left(\psi_{y}P\right)^{-1}}\\
P\ar[rr]^{y'P}\ar@/_{0.5pc}/[rrd]_{\alpha} &  & P'P\ar[rr]^{P'\alpha}\ar@/^{2pc}/[rrrrrr]^{\textnormal{id}} &  & P'P'\ar[rr]^{\mu'} &  & P'\ar[rr]^{P'y} &  & P'P\ar@{}[rr]|-{=} &  & P\ar@/^{0.7pc}/[rr]^{yP}\ar@/_{0.7pc}/[rr]_{Py}\ar@{}[rr]|-{\Uparrow\theta}\ar@/^{3pc}/[rrrr]^{y'P}\ar@/_{0.7pc}/[rrd]_{\alpha} &  & PP\ar[rr]^{\alpha P} & \; & P'P\\
 &  & P'\ar@/_{0.5pc}/[rru]^{y'P}\ar@{}[u]|-{\Uparrow\left(y'_{\alpha}\right)^{-1}}\ar@/_{0.7pc}/[rrrru]_{\textnormal{id}_{P'}} &  & \ar@{}[u]|-{\cong} &  &  &  &  &  &  &  & P'\ar@/_{0.8pc}/[rru]_{P'y}\ar@{}[ru]|-{\Uparrow\alpha_{y}}
}
\]
We will omit the triangle identities (as this is almost the same calculation
as earlier). The result then follows from \cite[Remark 16]{yonedakz}
and naturality and pseudomonad coherence axioms.
\end{proof}
\begin{rem}
Given a morphism of KZ pseudomonads, we automatically have an invertible
modification
\[
\xymatrix@=1em{PP\myar{\alpha\ast\alpha}{rr}\ar[dd]_{\mu} &  & P'P'\ar[dd]^{\mu'}\\
 & \ar@{}[]|-{\cong}\\
P\ar[rr]_{\alpha} &  & P'
}
\]
so that multiplication is respected. Indeed $\alpha\cdot\mu$ may
be seen as a left extension of $y'$ along $Py\cdot y$ exhibited
by the bijections\def\fCenter{\ \rightarrow\ } 
\def\ScoreOverhang{2pt}
\settowidth{\rhs}{$H\cdot Py\cdot y$} 
\settowidth{\lhs}{$\alpha\cdot\mu$}
\begin{prooftree}
\Axiom$\makebox[\lhs][r]{$\alpha\cdot\mu$} \fCenter \makebox[\rhs][l]{$H$}$
\RightLabel{$\qquad \textnormal{mates correspondence}$} 
\UnaryInf$\makebox[\lhs][r]{$\alpha$} \fCenter \makebox[\rhs][l]{$H\cdot Py$}$
\RightLabel{$\qquad \text{since \ensuremath{\alpha} is a left extension}$}
\UnaryInf$\makebox[\lhs][r]{$y'$} \fCenter \makebox[\rhs][l]{$H\cdot Py\cdot y$}$ 
\end{prooftree}and $\mu'\cdot\alpha\ast\alpha$ may be seen as left extension of
$y'$ along $yP\cdot y$ by recalling that $R_{L}=\textnormal{res}_{L}\cdot y_{\mathcal{B}}$
for admissible $L\colon\mathcal{A}\to\mathcal{B}$ \cite[Remark 16]{yonedakz}
and taking $L$ to be an arbitrary component of $yP\cdot y$ with
respect to $P'$-admissibility. In particular, noting that $P'y\dashv\mu'\cdot P'\alpha$
and $P'yP\dashv\mu'\cdot P'\alpha P$ gives us the necessary data
for constructing $R_{L}$. Finally, noting that $yP\cdot y\cong Py\cdot y$
gives the result. 
\end{rem}

\begin{defn}
Given KZ doctrines $\left(P,y\right)$ and $\left(P',y'\right)$ on
a 2-category $\mathscr{C}$ a \emph{morphism of KZ doctrines} $P\Longrightarrow P'$
consists of the assertions that: 

\begin{enumerate}
\item every $P$-admissible map is also $P'$-admissible;
\item for each $\mathcal{A}\in\mathscr{C}$, the resulting 2-cell exhibiting
the left extension $\alpha_{\mathcal{A}}$ 
\[
\xymatrix@=1em{P\mathcal{A}\ar[rr]^{\alpha_{\mathcal{A}}} &  & P'\mathcal{A}\\
 & \ar@{}[ur]|-{\overset{\psi_{y}^{\mathcal{A}}}{\Longleftarrow}}\\
 &  & \mathcal{A}\ar[uu]_{y'_{\mathcal{A}}}\ar[lluu]^{y_{\mathcal{A}}}
}
\]
is invertible;
\item for each $\mathcal{A},\mathcal{B}\in\mathscr{C}$, left extensions
along $y_{\mathcal{A}}$ into $P\mathcal{B}$ are preserved by $\alpha_{\mathcal{B}}$.\footnote{Consequently, components of $\alpha$ are $P$-homomorphisms.}
\end{enumerate}
\end{defn}

\begin{lem}
Suppose we are given two KZ doctrines $\left(P,y\right)$ and $\left(P',y'\right)$
on a 2-category $\mathscr{C}$, with corresponding KZ pseudomonads
$\left(P,y,\mu\right)$ and $\left(P',y',\mu'\right)$. Then morphisms
$P\Longrightarrow P'$ of KZ doctrines are in bijection with morphisms
$P\Longrightarrow P'$ of KZ pseudomonads (identified via uniqueness
of left extensions up to coherent isomorphism).
\end{lem}

\begin{proof}
Given that every $P$-admissible map is also $P'$-admissible, we
know that $P'y$ has a right adjoint (and that we have a left extension
$\alpha$ as above, assumed invertible). In particular, this right
adjoint may be constructed as in \cite[Prop. 15]{yonedakz}, and thus
we have an adjunction $P'y\dashv\mu'\cdot P'\alpha$ with unit and
counit as above. The triangle identities then force the coherence
condition. Pseudonaturality of $\alpha$ is equivalent to the preservation
condition.

Conversely, given a morphism of KZ pseudomonads (which always gives
rise to a usual morphism of pseudomonads) we know that every $P'$-cocomplete
object is also $P$-cocomplete (as the cocomplete objects may be characterized
as algebras), and similarly for homomorphisms. Hence given a $P$-admissible
map $L\colon\mathcal{A}\to\mathcal{B}$ and map $K\colon\mathcal{A}\to\mathcal{X}$
for a $P'$-cocomplete (and thus also $P$-cocomplete) object $\mathcal{X}$,
there exists a left extension $J\colon\mathcal{B}\to\mathcal{X}$
which is preserved by any $P'$-homomorphism (as such is necessarily
a $P$-homomorphism also). Consequently, $L$ must be $P'$-admissible.
\end{proof}
Combining this with the results of \cite{marm2012}, yields the following
proposition.
\begin{prop}
Given a 2-category $\mathscr{C}$, the assignation of \cite[Theorems 4.1,4.2]{marm2012}
underlies a biequivalence
\[
\mathbf{KZdoc}\left(\mathscr{C}\right)\simeq\mathbf{KZps}\left(\mathscr{C}\right)
\]
where $\mathbf{KZps}\left(\mathscr{C}\right)$ is the 2-category of
KZ pseudomonads, morphisms of KZ pseudomonads and isomorphisms of
left extensions, and $\mathbf{KZdoc}\left(\mathscr{C}\right)$ is
the preorder of KZ doctrines and morphisms of KZ doctrines.
\end{prop}

\subsection{Examples}

Consider the 2-monad $T$ on locally small categories for strict monoidal
categories, and take $P$ to be the free small cocompletion KZ doctrine
on locally small categories. Note that the pseudo-$T$-algebras are
unbiased monoidal categories (equivalent to (strict) monoidal categories
\cite{leinster}) and so we may write $\textnormal{ps-}T\textnormal{-alg}\simeq\textnormal{MonCat}_{\textnormal{ps}}$
with the latter being the 2-category of monoidal categories, strong
monoidal functors and monoidal transformations. 

Given a monoidal category $\left(\mathcal{A},\varotimes\right)$ we
may define a monoidal structure on $P\mathcal{A}$ by Day's convolution
formula
\[
F\varotimes_{\textnormal{Day}}G:=\int^{a,b\in\mathcal{A}}\mathcal{A}\left(-,a\varotimes b\right)\times Fa\times Gb
\]
for small presheaves $F$ and $G$ on $\mathcal{A}$. Note that $F\varotimes_{\textnormal{Day}}G$
is then small, see \cite[Section 7]{daylack2007}. This can be shown
to give a monoidal structure by the arguments of Day \cite{dayconvolution},
equivalent to the structure of a pseudo-$T$-algebra. As the convolution
algebra structure map is separately cocontinuous (and hence $T_{P}$-cocontinuous
\cite[Prop. 2.3.2]{markextension}) we have enough of Proposition
\ref{claim} to show condition (a) of Theorem \ref{liftkzequiv} is
met.

We thus know that $T$ preserves $P$-admissible maps. This says that
if we suppose that $L\colon\mathcal{A}\to\mathcal{B}$ is $P$-admissible,
meaning that each $\mathcal{B}\left(L-,b\right)$ is a small colimit
of representables, then each 
\[
T\mathcal{B}\left(TL\boldsymbol{-},\mathbf{b}\right)=T\mathcal{B}\left[\left(L-,\cdots L-\right),\left(b_{1},\cdots,b_{n}\right)\right]=\prod_{j=1}^{n}\mathcal{B}\left(L-,b_{j}\right)
\]
 is also a small colimit of representables. 

For simplicity, we will consider the preservation of the admissibility
of $L=y_{\mathcal{A}}$ (which is equivalent to preservation for all
$L$). The existence of a pseudo-distributive law of $T$ over $P$
then yields the following example.
\begin{prop}
Let $X,Y\colon\mathcal{A}^{\textnormal{op}}\to\mathbf{Set}$ be two
small presheaves on $\mathcal{A}$. Then 
\[
X\times Y\colon\left(\mathcal{A}\times\mathcal{A}\right)^{\textnormal{op}}\to\mathbf{Set},\qquad\left(a_{1},a_{2}\right)\mapsto X\left(a_{1}\right)\times Y\left(a_{2}\right)
\]
is a small presheaf on $\mathcal{A}\times\mathcal{A}$.
\end{prop}

\begin{proof}
Note that $Ty_{\mathcal{A}}$ is $P$-admissible, and hence 
\[
TP\mathcal{A}\left(Ty_{\mathcal{A}}\boldsymbol{-},\mathbf{X}\right)\colon\left(T\mathcal{A}\right)^{\textnormal{op}}\to\mathbf{Set}
\]
 is a small presheaf on $T\mathcal{A}$ for each $\mbox{\ensuremath{\mathbf{X}}}=\left(X_{1},\cdots,X_{n}\right)$
in $TP\mathcal{A}$. In particular, if we take $\mbox{\ensuremath{\mathbf{X}}}=\left(X,Y\right)$
then
\[
\begin{aligned}TP\mathcal{A}\left(y_{\mathcal{A}}\boldsymbol{-},\mathbf{X}\right) & =\begin{cases}
TP\mathcal{A}\left[\left(y_{\mathcal{A}}-,y_{\mathcal{A}}-\right),\left(X,Y\right)\right], & \mathbf{a}\in\left(\mathcal{A}\times\mathcal{A}\right)^{\textnormal{op}}\\
\emptyset, & \textnormal{otherwise}
\end{cases}\\
 & =\begin{cases}
X\left(-\right)\times Y\left(-\right), & \mathbf{a}\in\left(\mathcal{A}\times\mathcal{A}\right)^{\textnormal{op}}\\
\emptyset, & \textnormal{otherwise}
\end{cases}
\end{aligned}
\]
is a small presheaf on $\sum_{n\in\mathbb{N}}\mathcal{A}^{n}$ and
so $X\left(-\right)\times Y\left(-\right)$ is a small presheaf on
$\mathcal{A}\times\mathcal{A}$.
\end{proof}
Our results also apply to the less general setting of distributing
(co)KZ doctrines over KZ doctrines. The following is such an example.
\begin{example}
Consider the KZ doctrine for the free coproduct completion 
\[
\mathbf{Fam}_{\Sigma}\colon\mathbf{Cat}\to\mathbf{Cat}.
\]
 Here a map $L\colon\mathcal{A}\to\mathcal{B}$ is $\mathbf{Fam}_{\Sigma}$-admissible
when $\mathbf{Fam}_{\Sigma}L$ is a left adjoint; that is, when $L$
is a left multiadjoint. As noted by Diers \cite{diers}, this is to
say that for any $Z\in\mathcal{B}$ there exists a family of morphisms
$\left(h_{i}\colon LX_{i}\to Z\right)_{i\in\mathcal{I}}$ which is
universal in the sense that given any $k\colon LX\to Z$ there exists
a unique pair $\left(i,f\right)$ with $i\in\mathcal{I}$ and $f\colon X\to X_{i}$
such that $h_{i}\cdot Lf=k$.

It is well known the free product completion $\mathbf{Fam}_{\Pi}$
distributes over this doctrine \cite[Section 8]{marm2012}. Thus,
as a consequence of Theorem \ref{liftkzequiv}, we see that if a functor
$L$ is a left multiadjoint, then the functor $\mathbf{Fam}_{\Pi}L$
is a left multiadjoint also.
\end{example}

The following is a simple consequence of the essential uniqueness
of distributive laws over KZ doctrines, shown in Corollary \ref{distunique}.
\begin{example}
Let $\mathbf{Prof}$ be the bicategory of profunctors on small categories,
and let $\mathbf{PROF}$ be the Kleisli bicategory of the free small
cocompletion KZ doctrine $P$ on locally small categories. Clearly
$\mathbf{Prof}$ lies inside $\mathbf{PROF}$. By Corollary \ref{distunique},
the extension of a pseudomonad $T$ on locally small categories to
the bicategory $\mathbf{PROF}$ is essentially unique.
\end{example}

\section{Liftings of Locally Fully Faithful KZ Monads\label{liftlffkz}}

In this section, we consider the case in which the KZ doctrine $P$
being lifted is locally fully faithful. The reader will recall that
a KZ doctrine $P$ is locally fully faithful precisely when each unit
map $y_{\mathcal{A}}$ is fully faithful \cite{bungefunk}. 

The main goal of this section is to deduce an analogue of ``Doctrinal
Adjunction'' on the ``Yoneda structure'' induced by the locally
fully faithful KZ doctrine $P$. We start however with the following
basic properties concerning fully faithful and $P$-fully faithful
maps.
\begin{prop}
\noindent Suppose any of the equivalent conditions of Theorem \ref{liftkzequiv}
are satisfied. Then

(a) if $y_{\mathcal{A}}$ is fully faithful for every $\mathcal{A}\in\mathscr{C}$,
then every $Ty_{\mathcal{A}}$ is fully faithful;

(b) $T$ preserves maps which are both $P$-admissible and $P$-fully
faithful.
\end{prop}

\begin{proof}
Firstly, note that if each $y_{\mathcal{A}}$ is fully faithful (so
that $y_{T\mathcal{A}}$ is fully faithful) then so is $Ty_{\mathcal{A}}$,
since we have an isomorphism
\[
\xymatrix@=1em{TP\ar[rr]^{\lambda_{\mathcal{A}}} &  & PT\\
 & \;\ar@{}[ru]|-{\overset{\omega_{2}}{\Longleftarrow}}\\
 &  & T\mathcal{A}\ar[uu]_{y_{T\mathcal{A}}}\ar[lluu]^{Ty_{\mathcal{A}}}
}
\]
Secondly, note that if $L$ is a $P$-admissible $P$-fully faithful
map, meaning the unit $\eta$ of the admissibility adjunction is invertible,
then so is the unit $n$ exhibiting the admissibility of $TL$ by
Figure \ref{nformula}.
\end{proof}

\subsection{Doctrinal Partial Adjunctions\label{doctrinalpartialadjunctions}}

In this subsection we study how pseudomonads interact with absolute
left liftings (also called partial adjunctions or relative adjunctions),
which we now define. In particular, we show that we get an induced
oplax structure on a partial left adjoint under suitable conditions,
which gives a lifting of the partial adjunction to the setting of
pseudo algebras in a suitable sense. 

This is in the same spirit as subsection \ref{doctrinalleftextensions}
on algebraic left extensions, but not completely analogous (and therefore
not a dual). In particular, here we do not require any algebraic cocompleteness
conditions.
\begin{defn}
Suppose we are given a diagram of the form
\begin{equation}
\xymatrix@=1em{\mathcal{B}\ar[rr]^{R} &  & \mathcal{C}\ar@{}[ld]|-{\stackrel{\eta}{\Longleftarrow}}\\
 & \;\\
 &  & \mathcal{A}\ar[uu]_{I}\ar[uull]^{L}
}
\label{paradjeta}
\end{equation}
in a 2-category $\mathscr{C}$ equipped with a 2-cell $\eta\colon I\to R\cdot L$.
We call such a diagram a \emph{partial adjunction} and say that $L$
is a \emph{partial left adjoint} to $R$ if given any 1-cells $M$
and $N$ as below, for any 2-cell $\zeta\colon I\cdot M\to R\cdot N$
there exists a unique $\overline{\zeta}\colon L\cdot M\to N$ such
that $\zeta$ is equal to the pasting
\[
\xymatrix@=1em{\mathcal{B}\ar[rr]^{R} &  & \mathcal{C}\ar@{}[ld]|-{\stackrel{\eta}{\Longleftarrow}}\\
 & \;\ar@{}[ld]|-{\stackrel{\overline{\zeta}}{\Longleftarrow}\;}\\
\mathcal{D}\ar[rr]_{M}\ar[uu]^{N} &  & \mathcal{A}\ar[uu]_{I}\ar[uull]|-{L}
}
\]
That is, pasting 2-cells of the form $\overline{\zeta}$ above with
$\eta$ defines a bijection of 2-cells.
\end{defn}

\begin{rem}
It is an easy and well known exercise to check that we have an adjunction
$L\dashv R\colon\mathcal{B}\to\mathcal{A}$ with unit $\eta$ in a
2-category $\mathscr{C}$ if and only if 
\[
\xymatrix@=1em{\mathcal{B}\ar[rr]^{R} &  & \mathcal{A}\ar@{}[ld]|-{\stackrel{\eta}{\Longleftarrow}}\\
 & \;\\
 &  & \mathcal{A}\ar[uu]_{\textnormal{id}_{\mathcal{A}}}\ar[uull]^{L}
}
\]
exhibits $L$ as a partial left adjoint. 
\end{rem}

We now define a notion of partial adjunction in the context of pseudo
$T$-algebras and $T$-morphisms. 
\begin{defn}
Suppose we are given oplax $T$-morphisms $\left(I,\xi\right)$ and
$\left(L,\alpha\right)$ and a lax $T$-morphism $\left(R,\beta\right)$
equipped with a $T$-transformation $\eta$ (as in Remark \ref{Ttransremark}
with appropriate identities) as in the diagram
\[
\xymatrix@=1em{\left(\mathcal{B},T\mathcal{B}\overset{y}{\rightarrow}\mathcal{B}\right)\ar[rr]^{\left(R,\beta\right)} &  & \left(\mathcal{C},T\mathcal{C}\overset{z}{\rightarrow}\mathcal{C}\right)\ar@{}[ld]|-{\stackrel{\eta}{\Longleftarrow}}\\
 & \;\\
 &  & \left(\mathcal{A},T\mathcal{A}\overset{x}{\rightarrow}\mathcal{A}\right)\ar[uu]_{\left(I,\xi\right)}\ar[uull]^{\left(L,\alpha\right)}
}
\]
We call such a diagram a \emph{$T$-partial adjunction }if for any
given pseudo $T$-algebra $\left(\mathcal{D},w\right)$, lax $T$-morphism
$\left(M,\epsilon\right)$, and oplax $T$-morphism $\left(N,\varphi\right)$
as below
\[
\xymatrix@=1em{\left(\mathcal{B},T\mathcal{B}\overset{y}{\rightarrow}\mathcal{B}\right)\ar[rr]^{\left(R,\beta\right)} &  & \left(\mathcal{C},T\mathcal{C}\overset{z}{\rightarrow}\mathcal{C}\right)\ar@{}[ld]|-{\;\;\stackrel{\eta}{\Longleftarrow}}\\
 & \;\ar@{}[ld]|-{\stackrel{\overline{\zeta}}{\Longleftarrow}\quad}\\
\left(\mathcal{D},T\mathcal{D}\overset{w}{\rightarrow}\mathcal{D}\right)\ar[rr]_{\left(M,\epsilon\right)}\ar[uu]^{\left(N,\varphi\right)} &  & \left(\mathcal{A},T\mathcal{A}\overset{x}{\rightarrow}\mathcal{A}\right)\ar[uu]_{\left(I,\xi\right)}\ar[ulul]|-{\left(L,\alpha\right)}
}
\]
pasting $T$-transformations of the form $\overline{\zeta}$ above
with the $T$-transformation $\eta$ defines the bijection of $T$-transformations:
\[
\xymatrix@=1em{\left(\mathcal{B},y\right)\ar@{=}[rr] &  & \left(\mathcal{B},y\right)\ar@{}[ldld]|-{\stackrel{\overline{\zeta}}{\Longleftarrow}} &  &  & \left(\mathcal{B},y\right)\ar[rr]^{\left(R,\beta\right)} &  & \left(\mathcal{C},z\right)\ar@{}[ldld]|-{\stackrel{\zeta}{\Longleftarrow}}\\
 &  &  & \ar@{}[r]|-{\sim} & \;\\
\left(\mathcal{D},w\right)\ar[rr]_{\left(M,\epsilon\right)}\ar[uu]^{\left(N,\varphi\right)} &  & \left(\mathcal{A},x\right)\ar[uu]_{\left(L,\alpha\right)} &  &  & \left(\mathcal{D},w\right)\ar[rr]_{\left(M,\epsilon\right)}\ar[uu]^{\left(N,\varphi\right)} &  & \left(\mathcal{A},x\right)\ar[uu]_{\left(I,\xi\right)}
}
\]

This operation of pasting the $T$-transformation $\overline{\zeta}$
with $\eta$ is given by pasting the underlying 2-cells. The verification
that such a pasting of $T$-transformations yields a $T$-transformation
is a simple exercise.
\end{defn}

\begin{rem}
We may be more general here by replacing $\left(M,\epsilon\right)$
and $\left(N,\varphi\right)$ by a lax followed by an oplax, and an
oplax followed by a lax $T$-morphism respectively. However, this
level of generality will not be necessary for this paper.
\end{rem}

We now give the doctrinal properties enjoyed by partial adjunctions.
\begin{prop}
\label{docparadj} Suppose we are given a partial adjunction
\[
\xymatrix@=1em{\mathcal{B}\ar[rr]^{R} &  & \mathcal{C}\ar@{}[ld]|-{\stackrel{\eta}{\Longleftarrow}}\\
 & \;\\
 &  & \mathcal{A}\ar[uu]_{I}\ar[uull]^{L}
}
\]
in a 2-category $\mathscr{C}$ equipped with a pseudomonad $\left(T,u,m\right)$.
Suppose further that 
\[
\left(\mathcal{A},T\mathcal{A}\stackrel{x}{\longrightarrow}\mathcal{A}\right),\quad\left(\mathcal{B},T\mathcal{B}\stackrel{y}{\longrightarrow}\mathcal{B}\right),\quad\left(\mathcal{C},T\mathcal{C}\stackrel{z}{\longrightarrow}\mathcal{C}\right)
\]
are pseudo $T\text{-algebras}$. Then given an oplax $T\text{-morphism}$
structure $\xi$ on $I$ and a lax $T\text{-morphism}$ structure
$\beta$ on $R$, there exists a unique oplax $T\text{-morphism}$
structure $\alpha$ on $L$ such that $\eta$ is a $T$-transformation.
Moreover, this partial adjunction is then lifted to the $T$-partial
adjunction
\[
\xymatrix@=1em{\left(\mathcal{B},T\mathcal{B}\overset{y}{\rightarrow}\mathcal{B}\right)\ar[rr]^{\left(R,\beta\right)} &  & \left(\mathcal{C},T\mathcal{C}\overset{z}{\rightarrow}\mathcal{C}\right)\ar@{}[ld]|-{\quad\stackrel{\eta}{\Longleftarrow}}\\
 & \;\\
 &  & \left(\mathcal{A},T\mathcal{A}\overset{x}{\rightarrow}\mathcal{A}\right)\ar[uu]_{\left(I,\xi\right)}\ar[uull]^{\left(L,\alpha\right)}
}
\]
\end{prop}

\begin{proof}
Given our 2-cells 
\[
\xymatrix@=1em{T\mathcal{A}\ar[dd]_{TI}\ar[rr]^{x}\ar@{}[rrdd]|-{\Downarrow\xi} &  & \mathcal{A}\ar[dd]^{I} &  &  & T\mathcal{B}\ar[dd]_{TR}\ar[rr]^{x}\ar@{}[rrdd]|-{\Uparrow\beta} &  & \mathcal{B}\ar[dd]^{R}\\
\\
T\mathcal{C}\ar[rr]_{z} &  & \mathcal{C} &  &  & T\mathcal{C}\ar[rr]_{z} &  & \mathcal{C}
}
\]
exhibiting $I$ as an oplax $T\text{-morphism}$ and $R$ as a lax
$T\text{-morphism}$, we can take our oplax constraint cell of $L$
(which we call $\alpha$) as the unique solution to
\[
\xymatrix{T\mathcal{B}\ar[r]^{y}\;\ar@{}[rdd]|-{\Uparrow\alpha} & \mathcal{B}\ar[rd]^{R} &  &  & T\mathcal{B}\ar[r]^{y}\ar[dr]|-{TR} & \mathcal{B}\ar[rd]^{R}\ar@{}[d]|-{\Uparrow\beta}\\
 & \;\ar@{}[r]|-{\Uparrow\eta} & \mathcal{C} & = & \;\ar@{}[r]|-{\Uparrow T\eta} & T\mathcal{C}\ar[r]^{z}\ar@{}[d]|-{\Uparrow\xi} & \mathcal{C}\\
T\mathcal{A}\ar[r]_{x}\ar[uu]^{TL} & \mathcal{A}\ar[uu]_{L}\ar[ur]_{I} &  &  & T\mathcal{A}\ar[r]_{x}\ar[uu]^{TL}\ar[ur]|-{TI} & \mathcal{A}\ar[ur]_{I}
}
\]
which exists since $L$ is a partial left adjoint. That is, $\alpha$
is the unique oplax structure on $L$ for which $\eta:I\to R\cdot L$
is a $T$-transformation. The verification that $\alpha$ then satisfies
the unitary and multiplicative coherence axioms is a simple exercise
which we omit.
\end{proof}
The following example is an easy application of this result which
does not involve Yoneda structures.
\begin{prop}
\label{compositeoplax} Suppose $\mathscr{A},\mathscr{B}$ and $\mathscr{C}$
are bicategories. Consider a diagram 
\[
\xymatrix{\mathscr{A}\ar@{-->}[r]^{F}\ar@/_{1pc}/[rr]_{H} & \mathscr{B}\ar[r]^{G} & \mathscr{C}}
\]
where $G$ is a lax and locally fully faithful functor, $H$ is an
oplax functor, and $F$ is a locally defined functor
\[
\left(F_{X,Y}:\mathscr{A}\left(X,Y\right)\to\mathscr{B}\left(FX,FY\right):X,Y\in\mathscr{A}\right)
\]
where $G\cdot F=H$ locally. It then follows that $F$ extends to
an oplax functor.
\end{prop}

\begin{proof}
To see this, recall that the fully faithfulness of each $G_{M,N}$
(for objects $M,N\in\mathscr{B}$) may be characterized by saying
that each
\[
\xymatrix@=1em{\mathscr{B}\left(M,N\right)\myar{G_{M,N}}{rr} &  & \mathscr{C}\left(HM,HN\right)\ar@{}[dl]|-{\;\;\stackrel{\text{id}}{\Longleftarrow}}\\
 & \mbox{\;}\\
 &  & \mathscr{B}\left(M,N\right)\ar@{->}[ulul]^{\textnormal{id}_{\mathscr{B}\left(M,N\right)}}\ar[uu]_{G_{M,N}}
}
\]
is an absolute lifting \cite[Example 2.18]{weberyoneda}. As this
absolute left lifting is preserved upon whiskering by 
\[
F_{X,Y}:\mathscr{A}\left(X,Y\right)\to\mathscr{B}\left(FX,FY\right)
\]
we have the family of partial adjunctions
\[
\xymatrix@=1em{\mathscr{B}\left(FX,FY\right)\ar@{->}[rr]^{G_{FX,FY}} &  & \mathscr{C}\left(HX,HY\right)\ar@{}[dl]|-{\stackrel{\text{id}}{\Longleftarrow}}\\
 & \mbox{\;}\\
 &  & \mathscr{A}\left(X,Y\right)\ar@{->}[ulul]^{F_{X,Y}}\ar[uu]_{H_{X,Y}}
}
\]
Endowing with the bicategory structure of $\mathscr{A}$, and full
sub-bicategory structures of $\mathscr{B}$ and $\mathscr{C}$ restricted
to objects in the images of $F$ and $H$ respectively, we see by
Proposition \ref{docparadj} that $F$ extends to an oplax functor
$F\colon\mathscr{A}\to\mathscr{B}$.
\end{proof}
\begin{rem}
Clearly, this may be stated more generally in the setting of a pseudo
$T$-algebras. Also, it suffices to only have an isomorphism $GF\cong H$
on the underlying 2-category.
\end{rem}

\begin{rem}
In Kelly's setting of a doctrinal adjunction \cite{doctrinal}, if
both the left and right adjoint are lax, exhibited by a counit and
unit which are $T$-transformations of lax $T$-morphisms, then the
induced oplax structure on the left adjoint is inverse to the given
lax structure. In this partial adjunction case, the best we can say
is that if $\left(I,\xi\right)$ is pseudo, $\left(L,\alpha^{\ast}\right)$
lax, and $\eta\colon\left(I,\xi\right)\to\left(R,\beta\right)\cdot\left(L,\alpha^{\ast}\right)$
a $T$-transformation of lax $T$-morphisms, then the induced oplax
structure on $L$ given as $\alpha$ satisfies $\alpha^{\ast}\cdot\alpha=\textnormal{id}_{L\cdot x}$.
This means the identity 2-cell is a generalized $T$-transformation
from $\left(L,\alpha\right)$ to $\left(L,\alpha^{*}\right)$, but
not necessarily the other way around.
\end{rem}

\subsection{Doctrinal ``Yoneda Structures''\label{doctrinalyonedastructures}}

Kelly \cite{doctrinal} showed that given an adjunction $L\dashv R$
which lifts to pseudo algebras, oplax structures on the left adjoint
are in bijection with lax structures on the right adjoint. The goal
of this section is to give a similar result for ``Yoneda structure
diagrams'', that is diagrams of the form
\[
\xymatrix@=1em{\mathcal{B}\myar{R}{rr} &  & P\mathcal{A}\ar@{}[ld]|-{\stackrel{\varphi_{L}}{\Longleftarrow}}\\
 & \;\\
 &  & \mathcal{A}\ar[uu]_{y_{\mathcal{A}}}\ar[uull]^{L}
}
\]
for which $L$ is an absolute left lifting, and $R$ is a left extension
exhibited by the same 2-cell $\varphi_{L}$ (as appear in Yoneda structures
\cite{yonedastructures}, or in the setting of a locally fully faithful
KZ doctrine \cite{yonedakz}). 

We state the following as one of the main results of this paper, due
to its applications as a coherence result for oplax functors out of
certain bicategories, such as the bicategories of spans or polynomials.
This application will be briefly discussed at the end of this section.
\begin{thm}
[Doctrinal Yoneda Structures] Suppose we are given a 2-category $\mathscr{C}$
equipped with a pseudomonad $\left(T,u,m\right)$ and a locally fully
faithful KZ doctrine $\left(P,y\right)$. Suppose that $T$ pseudo-distributes
over $P$. Suppose we are given pseudo $T$-algebra structures
\[
\left(\mathcal{A},T\mathcal{A}\stackrel{x}{\longrightarrow}\mathcal{A}\right),\qquad\qquad\left(\mathcal{B},T\mathcal{B}\stackrel{y}{\longrightarrow}\mathcal{B}\right)
\]
Then for any $P$-admissible map $L\colon\mathcal{A}\to\mathcal{B}$
we have a Yoneda structure diagram as on the left, underlying a ``doctrinal
Yoneda structure'' diagram as on the right 
\[
\xymatrix@=1em{\mathcal{B}\myar{R_{L}}{rr} &  & P\mathcal{A}\ar@{}[ld]|-{\stackrel{\varphi_{L}}{\Longleftarrow}} &  &  & \left(\mathcal{B},y\right)\myar{\left(R_{L},\beta\right)}{rr} &  & \left(P\mathcal{A},z_{x}\right)\ar@{}[ld]|-{\stackrel{\varphi_{L}}{\Longleftarrow}}\\
 & \; &  &  &  &  & \;\\
 &  & \mathcal{A}\ar[uu]_{y_{\mathcal{A}}}\ar[uull]^{L} &  &  &  &  & \left(\mathcal{A},x\right)\ar[uu]_{\left(y_{\mathcal{A}},\xi\right)}\ar[uull]^{\left(L,\alpha\right)}
}
\]
in that 2-cells $\alpha$ as on the left below exhibiting $L$ as
an oplax $T$-morphism
\[
\xymatrix@=1em{T\mathcal{B}\ar[rr]^{y} &  & \mathcal{B} &  &  &  &  & T\mathcal{B}\ar[rr]^{y}\ar[dd]_{TR_{L}} &  & \mathcal{B}\ar[dd]^{R_{L}}\\
 & \ar@{}[]|-{\Uparrow\alpha} &  &  &  &  &  &  & \ar@{}[]|-{\Uparrow\beta}\\
T\mathcal{A}\ar[rr]_{x}\ar[uu]^{TL} &  & \mathcal{A}\ar[uu]_{L} &  &  &  &  & T\mathcal{C}\ar[rr]_{z_{x}} &  & \mathcal{C}
}
\]
are in bijection with 2-cells $\beta$ as on the right exhibiting
$R_{L}$ as a lax $T$-morphism.
\end{thm}

\begin{proof}
We need only check that the propositions concerning partial adjunctions
and left extensions\footnote{Note that Proposition \ref{docleftext} applies since each $z_x$ is $T_P$-cocontinuous by Proposition \ref{claim}.}
are inverse to each other. But this is just a consequence of the fact
that we can go between the defining equalities for these propositions
\[
\xymatrix{T\mathcal{B}\ar[r]^{y}\ar@{}[rdd]|-{\Uparrow\alpha} & \mathcal{B}\ar[rd]^{R_{L}} &  &  & T\mathcal{B}\ar[r]^{y}\ar[dr]|-{TR_{L}} & \mathcal{B}\ar[rd]^{R_{L}}\ar@{}[d]|-{\Uparrow\beta}\\
 & \;\ar@{}[r]|-{\Uparrow\varphi_{L}} & P\mathcal{A} & = & \;\ar@{}[r]|-{\Uparrow T\varphi_{L}} & TP\mathcal{A}\ar[r]^{z_{x}}\ar@{}[d]|-{\Uparrow\xi_{x}} & P\mathcal{A}\\
T\mathcal{A}\ar[r]_{x}\ar[uu]^{TL} & \mathcal{A}\ar[uu]|-{L}\ar[ur]_{y_{\mathcal{A}}} &  &  & T\mathcal{A}\ar[r]_{x}\ar[uu]^{TL}\ar[ur]|-{Ty_{\mathcal{A}}} & \mathcal{A}\ar[ur]_{y_{\mathcal{A}}}
}
\]
and
\[
\xymatrix{ & T\mathcal{B}\ar[r]^{y} & \mathcal{B}\ar[dd]^{R} &  &  & T\mathcal{B}\ar[r]^{y}\ar[dd]|-{TR_{L}} & \mathcal{B}\ar[dd]^{R_{L}}\\
T\mathcal{A}\ar[ru]^{TL}\ar[r]^{x}\ar[rd]_{Ty_{\mathcal{A}}} & \mathcal{A}\ar[dr]^{y_{\mathcal{A}}}\ar[ur]_{L}\ar@{}[r]|-{\Uparrow\varphi_{L}}\ar@{}[u]|-{\Uparrow\alpha}\ar@{}[d]|-{\Uparrow\xi_{x}^{-1}} & \; & = & T\mathcal{A}\ar[ru]^{TL}\ar[rd]_{Ty_{\mathcal{A}}}\ar@{}[r]|-{\Uparrow T\varphi_{L}} & \;\ar@{}[r]|-{\Uparrow\beta} & \;\\
 & TP\mathcal{A}\ar[r]_{z_{x}} & P\mathcal{A} &  &  & TP\mathcal{A}\ar[r]_{z_{x}} & P\mathcal{A}
}
\]
by pasting with $\xi_{x}$ and $\xi_{x}^{-1}$.
\end{proof}
\begin{rem}
In the ``doctrinal Yoneda structure'' of the above, $\varphi_{L}$
is a $T$-transformation exhibiting $\left(R_{L},\beta\right)$ as
a $T$-left extension and $\left(L,\alpha\right)$ as a $T$-partial
left adjoint, provided $\alpha$ and $\beta$ correspond via this
bijection.
\end{rem}

We observe that the bijection between oplax structures on left adjoints
and lax structures on right adjoints as in ``Doctrinal adjunction''
\cite{doctrinal} is a special case of this theorem.
\begin{cor}
[Kelly] Suppose we are given a 2-category $\mathscr{C}$ equipped
with a pseudomonad $\left(T,u,m\right)$, pseudo $T$-algebra structures
\[
\left(\mathcal{A},T\mathcal{A}\stackrel{x}{\longrightarrow}\mathcal{A}\right),\qquad\qquad\left(\mathcal{B},T\mathcal{B}\stackrel{y}{\longrightarrow}\mathcal{B}\right)
\]
and an adjunction $L\dashv R\colon\mathcal{B}\to\mathcal{A}$ in $\mathscr{C}$.
Then oplax structures on $L$ are in bijection with lax structures
on $R$.
\end{cor}

\begin{proof}
Let $P$ be the identity pseudomonad on $\mathscr{C}$, which is clearly
a locally fully faithful KZ doctrine. Trivially, any pseudomonad $T$
pseudo-distributes over the identity. Now observe that for the identity
pseudomonad, the admissible maps are the left adjoints and the ``Yoneda
structure diagrams'' are the units of adjunctions $\eta\colon\textnormal{id}_{\mathcal{A}}\to R\cdot L$.
Applying the above theorem then gives the result.
\end{proof}

\subsection{Applications and Future Work}

The motivating application of this result is not to give an analogous
result to doctrinal adjunction, but instead the observation that it
may be seen as a coherence result. In particular, consider the following
special case of this theorem concerning the bicategory of spans in
a category $\mathcal{E}$ with pullbacks, denoted $\mathbf{Span}\left(\mathcal{E}\right)$. 

For the following corollary, we recall that locally defined functors
are the morphisms of $\mathbf{CatGrph}$, and $\mathbf{CatGrph}$
gives rise to bicategories and oplax/lax functors via a suitable 2-monad
\cite{icons}.
\begin{cor}
Suppose we are given a small\footnote{Note that one may work in a larger universe to work around this condition.}
category with pullbacks $\mathcal{E}$ and a bicategory $\mathscr{C}$
with the same objects, as well as locally defined functors
\[
L_{X,Y}:\mathbf{Span}\left(\mathcal{E}\right)\left(X,Y\right)\to\mathscr{C}\left(X,Y\right)
\]
with corresponding left extensions $\left(R_{L}\right)_{X,Y}$ as
components in the diagram 
\[
\xymatrix@=1em{\mathscr{C}\myar{R_{L}}{rr} &  & \hat{\mathbf{Span}}\left(\mathcal{E}\right)\ar@{}[ld]|-{\quad\stackrel{\varphi_{L}}{\Longleftarrow}}\\
 & \;\\
 &  & \mathbf{Span}\left(\mathcal{E}\right)\ar[uu]_{Y}\ar@{->}[uull]^{L}
}
\]
where $\hat{\mathbf{Span}}\left(\mathcal{E}\right)$ is the local
cocompletion\footnote{The monoidal cocompletion as given by the Day convolution structure may be generalized to the setting of bicategories; we call this the local cocompletion.}
of $\mathbf{Span}\left(\mathcal{E}\right)$. Then oplax structures
on $L$ are in bijection with lax structures on $R_{L}$.
\end{cor}

To see why this is useful, recall that composition of spans is given
by taking the terminal diagram of the form
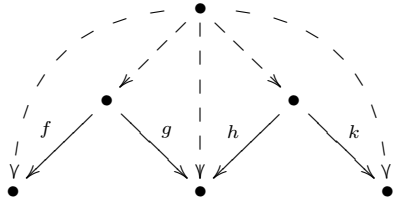
\[
\xymatrix{ &  & \bullet\ar@{-->}[dd]\ar@{-->}[ld]\ar@{-->}[rd]\ar@/_{2pc}/@{-->}[ldld]\ar@/^{2pc}/@{-->}[rdrd]\\
 & \bullet\ar[ld]_{f}\ar[rd]^{g} &  & \bullet\ar[ld]_{h}\ar[rd]^{k}\\
\bullet &  & \bullet &  & \bullet
}
\]
and so when evaluating the composite of two spans we may recover the
two morphisms of spans in the above diagram; that is, there is a relationship
between the way 2-cells are defined and how composition of 1-cells
is defined.

This relationship between composition and 2-cells is captured in Day's
convolution formula \cite{dayconvolution}, and causes the coend defining
the Day convolution to collapse to a more workable sum. In particular,
composition in $\hat{\mathbf{Span}}\left(\mathcal{E}\right)$ is given
by the convolution formula
\[
{\displaystyle GF\left(s;t\right)=\sum_{T\stackrel{h}{\longrightarrow}Y}F\left(s;h\right)G\left(h;t\right)}
\]
where $s;t$ is an arbitrary span from $X$ to $Z$ through $Y$,
and $F$ and $G$ are presheaves on $\mathbf{Span}\left(\mathcal{E}\right)\left(X,Y\right)$
and $\mathbf{Span}\left(\mathcal{E}\right)\left(Y,Z\right)$ respectively.
As a result, it is easier to show that a locally defined functor $L\colon\mathbf{Span}\left(\mathcal{E}\right)\to\mathscr{C}$
is oplax by instead showing that the corresponding $R_{L}\colon\mathscr{C}\to\hat{\mathbf{Span}}\left(\mathcal{E}\right)$
is lax. Indeed, the reader should notice here that the problem of
showing $L$ is oplax involves pullbacks, whereas the equivalent problem
of showing $R$ is lax does not (once this convolution formula has
been established).

A more involved application along the same lines deals not with the
bicategory of spans, but instead $\mathbf{Poly}_{c}\left(\mathcal{E}\right)$,
the bicategory of polynomials with cartesian 2-cells as studied by
Gambino, Kock and Weber \cite{weber,gambinokock}. We see that due
to the complicated nature of composition in $\mathbf{Poly}_{c}\left(\mathcal{E}\right)$,
showing that a locally defined functor $L:\mathbf{Poly}\left(\mathcal{E}\right)\to\mathscr{C}$
is oplax becomes a large calculation (especially for the associativity
coherence conditions); however if we instead show that $R_{L}:\mathscr{C}\to\hat{\mathbf{Poly}_{c}}\left(\mathcal{E}\right)$
is lax our work will be reduced significantly; in fact by this method
we can completely avoid coherences involving composition of distributivity
pullbacks.

In a soon forthcoming paper we will exploit this fact in more detail
to give a complete proof of the universal properties of polynomials
which avoids the majority of the coherence conditions.

\bibliographystyle{siam}
\bibliography{references}

\end{document}